\documentclass[reqno]{amsart}
  \usepackage{mathabx}
\usepackage{cite}
\usepackage{color}
\usepackage{enumerate}

\usepackage{appendix}
\usepackage{geometry}
\usepackage[linktocpage]{hyperref}
\allowdisplaybreaks

\theoremstyle{plain}
\newtheorem{thm}{Theorem}[section]

\newtheorem{lem}[thm]{Lemma}
\newtheorem{prop}[thm]{Proposition}
\newtheorem{rem}{Remark}[section]

\theoremstyle{definition}

\numberwithin{equation}{section}

\newcommand{\p}{{\partial}}

\begin{document}

\title[ convergence theory for MHD boundary layer]
 {Justification of Prandtl Ansatz for MHD boundary layer}
\author[Cheng-Jie Liu]{Cheng-Jie Liu}
\address{School of Mathematical Sciences, and Institute of Natural Sciences, Shanghai Jiao Tong University, Shanghai 200240, P.R. China}
\email{liuchengjie@sjtu.edu.cn}
\author[Feng Xie]{Feng Xie}
\address{School of Mathematical Sciences, and LSC-MOE,
 Shanghai Jiao Tong University,
Shanghai 200240, P.R.China}
\email{tzxief@sjtu.edu.cn}
\author[Tong Yang]{Tong Yang}
\address{Department of Mathematics,
City University of Hong Kong,
Tat Chee Avenue, Kowloon, Hong Kong}
\email{matyang@cityu.edu.hk}

\begin{abstract}
As a continuation of \cite{LXY}, the paper aims to justify the high Reynolds
numbers limit for the MHD system with Prandtl boundary layer expansion
 when no-slip boundary condition is
 imposed on velocity field and perfect conducting boundary condition on magnetic field.
  Under the assumption that the viscosity and resistivity coefficients are
  of the same order and the initial tangential
  magnetic field on the boundary is not degenerate, we justify the validity of the Prandtl boundary layer expansion and give a
 $L^\infty$ estimate on the error by multi-scale analysis.
\end{abstract}



\maketitle

\section{Introduction and Main Results} \label{S0}
For electrically conducting fluid such as plasmas
and  liquid metals, the system of magnetohydrodynamics(denoted by MHD) is a fundamental
system to describe the motion of fluid under the influence of
 electro-magnetic field. The study on the MHD was initiated by
Hannes Alfv\'en who showed that magnetic field can induce currents in
a moving conductive fluid with a new propagation
mechanism along the magnetic field, called Alfv\'en waves (see \cite{Alf}).
One important problem about MHD is to understand the inviscid and vanishing resistivity limit  in a domain with boundary.
The purpose of this paper is to justify this high Reynold numbers limit
when the tangential magnetic field is not degenerate on the boundary.

To this end, consider the following two-dimensional (2D) incompressible viscous MHD equations in the domain $\{(t,x,y)|t>0, x\in\mathbb{T}, y\in\mathbb{R}_+\}$,
\begin{align}\label{1.1}
\left\{
\begin{array}{ll}
\p_t \mathbf{u^\epsilon}+(\mathbf{u^\epsilon}\cdot\nabla)\mathbf{u^\epsilon}+\nabla p^\epsilon-(\mathbf{H^\epsilon}\cdot\nabla)\mathbf{H^\epsilon}=\mu\epsilon\triangle \mathbf{u^\epsilon},\\
\p_t \mathbf{H^\epsilon}+(\mathbf{u^\epsilon}\cdot\nabla)\mathbf{H^\epsilon}-(\mathbf{H^\epsilon}\cdot\nabla)\mathbf{u^\epsilon}=\kappa\epsilon\triangle \mathbf{H^\epsilon},\\
\nabla\cdot \mathbf{u^\epsilon}=0,\qquad \nabla\cdot \mathbf{H^\epsilon}=0.
\end{array}
\right.
\end{align}
Here $\mathbf{u^\epsilon}=(u^\epsilon, v^\epsilon)$ and $\mathbf{H^\epsilon}=(h^\epsilon, g^\epsilon)$ stand for the velocity field and magnetic field respectively, $p^\epsilon$ denotes the total pressure, the tangential variable is periodic: $x\in\mathbb{T}$, and the normal variable $y\in\mathbb{R}_+$. Also we assume $\mu,\kappa$ are positive constants, and the viscosity and resistivity coefficients are of
 the same order in a small parameter $\epsilon$. The initial data of (\ref{1.1}) is given by
\begin{align}\label{ID}
(\mathbf{u^\epsilon}, \mathbf{H^\epsilon})|_{t=0}=(\mathbf{u_0}, \mathbf{H_0})(x,y)=(u_0, v_0, h_0, g_0)(x,y)
\end{align}
independent of $\epsilon$. The no-slip boundary condition is imposed on velocity field, and the perfectly conducting boundary condition on magnetic field:
\begin{align}
\label{BCM}
\mathbf{u^\epsilon}|_{y=0}={\bf{0}},\quad (\p_yh^\epsilon, g^\epsilon)|_{y=0}={\bf{0}}.
\end{align}

The initial-boundary value problem \eqref{1.1}-\eqref{BCM} with fixed $\epsilon>0$ has been investigated and its global well-posedness is well known, see \cite{DL, ST} for instance. Let us also mention that there are vast literatures on the MHD system, in particular
in the case without boundaries, cf. \cite{A-Z, Bardos, C-L, C-W, CMRR, He,Lin, L-Z, Xu} and the references therein. 

In this paper, we are concerned with the asymptotic behavior of solutions 
$(\mathbf{u^\epsilon}, \mathbf{H^\epsilon})$ 
to problem \eqref{1.1}-\eqref{BCM} as $\epsilon\rightarrow0.$
Formally, when $\epsilon=0$, (\ref{1.1}) becomes
 the following incompressible ideal MHD system:
 \begin{align}\label{ILE0}
\left\{
\begin{array}{ll}
\p_t \mathbf{u^0}+(\mathbf{u^0}\cdot\nabla)\mathbf{u^0}+\nabla p^0-(\mathbf{H^0}\cdot\nabla)\mathbf{H^0}=\mathbf{0},\\
\p_t \mathbf{H^0}+(\mathbf{u^0}\cdot\nabla)\mathbf{H^0}-(\mathbf{H^0}\cdot\nabla)\mathbf{u^0}=\mathbf{0},\\
\nabla\cdot \mathbf{u^0}=0,\qquad \nabla\cdot \mathbf{H^0}=0,
\end{array}
\right.
\end{align}
where the velocity field $\mathbf{u^0}=(u^0,v^0)$ and magnetic field $\mathbf{H^0}=(h^0,g^0)$. Naturally, we endow (\ref{ILE0}) with homogeneous Dirichlet boundary condition on normal components of velocity and magnetic field:
\begin{align}
\label{IBE}
(v^0, g^0)|_{y=0}={\bf 0}.
\end{align}
Note that such boundary condition \eqref{IBE} is sufficient to solve (\ref{ILE0}), since the boundary $\{y=0\}$ is the streamline for \eqref{ILE0} under \eqref{IBE}.

Comparing the boundary conditions \eqref{BCM} with (\ref{IBE}),  there is a  mismatch between the tangential component $(u^\epsilon, h^\epsilon)(t,x,y)$ and $(u^0, h^0)(t,x,y)$ on the boundary $\{y=0\}$. According to the classical Prandtl boundary layer theory \cite{P}, there
is a thin layer of width of order
$\sqrt{\epsilon}$ near the boundary, in which $(u^\epsilon, h^\epsilon)$
changes dramatically from its boundary data to the outer flow $(u^0_e, h^0_e)(t,x, y)$.  In other words, there exist boundary layer profiles $u_b^0(t,x,\frac{y}{\sqrt{\epsilon}})$ and $h_0^b(t,x,\frac{y}{\sqrt{\epsilon}})$, such that the solution to the 
problem (\ref{1.1})-(\ref{BCM}) is expected to have the form:
\begin{align}\label{AA}
\begin{cases}
	(\mathbf{u^\epsilon}, \mathbf{H^\epsilon})(t,x,y)&=~(\mathbf{u^0}, \mathbf{H^0})(t,x,y)+(u_b^0, 0, h_b^0, 0)(t,x,\frac{y}{\sqrt{\epsilon}})+\small{o}(1),\\
\hspace{.2in} p^\epsilon(t,x,y)&=~ p^0(t,x,y)+\small{o}(1),
\end{cases}
\end{align}
where the error terms $\small{o}(1)$ tends to zero in $L^\infty$-norm as $\epsilon$ tends to zero.
 
We are devoted to verify the Prandtl boundary layer expansion in \eqref{AA} for the MHD system \eqref{1.1}-\eqref{BCM}. 
Let us first review some mathematical results on the classical boundary layer theory. It is well known that in both physics and mathematics, the study on fluid around a rigid body with high Reynolds number is important and challenging, and the fluid motion exhibit rather complicated behaviors, especially near the surface of body. This is partially due to the appearance of boundary layers, whose formation is formally explained by Prandtl \cite{P} in 1904. Prandtl also  derived the simplified equations, the well-known Prandtl equations, from the
incompressible Navier-Stokes equations with no-slip boundary condition on velocity, to describe the fluid motion in the boundary layer.
Under the monotone assumption on the tangential velocity in the normal direction,  Oleinik firstly in 1960s'
obtained the local existence of classical solutions of 2D Prandtl equations
by using the Crocco transformation, cf.  \cite{O}.
 The result
together with some other related works  are well written in the
 classical
book \cite{OS}. Recently, this well-posedness result was
established in the Sobolev spaces by using  energy methods in \cite{AWXY} and \cite{MW1} independently.
 Moreover, by imposing an additional
favorable condition on the pressure,
a global in time weak solution was obtained in
\cite{XZ}. Some of these results were
generalized to 3D case with special structure  in \cite{LWY1} and \cite{LWY3}.
In the absence of monotonicity condition, boundary separation can be observed, and so far we only gain the local-in-time solvability of the Prandtl equations in the analytic framework \cite{LCS,SC1, IV,KV,KMVW,ZZ} or Gevrey framework \cite{GM,LWX,LY,LY1}. To our knowledge, the solvability of Prandtl equations for general initial data in a Sobolev class is still open,
although some interesting ill-posedness (or instability) results for 
Prandtl equations have been established, cf.
 \cite{EE,GD,GN,G, GGN2, GGN3,GN1, LWY2, L-Y, KV1} and the references therein.
On the other hand, the rigorous verification of the Prandtl boundary layer theory, i.e., the
solution to the Navier-Stokes equations as a superposition of
solutions to the Euler and Prandtl systems in vanishing viscosity limit, was
achieved only for some specific cases, e.g., in
the analytic framework in \cite{SC1,SC2, WWZ}. In 2014, the convergence problem in 2D case was studied by Maekawa \cite{M}
that requires that the vorticity of flow vanishes
 in a neighborhood of boundary initially, and such condition implies the analyticity of the initial data with respect to the tangential variable in the same region. 
In 2016, the authors in \cite{GMM} improved the results of Sammartino \& Caflisch \cite{SC1,SC2} in Gevery class. For the steady case, Guo \& Nguyen in \cite{GN2} justified the Prandtl boundary layer expansions for the steady Navier-Stokes flows over a moving plate, and similar results for steady flows were obtained in \cite{Iy1,Iy2,Iy3}. Very recently, G\'erard-Varet \& Maekawa in \cite{G-M} shows the $H^1$ stability of shear flows of Prandtl type and verifies the Prandtl expansions for the steady Navier-Stokes equations with no-slip boundary condition. 
 It is remarkable that the solvability of Prandtl equations does not necessarily imply the validity of corresponding Prandtl boundary layer expansions, see the counterexamples in \cite{G, G-N1, G-N2}.

For plasma, the boundary layer equations can be derived from the fundamental
MHD system and they are more complicated than the classical Prandtl system because
of the coupling of magnetic field with velocity field through the
Maxwell equations. 
It should be emphasized that the MHD boundary layer is an important
problem in study of  plasma with fruitful results, cf. \cite{A,DL,D,G-P, R, WX, XXW}.
On one hand, if the magnetic field is transversal to the boundary,
there are extensive discussions on the so-called Hartmann boundary layer,
cf. \cite{C-P, davidson}. 
On the other hand, if the magnetic field is tangent to the boundary, it is just the case we are concerned with in this paper. 
Note that in physics, it is believed
that the magnetic field has a stabilizing effect on the boundary layer that
could provide a mechanism for containment of some kind of instability and singularity. 
Recently, the same authors of this paper established the well-posedness of MHD boundary layer equations in weighted Sobolev spaces without  monotonicity condition on the velocity in \cite{LXY}. The key assumption is that tangential magnetic field is not degenerate on the boundary. As the continuation
of \cite{LXY}, we  study the high Reynolds numbers limit problem for (\ref{1.1})-\eqref{BCM}. 
Precisely, by applying the
 multi-scale expansion of $(\mathbf{u^\epsilon}, \mathbf{H^\epsilon})$ in Sections 2 and 3, we will
 justify the validity of the Prandtl boundary layer theory in (\ref{AA}) under
the non-degeneracy condition on the initial tangential magnetic field on the boundary.
As it is well known that the Alfv\'en wave propagates along the magnetic field, this result in some sense justifies the physical phenomena that the Alfv\'en wave
along the boundary carries away the energy so that it stabilizes the boundary layer rigorously in mathematics.

We are now ready to state main result in this paper as follows.

\begin{thm}\label{MAIN THM}
Let $m\geq 36$ be an integer. Let the initial data $
(\mathbf{u_0},\mathbf{H_0})(x,y)\in H^m(\mathbb{T}\times\mathbb{R}^+)$ satisfy:
\begin{enumerate}[i)]
	\item $\mathbf{u_0}=(u_0,v_0)$ is a divergence free vector field vanishing on the boundary, and $\mathbf{H_0}=(h_0,g_0)$ is a divergence free vector field tangent to the boundary; 
	\item there exists a small $\epsilon_0>0$ such that the following `strong' compatibility conditions hold for any $\epsilon\in[0,\epsilon_0]$,
	$$\partial_t^i(u^\epsilon,v^\epsilon,g^\epsilon)(0)|_{y=0}=\mathbf{0},~\partial_t^{i-1}\partial_yh^\epsilon(0)|_{y=0}=0,\quad 1\leq i\leq [\frac{m}{2}]-3,$$
	where $[k], k\in\mathbb{R}$ stands for the largest integer less than or equal to $k$, and $\partial_t^i(u^\epsilon,v^\epsilon,h^\epsilon, g^\epsilon)(0)$ is the $i-$th time derivative at $\{t=0\}$ of any solution of \eqref{1.1}-\eqref{BCM}, as calculated from \eqref{1.1} to yield an expression in terms of derivatives of $(\mathbf{u_0},\mathbf{H_0})$;
	\item the initial tangential magnetic filed is non-degenerate:
\begin{align}\label{ass_key}
	h_0(x,0)\geq \delta_0>0\quad\mbox {for some constant}~ \delta_0.
\end{align} 
\end{enumerate}
 Then there exists  $T_*>0$ independent of $\epsilon$ such that, the problem (\ref{1.1})-\eqref{BCM} admits a solution $(\mathbf{u^\epsilon}, \mathbf{H^\epsilon})$ in the time interval $[0,T_*]$, and there exists a smooth solution $(\mathbf{u^0}, \mathbf{H^0})(t,x,y)\in C\big([0,T_*],H^m\big)$ to the problem \eqref{ILE0}-\eqref{IBE} with the initial data $(\mathbf{u_0},\mathbf{H_0})$, and a boundary layer profile
 \[(u_b^0, v_b^0, h_b^0, g_b^0)\in C\big([0,T_*]\times\mathbb{T}\times\mathbb{R}^+\big),\]
  such that for any arbitrarily small $\sigma>0$,
\begin{align}\label{ASP}
\begin{aligned}
	&\sup_{0\leq t\leq T_*}
\big\|(\mathbf{u^{\epsilon}},\mathbf{H^{\epsilon}})(t,x,y)-(\mathbf{u^0},\mathbf{H^0})(t,x,y)-\big(u^0_b, \sqrt{\epsilon}v^0_b, h^0_b, \sqrt{\epsilon}g^0_b\big)\big(t,x,\frac{y}{\sqrt{\epsilon}}\big)\big\|_{L^\infty_{xy}}
\leq~ C\epsilon^{3/8-\sigma},
\end{aligned}
\end{align}
where the constant $C>0$ is independent of $\epsilon$.
\end{thm}\begin{rem}
	The assumptions on the regularity and compatibility conditions of the initial data $(\mathbf{u_0}, \mathbf{H_{0}})$ are not optimal. Here, we mainly utilize such assumptions to verify Proposition \ref{PROP2} in Subsection \ref{S2}. One may relax the requirement on the regularity and compatibility conditions.
\end{rem}

Finally, we would like to comment why the justification of the high
Reynolds numbers limit can be achieved for MHD in the framework of
Sobolev space, but the corresponding problem for incompressible
Navier-Stokes equations remains open. Note that after receiving the solvability of boundary layer equations, the key issue to verify the Pandtl ansatz is to control the interaction, mainly in the boundary layer, between the vorticity produced by the boundary layer and the outer vorticity generated by the initial one.  For the Navier-Stokes equations, some essential  cancellations are observed
in \cite{AWXY} and \cite{MW1} for recovering the loss of derivatives in the well-posedness theory of classical Prandtl equations. If we use these cancellations to govern the above interaction, it will destroy the divergence free
structure of newly introduced unknown functions for velocity field, and then the uniform
estimation on the pressure function becomes a challenging and unsolved
problem. However, for MHD system, the newly
observed cancellation mechanism for MHD boundary layer equations in \cite{LXY} not only can be utilized to control the desired interaction,
but also preserves the divergence free condition of the newly defined unknown function for velocity field.
It needs to be emphasized that in this analysis,  the non-degeneracy condition on the tangential magnetic field plays an essential role.

This paper is organized as follows: In Section 2, we will construct a
suitable approximation solution and derive some necessary estimates. In Section 3, the error of the approximation is estimated in $L^\infty$-norm
for the proof of Theorem \ref{MAIN THM}. Finally, to make the paper self-contained, we will provide several proofs and computations in the appendix.

\section{Construction of approximate solution} \label{S2}

To prove Theorem \ref{MAIN THM}, we need to construct high order approximate solutions to the problem (\ref{1.1})-\eqref{BCM}. Precisely, we take the forms of approximate solutions as follows:
\begin{align}\label{2.1}
\left\{
\begin{array}{ll}
(\mathbf{u^a},\mathbf{H^a})(t,x,y)&=(\mathbf{u^0},\mathbf{H^0})(t,x,y)+\big(u^0_b, \sqrt{\epsilon} v_b^0, h^0_b, \sqrt{\epsilon} g_b^0\big)\big(t,x,\frac{y}{\sqrt{\epsilon}}\big)\\
&\quad+\sqrt{\epsilon}\Big[(\mathbf{u^1},\mathbf{H^1})(t,x,y)+\big(u^1_b, \sqrt{\epsilon} v_b^1, h^1_b, \sqrt{\epsilon} g_b^1\big)\big(t,x,\frac{y}{\sqrt{\epsilon}}\big)\Big],\\
\hspace{.38in} p^a(t,x,y)&=p^0(t,x,y)+\sqrt{\epsilon}p^1(t,x,y)+\epsilon p^1_b(t,x,\frac{y}{\sqrt{\epsilon}}),
\end{array}
\right.
\end{align}
where 
the functions with subscript $b$ denote the boundary layer profile. In the next six subsections, we will give the construction of the profiles in the above approximation \eqref{2.1}.

Keep in mind that the fast variable $\eta=\frac{y}{\sqrt{\epsilon}}$, and in the following derivation we assume first that for $i=0,1$,
\begin{align*}
\lim_{\eta\rightarrow +\infty}(u_b^i, v_b^i, h_b^i, g_b^i)(t,x,\eta)={\bf{0}}, \quad \lim_{\eta\rightarrow +\infty}p_b^1(t,x,\eta)=0,
\end{align*}
which means the boundary layer profiles decay  to zero away from the boundary. 

\subsection{Zeroth-order inner flow}\label{S1}

From the arguments in the previous section, we know that the leading order inner flow $(\mathbf{u^0}, \mathbf{H^0}, p^0)(t,x,y)$ in the ansatz \eqref{2.1} satisfies the following initial-boundary value problem for incompressible ideal MHD equations,
 \begin{align}\label{LE0}
\left\{
\begin{array}{ll}
\p_t \mathbf{u^0}+(\mathbf{u^0}\cdot\nabla)\mathbf{u^0}+\nabla p^0-(\mathbf{H^0}\cdot\nabla)\mathbf{H^0}=\mathbf{0},\\
\p_t \mathbf{H^0}+(\mathbf{u^0}\cdot\nabla)\mathbf{H^0}-(\mathbf{H^0}\cdot\nabla)\mathbf{u^0}=\mathbf{0},\\
\nabla\cdot \mathbf{u^0}=0,\qquad \nabla\cdot \mathbf{H^0}=0,\\
(v^0, g^0)|_{y=0}=\mathbf{0}, \quad (\mathbf{u^0}, \mathbf{H^0})|_{t=0}=(\mathbf{u_0}, \mathbf{H_0})(x,y)
\end{array}
\right.
\end{align}
in the domain $\{(t,x,y)|t>0, x\in\mathbb{T}, y\in\mathbb{R}_+\}$, where the velocity field $\mathbf{u^0}=(u^0,v^0)$ and magnetic field $\mathbf{H^0}=(h^0,g^0)$. Note that the equations of \eqref{LE0} can also be obtained by putting the ansatz (\ref{2.1}) into (\ref{1.1}), setting the terms of order $\epsilon^0$ equal to zero and letting the fast variable $\eta\rightarrow +\infty$.

The well-posedness of problem (\ref{LE0}) is guaranteed by the results in \cite{Secchi, YM} that can be stated as follows.
\begin{prop}\label{PROP1.1}
Let $m>1$ be a integer, and let the initial data $(\mathbf{u_0}, \mathbf{H_0})(x,y)\in H^m(\mathbb{T}\times\mathbb{R}_+)$ 
satisfy that $\mathbf{u_0}$ and $\mathbf{H_0}$ are divergence free vector fields tangent to the boundary.
 Then there exists a time $T_0>0$ and a smooth solution $(\mathbf{u^0}, \mathbf{H^0}, p^0)(t,x,y)$ to (\ref{LE0}) satisfying
\begin{align*}
(\mathbf{u^0}, \mathbf{H^0}, \nabla p^0)(t,x,y)\in&~\bigcap_{j=0}^{m}C^{j}\big([0,T_0]; H^{m-j}(\mathbb{T}\times\mathbb{R}_+)\big),
\end{align*}
\end{prop}

\begin{rem}\label{rem_e0}
When the initial data of \eqref{LE0} satisfy the assumption of initial data in  Theorem \ref{MAIN THM}, we can gain more information on the trace of the tangential components $u^0$ and $h^0$, which will be used to ensure the compatibility conditions of the problems investigated in the following subsections. 
More precisely, 
\begin{align}\label{ass-e}
 u^0(t,x,0)|_{t=0}=u_0(x,0)=0,\quad h^0(t,x,0)|_{t=0}=h_0(x,0)\geq\delta_0>0
\end{align}
 and
\begin{align}\label{ass-e1}
\partial_t^i u^0(t,x,0)|_{t=0}=0,\quad \partial_t^{i-1}\partial_y h^0(t,x,0)|_{t=0}=0,\qquad 1\leq i\leq [\frac{m}{2}]-3.
\end{align}
 Then,
by the properties of the solution $(\mathbf{u^0}, \mathbf{H^0}, p^0)(t,x,y)$ established in Proposition \ref{PROP1.1}, it is not hard to see that there exists a time $T_1 \leq T_0$, such that the boundary value $h^0(t, x,0)\geq \frac{\delta_0}{2}$ for all $t\in[0,T_1]$. Moreover, the trace theorem yields
\begin{equation}\label{ass_outflow}
\sup_{0\leq t\leq T_1}\sum_{i=0}^{m-1}\big\|\p_t^i(u^0, h^0, \partial_x p^0)(t,x,0)\big\|_{H^{m-1-i}(\mathbb{T}_x)}
<+\infty.
	\end{equation}
\end{rem}
After establishing the leading order inner profile $(\mathbf{u^0}, \mathbf{H^0}, p^0)(t,x,y)$, we now turn to construct the leading order MHD boundary layer profile.

\subsection{Zero-order boundary layer}\label{S2}

As in \cite{LXY}, we know that the zero-order MHD boundary layer profile 
$(u^0_b, v^0_b, h^0_b, g^0_b)(t, x,\eta)$ 
is given by
\begin{equation}\label{bl_def}
\begin{cases}
	(u^0_b,h^0_b)(t,x,\eta)~:=~(u^p,h^p)(t,x,\eta)-(u^0,h^0)(t,x,0),\\
	v^0_b(t,x,\eta)~:=~\int_\eta^\infty\partial_x u^0_b(t,x,\tilde\eta)d\tilde\eta,\quad g^0_b(t,x,\eta)~:=~\int_\eta^\infty\partial_x h^0_b(t,x,\tilde\eta)d\tilde\eta,
\end{cases}\end{equation}
and $(u^p, h^p)(t, x,\eta)$ can be solved by the following boundary layer system:
\begin{equation}\label{LB01}
\begin{cases}
\partial_t u^p+(u^p\partial_x+v^p\partial_\eta)u^p-(h^p\partial_x+g^p\partial_\eta)h^p=\mu\partial_\eta^2 u^p-\partial_x p^0(t,x,0),\\
\partial_t h^p+(u^p\partial_x+v^p\partial_\eta)h^p-(h^p\partial_x+g^p\partial_\eta)u^p
=\kappa \partial_\eta^2h^p,\\
\partial_x u^p+\partial_\eta v^p=0,\quad\partial_x h^p+\partial_\eta g^p=0,\\
(u^p, v^p,\partial_\eta h^p, g^p)|_{\eta=0}=\mathbf 0,
\quad \lim\limits_{\eta\rightarrow+\infty}(u^p, h^p)(t,x,\eta)=(u^0,h^0)(t,x,0),\\
(u^p, h^p)|_{t=0}=(u^0,h^0)(0,x,0)=\big(0, h_0(x,0)\big)
\end{cases}\end{equation}
in the domain $\{(t,x,\eta)|t\in[0,T_0],  x\in\mathbb{T}, \eta\in\mathbb{R}_+\}$, where we have used \eqref{ass-e} in the above initial data. 
Moreover, it follows that from \eqref{ass-e},
\begin{align}\label{positive-ini}
	h^p(0,x,\eta)~\geq~\delta_0~>~0.
\end{align}

By the main theorem in \cite{LXY}, we have the local well-posedness theory of solutions to the initial-boundary value problem (\ref{LB01}). Before we state the well-posedness theorem,
let us introduce some weighted Sobolev spaces used in this subsection. Denote by
\begin{align*}
\Omega:=\big\{(x,\eta):x\in\mathbb{T},\quad\eta\in\mathbb{R}_+\big\}.
\end{align*}
For any $l\in\mathbb{R},$  denote by $L_l^2(\Omega)$ the weighted Lebesgue space with respect to the spatial variables:
\begin{align*}
L_l^2(\Omega):=\Big\{f(x,\eta):\Omega\rightarrow\mathbb{R},\quad
\|f\|_{L^2_l(\Omega)}:=\Big(\int_{\Omega}\langle \eta\rangle^{2l}|f(x,\eta)|^2dx d\eta\Big)^{\frac{1}{2}}<+\infty\Big\},\qquad \langle \eta\rangle=1+\eta,
\end{align*}
and then, for any given $m\in\mathbb{N},$ denote by $H_l^m(\Omega)$ the weighted Sobolev spaces:
$$
H_l^m(\Omega)~:=~\Big\{f(x,\eta):~\Omega\rightarrow\mathbb{R},~\|f\|_{H_l^m(\Omega)}:=\Big(\sum_{m_1+m_2\leq m}\|\langle \eta\rangle^{l+m_2}\p_x^{m_1}\p_\eta^{m_2}f\|_{L^2(\Omega)}^2\Big)^{\frac{1}{2}}<+\infty\Big\}.
$$
Combining Remark \ref{rem_e0} with the condition \eqref{positive-ini}, we have the following result by the main theorem in \cite{LXY}.
\begin{prop}\label{PROP2}
Let $(\mathbf{u^0}, \mathbf{H^0}, p^0)(t,x,y)$ be the leading order inner flow with the initial data $(\mathbf{u_0}, \mathbf{H_0})$ which satisfy the assumptions of Theorem \ref{MAIN THM}. 
Then, there exist a positive time $0<T_2\leq T_1$ and a unique solution $(u^p, v^p, h^p, g^p)(t,x,\eta)$ to the initial boundary value problem (\ref{LB01}), such that
\begin{align}\label{positive-hp}
	h^p(t,x,\eta)~\geq~\frac{\delta_0}{2},\quad \forall~ (t,x,\eta)\in[0,T_2]\times\Omega
\end{align}
with the constant $\delta_0>0$ given in \eqref{positive-ini}, and for any $l\geq0,$
\begin{align}\label{est_main1}
	\big(u^p(t,x,\eta)-u^0(t,x,0), h^p(t,x,\eta)-h^0(t,x,0)\big)&\in\bigcap_{i=0}^{[m/2]-1}W^{i,\infty}\Big(0,T_2;H_l^{[m/2]-1-i}(\Omega)\Big).
	\end{align}
Moreover, it holds that for the profile $(u_b^0, v_b^0, h_b^0, g_b^0)(t,x,\eta)$ defined by \eqref{bl_def},
\begin{align}
	(u_b^0, h_b^0)(t,x,\eta)&\in\bigcap_{i=0}^{[m/2]-1}W^{i,\infty}\Big(0,T_{2};H_{l}^{[m/2]-1-i}(\Omega)\Big),\label{est_main10}\\
	(v_b^0, g_b^0)(t,x,\eta),~(\partial_\eta v_b^0, \partial_\eta g_b^0)(t,x,\eta)&\in\bigcap_{i=0}^{[m/2]-2}W^{i,\infty}\big(0,T_{2};H_{l}^{[m/2]-2-i}(\Omega)\big).\label{est_main12}
\end{align}
\end{prop}
\begin{proof}[\bf{Proof.}]
	First of all, by using \eqref{ass-e} and \eqref{ass-e1} it can be deduced that, the initial-boundary values of \eqref{LB01} satisfy the compatibility conditions up to order $[m/2]-3$.
Then, the local well-posedness theory of the solution $(u^p, v^p, h^p, g^p)(t,x,\eta)$ to problem (\ref{LB01}), and the relation \eqref{est_main1} has been obtained in \cite{LXY}. Note that the initial data of $(u^p, h^p)$, given in \eqref{LB01}, is independent of normal variable $\eta$, therefore the index $l$ of weight with respect to $\eta$ in \eqref{est_main1} can be arbitrary large. Moreover,  \eqref{est_main10} follows automatically by combining $\eqref{bl_def}_1$ with \eqref{est_main1}. Therefore, we only need to show \eqref{est_main12}.
	
	 From $\eqref{bl_def}_2$ we obtain $(\partial_\eta v_b^0, \partial_\eta g_b^0)=-(\partial_x u_b^0, \partial_x h_b^0),$ and for $\alpha\in\mathbb{N}^2, l\geq0,$
	 \begin{align*}
	 	\big|\langle\eta\rangle^l\partial_{tx}^\alpha v_b^0(t,x,\eta)\big|&=\big|\langle\eta\rangle^l \int_\eta^\infty \partial_{tx}^\alpha\partial_x u_b^0(t,x,\tilde\eta)d\tilde\eta\big| \leq \int_\eta^\infty \big|\langle\tilde\eta\rangle^l  \partial_{tx}^\alpha\partial_x u_b^0(t,x,\tilde\eta)\big|d\tilde\eta\\
	 	&\lesssim\langle\eta\rangle^{\frac{1}{2}-l_0}\big\|\langle\eta\rangle^{l+l_0}  \partial_{tx}^\alpha\partial_x u_b^0(t,x,\eta)\big\|_{L^2_\eta}
	 \end{align*}
	 provided $l_0>\frac{1}{2}$, which implies that
	 \begin{align*}
	 	\big\|\partial_{tx}^\alpha v_b^0(t,\cdot)\big\|_{L^2_l(\Omega)}\lesssim\big\|\langle\eta\rangle^{\frac{1}{2}-l_0}\big\|_{L^2_\eta}\cdot\big\|  \partial_{tx}^\alpha\partial_x u_b^0(t,\cdot)\big\|_{L^2_{l+l_0}(\Omega)}\lesssim\big\|  \partial_{tx}^\alpha\partial_x u_b^0(t,\cdot)\big\|_{L^2_{l+l_0}(\Omega)}
	 \end{align*}
	 provided $l_0>1.$ Similarly, we have that for $l_0>1$,
	  \begin{align*}
	 	\big\|\partial_{tx}^\alpha g_b^0(t,\cdot)\big\|_{L^2_l(\Omega)}\lesssim\big\|  \partial_{tx}^\alpha\partial_x h_b^0(t,\cdot)\big\|_{L^2_{l+l_0}(\Omega)}.
	 \end{align*}
	 By using the above two inequalities and combining with \eqref{est_main10}, we get \eqref{est_main12} immediately.
\end{proof}

From \eqref{bl_def} and the divergence free conditions in \eqref{LB01} we have another expression for $(u_b^0, v_b^0, h_b^0, g_b^0)$:
\begin{align}\label{def_vg}
\begin{cases}
(u^0_b, h^0_b)(t,x,\eta)~=~&(u^p, h^p)(t,x,\eta)-(u^0, h^0)(t,x,0),\\
	(v^0_b, g^0_b)(t,x,\eta)~=~&(v^p, g^p)(t,x,\eta)+\eta(\partial_x u^0, \partial_x h^0)(t,x,0)\\
	&+\int_0^\infty\Big(\partial_x u^p(t,x,\eta)-\partial_x u^0(t,x,0),  \partial_x h^p(t,x,\eta)-\partial_x h^0(t,x,0)\Big)d\eta,
\end{cases}\end{align}
which implies that by virtue of the boundary conditions $(v^p,g^p)|_{\eta=0}=0$ in \eqref{LB01},
\begin{align}\label{bd-vg}
	(v^0_b, g^0_b)(t,x,0)~=~\int_0^\infty\Big(\partial_x u^p(t,x,\eta)-\partial_x u^0(t,x,0),  \partial_x h^p(t,x,\eta)-\partial_x h^0(t,x,0)\Big)d\eta.
\end{align}
Then we can derive the problem of $(u^0_b, v^0_b, h^0_b, g^0_b)(t,x,\eta)$. Indeed, from \eqref{def_vg} and \eqref{bd-vg} it holds
\begin{align*}
\begin{cases}
	(u^p, h^p)(t,x,\eta)~=~(u^0_b, h^0_b)(t,x,\eta)+(u^0,h^0)(t,x,0),\\
	(v^p, g^p)(t,x,\eta)~=~(v^0_b, g^0_b)(t,x,\eta)-(v^0_b, g^0_b)(t,x,0)-\eta(\partial_x u^0, \partial_x h^0)(t,x,0).
\end{cases}\end{align*}
By using the notation of $\bar{f}(t,x)$ to stand for the trace of function $f(t,x,y)$ on the boundary $\{y=0\}$, we substitute the above expression into the problem \eqref{LB01} to obtain that
\begin{align}\label{LB0}
\left\{\begin{array}{ll}
\p_t u_b^0+(\overline{u^0}+u_b^0)\p_x u_b^0+(v_b^0-\overline{v_b^0}-\eta\overline{\p_x u^0})\p_\eta u_b^0-(\overline{h^0}+h_b^0)\p_x h_b^0-(g_b^0-\overline{g^0_b}-\eta\overline{\p_x h^0})\p_\eta h_b^0\\
\qquad+\overline{\partial_x u^0} ~u_b^0-\overline{\partial_x h^0}~ h_b^0=\mu\partial^2_\eta u_b^0,\\
\p_t h_b^0+(\overline{u^0}+u_b^0)\p_x h_b^0+(v_b^0-\overline{v_b^0}-\eta\overline{\p_x u^0})\p_\eta h_b^0-(\overline{h^0}+h_b^0)\p_x u_b^0-(g_b^0-\overline{g^0_b}-\eta\overline{\p_x h^0})\p_\eta u_b^0\\
\qquad+\overline{\partial_x u^0} ~h_b^0-\overline{\partial_x h^0}~ u_b^0=\kappa\partial^2_\eta h_b^0,\\
\p_x u^0_b+\p_\eta v^0_b=0,\qquad \p_x h^0_b+\p_\eta g^0_b=0,\\
(u^0_b,  \partial_\eta h^0_b)|_{\eta=0}=-\big(\overline{u^0}(t,x), 0\big),\quad \lim\limits_{\eta\rightarrow+\infty}(u_b^0, h_b^0)={\bf{0}},\\
(u^0_b,  h^0_b)|_{t=0}=\mathbf{0}, 
\end{array}
\right.
\end{align}
where we have used the equations of $(u^0, h^0)$ on the boundary $\{y=0\}$ from the problem \eqref{LE0}-\eqref{BE}. 
Moreover, from \eqref{LB0} we know that $g_b^0$ satisfies the following equation:
\begin{align}\label{eq_g}
\begin{aligned}
	&\partial_t g_0^b+(\overline{u^0}+u_b^0)\p_x g_b^0+(v_b^0-\overline{v_b^0}-\eta\overline{\p_x u^0})\p_\eta g_b^0-(\overline{h^0}+h_b^0)\p_x v_b^0-(g_b^0-\overline{g^0_b}-\eta\overline{\p_x h^0})\p_\eta v_b^0\\
&-\partial_x\big(\overline{g_b^0}+\eta\overline{\partial_x h^0}\big)u_b^0-\overline{\partial_x h^0}v_b^0+\partial_x\big(\overline{v_b^0}+\eta\overline{\partial_x u^0}\big)h_b^0+\overline{\partial_x u^0}~ g_b^0=\kappa\partial^2_\eta g_b^0.
\end{aligned}\end{align}

After constructing the leading order inner flow $(\mathbf{u^0}, \mathbf{H^0}, p^0)$ and boundary layer profile $(u_b^0, v_b^0, h_b^0, g_b^0)$, we proceed to construct the next order inner MHD flow.

\subsection{First-order inner flow}\label{S3}
Put the ansatz (\ref{2.1}) into (\ref{1.1}) and set the terms of order $\epsilon^{1/2}$ equal to zero, then letting $\eta\rightarrow +\infty$ yields the first order inner flow $(\mathbf{u^1},\mathbf{H^1}, p^1)(t,x,y)$
satisfies the following linearized ideal MHD equations in the region $\{(t,x,y)|t\in[0, T_2],  x\in\mathbb{T}, y\in\mathbb{R}_+\}$:
\begin{align}\label{LE1}
\left\{\begin{array}{ll}
\p_t \mathbf{u^1}+(\mathbf{u^0}\cdot\nabla)\mathbf{u^1}+\nabla p^1-(\mathbf{H^0}\cdot\nabla)\mathbf{H^1}+(\mathbf{u^1}\cdot\nabla)\mathbf{u^0}-(\mathbf{H^1}\cdot\nabla)\mathbf{H^0}=\mathbf{0},\\
\p_t \mathbf{H^1}+(\mathbf{u^0}\cdot\nabla)\mathbf{H^1}-(\mathbf{H^0}\cdot\nabla)\mathbf{u^1}+(\mathbf{u^1}\cdot\nabla)\mathbf{H^0}-(\mathbf{H^1}\cdot\nabla)\mathbf{u^0}=\mathbf{0},\\
\nabla\cdot \mathbf{u^1}=0,\qquad \nabla\cdot \mathbf{H^1}=0,
\end{array}
\right.
\end{align}
where $\mathbf{u^1}=(u^1,v^1)$ and $\mathbf{H^1}=(h^1,g^1)$.
The initial data is chosen to be zero:
\begin{align}
\label{SI}
(\mathbf{u^1}, \mathbf{H^1})|_{t=0}={\bf{0}},
\end{align}
and the boundary conditions of $(v^1_e, g^1_e)$ in (\ref{LE1}) are imposed by
\begin{align}
\label{NB}
(v^1_e, g^1_e)(t,x,0)=-(v^0_b, g^0_b)(t,x,0)=\left(-\int_{0}^{\infty}\p_xu_b^0(t,x,\tilde{\eta})d\tilde{\eta}, -\int_{0}^{\infty}\p_xh_b^0(t,x,\tilde{\eta})d\tilde{\eta}\right),
\end{align}
to homogenous the boundary conditions of oder $\epsilon^{1/2}$ for the normal components of $(\mathbf{u^\epsilon}, \mathbf{H^\epsilon})$. It is noted that the boundary condition \eqref{NB} is sufficient to solve the initial-boundary value problem (\ref{LE1})-(\ref{NB}). 
Moreover, from \eqref{est_main10} we know that
\begin{align}\label{ass-e1}
	(v^1_e, g^1_e)(t,x,0)\in\bigcap_{i=0}^{[m/2]-3}W^{i,\infty}\Big(0,T_{2};H^{[m/2]-2-i}(\mathbb{T}_x)\Big).
\end{align}

By a similar argument as for the Proposition \ref{PROP1.1} for initial-boundary value problem of the linearized ideal MHD equations (\ref{LE1})-(\ref{NB}), or as a direct consequence of the main results in \cite{O-S},   we have
\begin{prop}\label{PROP1.3}
Let $(\mathbf{u^0}, \mathbf{H^0}, p^0)(t,x,y)$ be the leading order inner flow with the initial data $(\mathbf{u_0}, \mathbf{H_0})$ which satisfy the assumptions of Theorem \ref{MAIN THM}.  Let the boundary condition \eqref{ass-e1} hold. 
Then there exists a smooth solution $(\mathbf{u^1}, \mathbf{H^1}, p^1)(t,x,y)$ to problem (\ref{LE1})-(\ref{NB}) in the time interval $[0, T_3]$, such that
\begin{align*}
(\mathbf{u^1}, \mathbf{H^1}, \nabla p^1)(t,x,y)\in&~\bigcap_{j=0}^{[m/2]-2}C^j\Big([0,T_3]; H^{[m/2]-2-j}\big(\mathbb{T}\times\mathbb{R}_+\big)\Big),
\end{align*}
where $0<T_3\leq T_2$ is the local lifespan of solution $(\mathbf{u^1}, \mathbf{H^1}, p^1)(t,x,y)$.
\end{prop}

Now  we consider the following approximation solutions to the problem  (\ref{1.1})-\eqref{BCM}:
\begin{align}\label{app0}
\left\{\begin{array}{ll}
(\mathbf{u^{a0}},\mathbf{H^{a0}})(t,x,y)&=(\mathbf{u^0},\mathbf{H^0})(t,x,y)+\big(u^0_b, \sqrt{\epsilon} v_b^0, h^0_b, \sqrt{\epsilon} g_b^0\big)\big(t,x,\frac{y}{\sqrt{\epsilon}}\big)+\sqrt{\epsilon}(\mathbf{u^1},\mathbf{H^1})(t,x,y),\\
\hspace{.38in}p^{a0}(t,x,y)~&=p^0(t,x,y)+\sqrt{\epsilon}p^1(t,x,y).
\end{array}
\right.
\end{align}
From the above construction of $(\mathbf{u^i}, \mathbf{H^i}, p^i) (i=0,1)$ and $(u_b^0, v_b^0, h_b^0, g_b^0)$, a direct calculation reads
\begin{align}\label{EA0}
\left\{\begin{array}{ll}
\p_t \mathbf{u^{a0}}+(\mathbf{u^{a0}}\cdot\nabla)\mathbf{u^{a0}}+\nabla p^{a0}-(\mathbf{H^{a0}}\cdot\nabla)\mathbf{H^{a0}}-\mu\epsilon\triangle \mathbf{u^{a0}}=\mathbf{R^{a0}_u},\\
\p_t \mathbf{H^{a0}}+(\mathbf{u^{a0}}\cdot\nabla)\mathbf{H^{a0}}-(\mathbf{H^{a0}}\cdot\nabla)\mathbf{u^{a0}}-\kappa\epsilon\triangle \mathbf{H^{a0}}=\mathbf{R^{a0}_H},\\
\nabla\cdot\mathbf{u^{a0}}=0,\quad\nabla\cdot\mathbf{H^{a0}}=0,\\
\big(\mathbf{u^{a0}},\mathbf{H^{a0}}\big)|_{t=0}=(\mathbf{u_0}, \mathbf{H_0})(x,y),\\
\mathbf{u^{a0}}|_{y=0}=\big(\sqrt{\epsilon}u^1(t,x,0),0\big), ~\partial_y h^{a0}|_{y=0}=\big(\partial_y h^0+\sqrt{\epsilon}\partial_y h^1\big)(t,x,0),~g^{a0}|_{y=0}=0,
\end{array}\right.
\end{align}
where the remainder terms $\mathbf{R^{a0}_u}=(R^{a0}_1, R^{a0}_2)$ and $\mathbf{R^{a0}_H}=(R^{a0}_3, R^{a0}_4)$ are summarized as follows,
\begin{align}\label{RL1}\begin{cases}
R^{a0}_1=&\big(u^0-\overline{u^0}+\sqrt{\epsilon}u^1\big)\p_x u_b^0+\big[v^0-y\overline{\p_y v^0}+\sqrt{\epsilon}\big(v^1-\overline{v^1}\big)\big]\p_y u_b^0-\big(h^0-\overline{h^0}+\sqrt{\epsilon}h^1\big)\p_x h_b^0\\
&-\big[g^0-y\overline{\p_y g^0}+\sqrt{\epsilon}\big(g^1-\overline{g^1}\big)\big]\p_y h_b^0+u_b^0\big(\p_x u^0-\overline{\p_x u^0}+\sqrt{\epsilon}\partial_x u^1\big)+\sqrt{\epsilon}v_b^0\partial_y u^0\\
&-h_b^0\big(\p_x h^0-\overline{\p_x h^0}+\sqrt{\epsilon}\partial_x h^1\big)-\sqrt{\epsilon}g_b^0\partial_y h^0+R_1^{high},\\
R^{a0}_2=&\sqrt{\epsilon}\Big[\partial_t v_b^0+(u^0+u_b^0)\partial_x v_b^0+\big(v^0+\sqrt{\epsilon}v^1+\sqrt{\epsilon}v_b^0\big)\partial_y v_b^0-(h^0+h_b^0)\partial_x g_b^0\\
&\qquad-\big(g^0+\sqrt{\epsilon}g^1+\sqrt{\epsilon}g_b^0\big)\partial_y g_b^0+v_0^b\partial_y v^0-g_0^b\partial_y g^0-\mu\epsilon\partial_y^2v_b^0\Big]\\
&+u_b^0\big(\partial_x v^0+\sqrt{\epsilon}\partial_x v^1\big)-h_b^0\big(\partial_x g^0+\sqrt{\epsilon}\partial_x g^1\big)+R_2^{high},\\
R^{a0}_3=&\big(u^0-\overline{u^0}+\sqrt{\epsilon}u^1\big)\p_x h_b^0+\big[v^0-y\overline{\p_y v^0}+\sqrt{\epsilon}\big(v^1-\overline{v^1}\big)\big]\p_y h_b^0-\big(h^0-\overline{h^0}+\sqrt{\epsilon}h^1\big)\p_x u_b^0\\
&-\big[g^0-y\overline{\p_y g^0}+\sqrt{\epsilon}\big(g^1-\overline{g^1}\big)\big]\p_y u_b^0+u_b^0\big(\p_x h^0-\overline{\p_x h^0}+\sqrt{\epsilon}\partial_x h^1\big)+\sqrt{\epsilon}v_b^0\partial_y h^0\\
&-h_b^0\big(\p_x u^0-\overline{\p_x u^0}+\sqrt{\epsilon}\partial_x u^1\big)-\sqrt{\epsilon}g_b^0\partial_y u^0+R_3^{high},\\
R^{a0}_4=&\sqrt{\epsilon}\Big[\partial_t g_b^0+(u^0+u_b^0)\partial_x g_b^0+ \big(v^0+\sqrt{\epsilon}v^1+\sqrt{\epsilon}v_b^0\big)\partial_y g_b^0-(h^0+h_b^0)\partial_x v_b^0\\
&\qquad-\big(g^0+\sqrt{\epsilon}g^1+\sqrt{\epsilon}g_b^0\big)\partial_y v_b^0+v_0^b\partial_y g^0-g_0^b\partial_y v^0-\kappa\epsilon\partial_y^2g_b^0\Big]\\
&+u_b^0\big(\partial_x g^0+\sqrt{\epsilon}\partial_x g^1\big)-h_b^0\big(\partial_x v^0+\sqrt{\epsilon}\partial_x v^1\big)+R_4^{high},
	\end{cases}\end{align}
with
\begin{align}\label{RH3}
\begin{aligned}\begin{cases}	
&R_1^{high}=\epsilon\Big\{u^1\partial_x u^1+(v^1+v_b^0)\partial_y u^1-h^1\partial_x h^1-(g^1+g_b^0)\partial_y h^1-\mu\triangle \big(u^0+\sqrt{\epsilon} u^1\big)-\mu\p_x^2u_b^0\Big\},\\
&R_2^{high}=\epsilon\Big\{u^1\partial_x(v^1+v_b^0)+(v^1+v_b^0)\partial_y v^1-h^1\partial_x(g^1+g_b^0)-(g^1+g_b^0)\partial_y g^1-\mu\triangle\big( v^0+\sqrt{\epsilon} v^1\big)-\mu\sqrt{\epsilon}\partial_x^2 v_b^0\Big\},\\
&R_3^{high}=\epsilon\Big\{u^1\partial_x h^1+(v^1+v_b^0)\partial_y h^1-h^1\partial_x u^1-(g^1+g_b^0)\partial_y u^1-\kappa\triangle\big( h^0+\sqrt{\epsilon}h^1\big)-\kappa\p_x^2h_b^0\Big\},\\
&R_4^{high}=\epsilon\Big\{u^1\partial_x(g^1+g_b^0)+(v^1+v_b^0)\partial_y g^1-h^1\partial_x(v^1+v_b^0)-(g^1+g_b^0)\partial_y v^1-\kappa\triangle\big(g^0+\sqrt{\epsilon} g^1\big)-\kappa\sqrt{\epsilon}\partial_x^2 g_b^0\Big\}.
\end{cases}\end{aligned}\end{align}

The leading order terms of the remainders $R^{a0}_i, 1\leq i\leq4$ mainly exist in the boundary layer, i.e., in the neighborhood of size of $\sqrt{\epsilon}$ near the boundary, and it is easy to check that 
$$R^{high}_i~=~O(\epsilon),\quad i=1\sim4.$$
Thanks to the equation \eqref{eq_g} for $g_b^0$ we find that the remainder $R^{a0}_4$ is actually of order $\epsilon$. Indeed, by virtue of \eqref{eq_g} we can rewrite $R^{a0}_4$ as
\begin{align}\label{RH4}
\begin{aligned}
	R_4^{a0}=&\sqrt{\epsilon}\Big\{\big(u^0-\overline{u^0}\big)\partial_x g_b^0+\big[v^0-y\overline{\partial_y v^0}+\sqrt{\epsilon}\big(v^1-\overline{v^1}\big)\big]\partial_y g_b^0-\big(h^0-\overline{h^0}\big)\partial_x v_b^0\\
&\qquad-\big[g^0-y\overline{\partial_y g^0}+\sqrt{\epsilon}\big(g^1-\overline{g^1}\big)\big]\partial_y v_b^0+v_b^0\big(\partial_y g^0-\overline{\partial_y g^0} \big)-g_b^0\big(\partial_y v^0-\overline{\partial_y v^0} \big)\Big\}\\
&+u_b^0\big[\partial_x g^0-y\overline{\partial_{xy}^2 g^0}+\sqrt{\epsilon}\big(\partial_x g^1-\overline{\partial_x g^1}\big)\big]-h_b^0\big[\partial_x v^0-y\overline{\partial_{xy}^2 v^0}+\sqrt{\epsilon}\big(\partial_x v^1-\overline{\partial_x v^1}\big)\big]+R_4^{high}.	
\end{aligned}\end{align}
Note that the major items of the remainders $R^{a0}_1, R^{a0}_2$ and $R^{a0}_3$ with respect to $\epsilon$ 
 are in fact of order $\sqrt{\epsilon}$. 
 Thus as shown in the ansatz \eqref{2.1}, we proceed in the next subsection to construct the boundary layer profile $\epsilon p_b^1(t,x,\eta)$ for pressure to cover the major items in $\epsilon$ of $R^{a0}_2$.  And in subsection \ref{S4}, we will construct boundary layer profile $\sqrt{\epsilon}(u_b^1, \sqrt{\epsilon}v_b^1, h_b^1, \sqrt{\epsilon}g_b^1)(t,x,\eta)$ 
 to cover the major items in $\epsilon$ of $R^{a0}_1$ and $R^{a0}_3$, and homogenize the boundary conditions of \eqref{EA0} as well. 

\subsection{Leading order boundary layer of pressure}\label{Sp}

In order to eliminate the major items in $\epsilon$ of $R_2^{a0}$ given in \eqref{RL1}, that is, the terms of order $\sqrt{\epsilon}$ produced by $(u_b^0, v_b^0,h_b^0, g_b^0)$, we define the boundary layer profile $\epsilon p_b^1(t,x,\eta)$ for pressure in the following way,
\begin{align*}
	\partial_\eta p_b^1=&-\p_t v_b^0-\big(u_b^0+\overline{u^0}\big)\partial_x v_b^0-\big(v_b^0+\overline{v^1}+\eta\overline{\partial_y v^0} \big)\partial_\eta v_b^0+\big(h_b^0+\overline{h^0}\big)\partial_x g_b^0+\big(g_b^0+\overline{g^1}+\eta\overline{\partial_y g^0} \big)\partial_\eta g_b^0+\mu\partial_\eta^2 v_b^0\\
	&-\big(\overline{\partial_x v^1}+\eta\overline{\partial_{xy}^2v^0}\big)u_b^0-\overline{\partial_y v^0}~v_b^0+\big(\overline{\partial_x g^1}+\eta\overline{\partial_{xy}^2g^0}\big)h_b^0+\overline{\partial_y g^0}~g_b^0,
\end{align*}
or
\begin{align}\label{LP}
\begin{aligned}
p_b^1(t,x,\eta)=&\int_\eta^\infty\Big\{\p_t v_b^0+\big(u_b^0+\overline{u^0}\big)\partial_x v_b^0-\big(h_b^0+\overline{h^0}\big)\partial_x g_b^0+\big(\overline{\partial_x v^1}+\tilde\eta\overline{\partial_{xy}^2v^0}\big)u_b^0-\big(\overline{\partial_x g^1}+\tilde\eta\overline{\partial_{xy}^2g^0}\big)h_b^0\Big\}(t,x,\tilde\eta)d\tilde{\eta}\\
&+\Big[-\big(\frac{v_b^0}{2}+\overline{v^1}+\eta\overline{\partial_y v^0}\big)v_b^0+\big(\frac{g_b^0}{2}+\overline{g^1}+\eta\overline{\partial_y g^0}\big)g_b^0+\mu\partial_\eta v_b^0\Big](t,x,\eta).
\end{aligned}\end{align}
By using the above expression \eqref{LP} and combining with \eqref{est_main10}-\eqref{est_main12}, we can obtain that
\begin{align}\label{est_p}
	p_b^1(t,x,\eta),~\partial_\eta p_b^1(t,x,\eta)&\in\bigcap_{i=0}^{[m/2]-3}W^{i,\infty}\big(0,T_{2};H_{l}^{[m/2]-3-i}(\Omega)\big),\quad \forall l\geq0.
\end{align}

\subsection{First-order boundary layer}\label{S4}

In this subsection, we construct the first-oder boundary layer profiles $(u_b^1, \sqrt{\epsilon}v_b^1, h_b^1, \sqrt{\epsilon}g_b^1)(t,x,\eta)$ of the ansatz \eqref{2.1} to cover the terms of order $O(\sqrt{\epsilon})$ 
in remainders $R^{a0}_1$ and $R^{a0}_3$ given in \eqref{RL1}. To this end, applying \eqref{2.1} into the equations $(\ref{1.1})_1$ and $\eqref{1.1}_3$ and considering the terms of order $\sqrt{\epsilon}$, 
it leads to the following equations for tangential components $u_b^1(t,x,\eta)$ and $h_b^1(t,x,\eta)$ respectively:
\begin{align}\label{LU1}
\begin{aligned}
&\p_t u_b^1+\big(u_b^0+\overline{u^0}\big)\p_x u_b^1
+\big(v_b^0+\overline{v^1}+\eta\overline{\p_y v^0}\big)\p_{\eta}u_b^1-\big(h_b^0+\overline{h^0}\big)\p_x h_b^1-\big(g_b^0+\overline{g^1}+\eta\overline{\p_y g^0}\big)\p_{\eta}h_b^1\\
&+\big(\p_x u_b^0+\overline{\p_x u^0}\big)u_b^1+\p_{\eta}u_b^0~v_b^1-\big(\p_x h_b^0+\overline{\p_x h^0}\big)h_b^1
-\p_{\eta}h_b^0~g_b^1-\mu\p^2_\eta u_b^1\\
=&-\big(\overline{u^1}+\eta\overline{\partial_y u^0} \big)\partial_x u_b^0-\big(\eta\overline{\partial_y v^1}+\frac{\eta^2}{2}\overline{\partial_y^2v^0} \big)\partial_\eta u_b^0+\big(\overline{h^1}+\eta\overline{\partial_y h^0} \big)\partial_x h_b^0+\big(\eta\overline{\partial_y g^1}+\frac{\eta^2}{2}\overline{\partial_y^2g^0} \big)\partial_\eta h_b^0\\
&-\big(\overline{\partial_x u^1}+\eta\overline{\partial_{xy}^2u^0}\big)~u_b^0-\overline{\partial_y u^0}~v_b^0+\big(\overline{\partial_x h^1}+\eta\overline{\partial_{xy}^2h^0}\big)~h_b^0+\overline{\partial_y h^0}~g_b^0,
\end{aligned}\end{align}
and
\begin{align}\label{LM}\begin{aligned}
&\p_t h_b^1+\big(u_b^0+\overline{u^0}\big)\p_x h_b^1+\big(v_b^0+\overline{v^1}+\eta\overline{\p_y v^0}\big)\p_\eta h_b^1-\big(h_b^0+\overline{h^0}\big)\p_x u_b^1-\big(g_b^0+\overline{g^1}+\eta\overline{\p_y g^0}\big)\p_\eta u_b^1\\
&+\big(\partial_x h_b^0+\overline{\p_x h^0}\big)u_b^1+\p_\eta h_b^0~v_b^1-\big(\partial_x u_b^0+ \overline{\partial_x u^0}\big)h_b^1-\p_\eta u_b^0~g_b^1-\kappa\p_\eta^2h_b^1\\
=&-\big(\overline{u^1}+\eta\overline{\partial_y u^0} \big)\partial_x h_b^0-\big(\eta\overline{\partial_y v^1}+\frac{\eta^2}{2}\overline{\partial_y^2v^0} \big)\partial_\eta h_b^0+\big(\overline{h^1}+\eta\overline{\partial_y h^0} \big)\partial_x u_b^0+\big(\eta\overline{\partial_y g^1}+\frac{\eta^2}{2}\overline{\partial_y^2g^0}\big)\partial_\eta u_b^0\\
&-\big(\overline{\partial_x h^1}+\eta\overline{\partial_{xy}^2h^0}\big)u_b^0-\overline{\partial_y h^0}~v_b^0+\big(\overline{\partial_x u^1}+\eta\overline{\partial_{xy}^2u^0}\big)h_b^0+\overline{\partial_y u^0}~g_b^0.
\end{aligned}\end{align}
Comparing the terms of right-hand side of (\ref{LU1}) and \eqref{LM} with the errors   $R_1^{a0}$ and $R_3^{a0}$ given in \eqref{RL1}, it implies that we utilize the equations of $u_1^p$ and $h_1^p$ to cancel the terms of order $\sqrt{\epsilon}$ in $R_1^{a0}$ and $R_3^{a0}$ respectively.
Of course, we still impose the divergence free conditions:
\begin{align}\label{div}
	\partial_x u_b^1+\partial_\eta v_b^1=0,\qquad \partial_x h_b^1+\partial_\eta g_b^1=0,
\end{align}
and the zero initial data:
\begin{align}\label{ILB}
(u_b^1, h_b^1)|_{t=0}={\bf {0}}.
\end{align}
Moreover, we choose the boundary conditions
\begin{align}\label{bd_b1}
(u_b^1, v_b^1, \partial_\eta h_b^1, g_b^1)|_{\eta=0}~=	~\big(-u^1(t,x,0), 0, -\partial_y h^0(t,x,0), 0\big),
\end{align}
to eliminate the major items in $\epsilon$ of  
 the boundary conditions, given in \eqref{EA0}, of approximate solutions \eqref{app0}. Thereby, we obtain the initial-boundary value problem \eqref{LU1}-\eqref{bd_b1} for the first order boundary layer profile $(u_b^1, v_b^1, h_b^1, g_b^1)(t, x,\eta).$

\begin{rem}\label{rem}
	From the above construction, the components $v_b^1$ and $g_b^1$ of $(u_b^1, v_b^1, h_b^1, g_b^1)$ are determined by the divergence free conditions \eqref{div} and boundary conditions \eqref{bd_b1}, in other words, 
	\begin{align*}
		(v_b^1, g_b^1)(t,x,\eta)~=~-\int_0^\eta\big(\partial_x u_b^1, \partial_x h_b^1\big)(t,x,\tilde\eta)d\tilde\eta,
	\end{align*}
	which shows that in general, $(v_b^1, g_b^1)$ doesn't decay to zero as $\eta\rightarrow+\infty$. Note that such profile $(v_b^1, g_b^1)$ is
	slightly different from the corresponding one given in the ansatz \eqref{2.1}, which is expected to decay rapidly as $\eta\rightarrow+\infty$. In fact, the difference between $(v_b^1, g_b^1)$, constructed in this subsection, and the corresponding one in \eqref{2.1} is only a function independent of normal variable $\eta$.
\end{rem}

Now, we establish the well-posedness of the solution $(u_b^1, v_b^1, h_b^1, g_b^1)(t,x,\eta)$ to problem (\ref{LU1})-(\ref{bd_b1}).
To this end, we use the energy methods developed in \cite{LXY}. Specifically speaking, from the divergence free conditions for $(\mathbf{u^i}, \mathbf{H^i}) $ and $(u_b^i, v_b^i, h_b^i, g_b^i) (i=0,1)$,  we rewrite the equation (\ref{LM}) as follows.
\begin{align}\label{PE}\begin{aligned}
&\p_th_b^1+\p_\eta\Big[\big(v_b^0+\overline{v^1}+\eta\overline{\p_y v^0}\big)h_b^1-\big(u_b^0+\overline{u^0}\big)g_b^1-\big(g_b^0+\overline{g^1}+\eta\overline{\p_y g^0}\big)u_b^1+\big(h_b^0+\overline{h^0}\big)v_b^1\Big]-\kappa\p_\eta^2h_b^1\\
=~&\partial_\eta\Big[\big(\eta\overline{\partial_y g^1}+\frac{\eta^2}{2}\overline{\partial_y^2g^0}\big)u_b^0-\big(\overline{h^1}+\eta\overline{\partial_y h^0} \big)v_b^0-\big(\eta\overline{\partial_y v^1}+\frac{\eta^2}{2}\overline{\partial_y^2v^0}\big)h_b^0+\big(\overline{u^1}+\eta\overline{\partial_y u^0} \big)g_b^0 \Big].
\end{aligned}\end{align}
Define
\begin{align*}
\psi(t,x,\eta)=\int_0^\eta h_b^1(t,x,\tilde{\eta})d\tilde{\eta},
\end{align*}
and it implies that by the divergence free condition $\partial_x h_b^1+\partial_\eta g_b^1=0$ and the boundary condition $g_b^1|_{\eta=0}=0$,
\begin{align*}
	\partial_x\psi(t,x,\eta)~=~-g_b^1(t,x,\eta).
\end{align*}
 Integrating the equation (\ref{PE}) over $[0,\eta]$ leads to
\begin{align}\label{PE1}\begin{aligned}
&\p_t\psi+\big(u_b^0+\overline{u^0}\big)\p_x\psi+\big(v_b^0+\overline{v^1}+\eta\overline{\p_y v^0}\big)\p_\eta \psi-\big(g_b^0+\overline{g^1}+\eta\overline{\p_y g^0}\big)u_b^1
+\big(h_b^0+\overline{h^0}\big)v_b^1-\kappa\p_\eta^2\psi\\
=~&\big(\eta\overline{\partial_y g^1}+\frac{\eta^2}{2}\overline{\partial_y^2g^0}\big)u_b^0-\big(\overline{h^1}+\eta\overline{\partial_y h^0} \big)v_b^0-\big(\eta\overline{\partial_y v^1}+\frac{\eta^2}{2}\overline{\partial_y^2v^0}\big)h_b^0+\big(\overline{u^1}+\eta\overline{\partial_y u^0} \big)g_b^0\\
&+\overline{u^1g^1-h^1v^1+\kappa\p_y h^0},
\end{aligned}\end{align}
where we have used the following boundary conditions
\begin{align*}
(u_b^0, v_b^0, g_b^0)|_{\eta=0}=-\big(\overline{u^0}, \overline{v^1}, \overline{g^1}\big)(t,x), \quad (v_b^1,\partial_\eta h_b^1)|_{\eta=0}=\big(0,-\overline{\partial_y h^0} \big)(t, x).
\end{align*}
Then, for simplicity of presentation, we only give the outline about the applications of the energy estimate method developed in \cite{LXY} here. First, we derive the $L^2$-estimates of
$$\p_{tx}^{\alpha}\p^{j}_\eta (u_b^1, h_b^1), \quad \alpha\in\mathbb{N}^2,~j\in\mathbb{N},~ |\alpha|+i\leq k, ~|\alpha|\leq k-1$$
 from the problem (\ref{LU1})-(\ref{bd_b1}) by standard energy methods. Next, it is left to derive 
$L^2$-estimates of $\p_{tx}^{\alpha} (u_b^1, h_b^1)$ with $|\alpha|=k$. By introducing the following new quantities:
\begin{align*}
u^\alpha_\tau=\p_{tx}^{\alpha} u_b^1-\frac{\p_\eta u_b^0}{h_b^0+\overline{h^0}}\p_{tx}^{\alpha} \psi,\qquad
h^\alpha_\tau=\p_{tx}^{\alpha}  h_b^1-\frac{\p_\eta h_b^0}{h_b^0+\overline{h^0}}\p_{tx}^{\alpha} \psi,
\end{align*}
and  from the equations (\ref{LU1}), (\ref{LM}) and (\ref{PE1}), we can derive the equations of $u^\alpha_\tau$ and $h^\alpha_\tau$, in which the terms involving $\partial_{tx}^\alpha (v_b^1,g_b^1)$ disappear. Therefore, it is possible to obtain the 
$L^2$-estimates of $(u^\alpha_\tau, h^\alpha_\tau)$. 
Then we obtain the desired estimates for $\p_{tx}^{\alpha} (u_b^1, h_b^1)$ by proving the equivalence of  
$L^2$-norm between $(u^\alpha_\tau, h^\alpha_\tau)$ and $\p_{tx}^{\alpha} (u_b^1, h_b^1)$, and close the whole energy estimates. Consequently, the well-posedness results of solution $(u_b^1, v_b^1, h_b^1, g_b^1)$ to the initial-boundary value problem (\ref{LU1})-(\ref{bd_b1}) are concluded as follows. 
\begin{prop}\label{PROP1.4}
Let $(u_e^0, v_e^0, h_e^0, g_e^0)$ and $(u_e^1, v_e^1, h_e^1, g_e^1)$ be solutions constructed in Propositions \ref{PROP1.1} and \ref{PROP1.3} respectively. Let $(u_b^0, v_b^0, h_b^0, g_b^0)$ be the boundary layer profile 
constructed in Proposition \ref{PROP2}. 
Then, there exist a positive time $0<T_{4}\leq T_3$ and a unique solution $(u_b^1, v_b^1, h_b^1, g_b^1)$ to the initial-boundary value problem (\ref{LU1})-(\ref{bd_b1}), such that for any $l\geq0,$
\begin{align}\label{est_main11}\begin{aligned}
	(u_b^1, h_b^1)(t,x,\eta)&\in\bigcap_{i=0}^{[m/4]-3}W^{i,\infty}\Big(0,T_{4};H_{l}^{[m/4]-3-i}(\Omega)\Big),\\
(v_b^1, g_b^1)(t,x,\eta)&\in\bigcap_{i=0}^{[m/4]-4}W^{i,\infty}\Big(0,T_{4};L^\infty\big(\mathbb{R}_{\eta,+};H^{[m/4]-4-i}(\mathbb{T}_x)\big)\Big),\\
(\partial_\eta v_b^1, \partial_\eta g_b^1)(t,x,\eta)&\in\bigcap_{i=0}^{[m/4]-4}W^{i,\infty}\big(0,T_{4};H_{l}^{[m/4]-4-i}(\Omega)\big).
\end{aligned}\end{align}
\end{prop}

\subsection{Construction of approximate solutions}\label{S5}

We are in a position to complete the construction of approximate solutions of problem \eqref{1.1}-\eqref{BCM}.
Indeed, based on the profiles given in the above five subsections, we can write down the approximate solutions $(\mathbf{u^{a}},\mathbf{H^{a}}, p^{a})(t,x,y)$ used in this paper.
\begin{align}\label{NAP}
\begin{aligned}
\left\{\begin{array}{ll}
(\mathbf{u^{a}},\mathbf{H^{a}})(t,x,y)=&(\mathbf{u^0},\mathbf{H^0})(t,x,y)+\big(u^0_b, \sqrt{\epsilon} v_b^0, h^0_b, \sqrt{\epsilon} g_b^0\big)\big(t,x,\frac{y}{\sqrt{\epsilon}}\big)\\
&+\sqrt{\epsilon}\Big[(\mathbf{u^1},\mathbf{H^1})(t,x,y)+\chi(y)\big(u^1_b, \sqrt{\epsilon} v_b^1, h^1_b, \sqrt{\epsilon} g_b^1\big)\big(t,x,\frac{y}{\sqrt{\epsilon}}\big)\Big]\\
&+\epsilon\Big(\chi'(y)\int_0^{\frac{y}{\sqrt{\epsilon}}} u_b^1(t,x,\tilde{\eta})d\tilde{\eta},~0,~\chi'(y)\int_0^{\frac{y}{\sqrt{\epsilon}}} h_b^1(t,x,\tilde{\eta})d\tilde{\eta}+\rho(t,x,\frac{y}{\sqrt{\epsilon}}),~-\sqrt{\epsilon}\int_0^{\frac{y}{\sqrt{\epsilon}}}\p_x\rho(t,x,\tilde{\eta})d\tilde{\eta}\Big),\\
\hspace{.36in} p^{a}(t,x,y)=&p^0(t,x,y)+\sqrt{\epsilon}p^1(t,x,y)+\epsilon p^1_b(t,x,\frac{y}{\sqrt{\epsilon}}).
\end{array}\right.
\end{aligned}
\end{align}
Here, $\mathbf{u^a}=(u^{a}, v^{a}), \mathbf{H^a}=(h^{a}, g^{a})$, the smooth cut-off function $\chi(\cdot)$ is defined as follows
\begin{align}
\chi(y)=\left\{
\begin{array}{ll}
1,\quad y\in[0,1]\\
0,\quad y\in[2,+\infty)
\end{array}
\right.
\end{align}
and the boundary corrector $\rho(t,x,\eta)$ is a smooth function with compact support, which is chosen to satisfy the following two conditions:
\begin{align}\label{BBCC}
\rho(0,x,\eta)\equiv 0,\quad \p_\eta\rho(t,x,0)=-\p_y h^1(t,x,0).
\end{align}
It is noted that such function $\rho(t,x,\eta)$ exists since the two conditions in \eqref{BBCC} are compatible from \eqref{S1}, e.g., we can choose $\rho$ as
\begin{align*}
	\rho(t,x,\eta)~=~-\p_y h^1(t,x,0)\cdot\eta\chi(\eta),
\end{align*}
therefore, we can expect  from Proposition \ref{PROP1.3} that,
\begin{align}\label{rho}
	\rho(t,x,\eta)\in\bigcap_{j=0}^{[m/2]-4}C^j\Big([0,T_3]; H_l^{[m/2]-4-j}\big(\Omega\big)\Big), \quad \forall~l\geq0.
\end{align}
\begin{rem}
We introduce the cut-off function $\chi$ such that the normal components $v^a$ and $g^a$ 
 decay rapidly when $y\geq2$ as $\epsilon\rightarrow0$, see Remark \ref{rem}. Along with it are the additional terms involving $\chi'(y)$ in the tangential components $u^a$ and $h^a$, to ensure the divergence free conditions. On the other hand, the boundary corrector $\rho$ is used to cancel the boundary value of $\p_y h^1$ on $\{y=0\}$, so that it still holds that $\partial_y h^a|_{y=0}=0$ for the approximation \eqref{NAP}.
\end{rem}

Firstly, direct calculation reads that $\mathbf{u^a}=(u^{a}, v^{a})$ and $\mathbf{H^a}=(h^{a}, g^{a})$ satisfy the divergence free conditions:
\begin{align*}
\nabla\cdot\mathbf{u^a}=0,\quad \nabla\cdot\mathbf{H^a}=0,	
\end{align*}
 and the following initial-boundary conditions:
 \begin{align*}
 	(\mathbf{u^a}, \mathbf{H^a})|_{t=0}=(\mathbf{u_0},\mathbf{H_0})(x,y),\qquad \mathbf{u^a}|_{y=0}=\mathbf{0}, \quad (\partial_y h^a, g^a)|_{y=0}=\mathbf0.
 \end{align*}

Based on the construction in above Subsections \ref{S1}-\ref{S4}, we find that the approximate solution  $(\mathbf{u^{a}}, \mathbf{H^{a}}, p^{a})$ in (\ref{NAP}) solves the incompressible viscous MHD equations (\ref{1.1}) with some high order error terms with respect to the small parameter $\epsilon$. More precisely,
\begin{align}\label{EA}
\left\{\begin{array}{ll}
\p_t \mathbf{u^{a}}+(\mathbf{u^{a}}\cdot\nabla)\mathbf{u^{a}}+\nabla p^{a}-(\mathbf{H^{a}}\cdot\nabla)\mathbf{H^{a}}-\mu\epsilon\triangle \mathbf{u^{a}}=\mathbf{R_u},\\
\p_t \mathbf{H^{a}}+(\mathbf{u^{a}}\cdot\nabla)\mathbf{H^{a}}-(\mathbf{H^{a}}\cdot\nabla)\mathbf{u^{a}}-\kappa\epsilon\triangle \mathbf{H^{a}}=\mathbf{R_H},\\
\nabla\cdot\mathbf{u^{a}}=0,\quad\nabla\cdot\mathbf{H^{a}}=0,\\
\big(\mathbf{u^{a}},\mathbf{H^{a}}\big)|_{t=0}=(\mathbf{u_0}, \mathbf{H_0})(x,y),\\
\mathbf{u^{a}}|_{y=0}=\mathbf{0}, ~(\partial_y h^{a}, g^{a})|_{y=0}=\mathbf{0}.
\end{array}
\right.
\end{align}
The expressions of the error terms $\mathbf{R_u}=(R_1, R_2)$ and $\mathbf{R_H}=(R_3, R_4)$ caused by the approximation will be given in Appendix.
And for $R_i (i=1\sim4)$ we have the following result, and its proof is also in Appendix A.
\begin{prop}
\label{PROP1.5} Let the approximate solutions $(\mathbf{u^{a}}, \mathbf{H^{a}}, p^{a})$ established in (\ref{NAP}), then the error terms $R_i (i=1\sim4)$ in \eqref{EA} satisfy the following estimates:
\begin{align}\label{ERED}
\big\|\p^\alpha_{tx}R_i(t,\cdot)\big\|_{L^2}~\leq~ C\epsilon,\quad \alpha \in\mathbb{N}^2,~ |\alpha|\leq 3,~t\in[0,T_4],
\end{align}
for some positive constant $C$ independent of $\epsilon.$
\end{prop}

\section{Estimates of the remainder and proof of the main theorem}

Recall that in the above section, we have constructed the approximate solution  $(\mathbf{u^{a}}, \mathbf{H^{a}}, p^{a})$ in (\ref{NAP}), which satisfies the problem \eqref{EA}.
Let $(\mathbf{u^\epsilon}, \mathbf{H^\epsilon}, p^\epsilon)(t,x,y)$ be the solution of the original viscous MHD problem \eqref{1.1}-\eqref{BCM}, and set
\begin{align}\label{def_err}
(\mathbf{u^\epsilon}, \mathbf{H^\epsilon}, p^\epsilon)(t,x,y)=(\mathbf{u^{a}}, \mathbf{H^{a}}, p^{a})(t,x,y)
+\epsilon(\mathbf{u}, \mathbf{H}, p)(t,x,y)
\end{align}
with $\mathbf{u}=(u,v)$ and $\mathbf{H}=(h,g)$.
Then applying \eqref{EA} and \eqref{def_err} in the problem \eqref{1.1}-\eqref{BCM}, we derive the initial-boundary value problem for the remainder $(\mathbf{u}, \mathbf{H}, p)(t,x,y)$:
\begin{align}\label{ER}
\left\{\begin{array}{ll}
\p_t \mathbf{u}+(\mathbf{u^{\epsilon}}\cdot\nabla)\mathbf{u}-(\mathbf{H^{\epsilon}}\cdot\nabla)\mathbf{H}+\nabla p+(\mathbf{u}\cdot\nabla)\mathbf{u}^{a}-(\mathbf{H}\cdot\nabla)\mathbf{H}^{a}-\mu\epsilon\triangle \mathbf{u^{a}}=\mathbf{r_u},\\
\p_t \mathbf{H}+(\mathbf{u^{\epsilon}}\cdot\nabla)\mathbf{H}-(\mathbf{H^{\epsilon}}\cdot\nabla)\mathbf{u}+(\mathbf{u}\cdot\nabla)\mathbf{H}^a-(\mathbf{H}\cdot\nabla)\mathbf{u}^a-\kappa\epsilon\triangle \mathbf{H}=\mathbf{r_H},\\
\nabla\cdot\mathbf{u}=0,\quad\nabla\cdot\mathbf{H}=0,\\
\big(\mathbf{u},\mathbf{H}\big)|_{t=0}=\mathbf{0},\quad
\mathbf{u}|_{y=0}=\mathbf{0}, ~(\partial_y h, g)|_{y=0}=\mathbf{0}.
\end{array}\right.
\end{align}
where $r_i^\epsilon=\epsilon^{-1}R_i , i=1\sim4$ with $R_i$ given by \eqref{EA}. Moreover, from Proposition \ref{PROP1.5} we can achieve that
\begin{equation}
	\label{est_err}
	\|\partial_{tx}^\alpha r_i^\epsilon(t,\cdot)\|_{L^2}\leq C,\quad |\alpha|\leq3,~i=1\sim 4,~t\in[0,T_4]
\end{equation}
for some positive constant $C$ independent of $\epsilon$.

The key difficulty in the analysis for problem \eqref{ER} in Sobolev spaces is to control the strong interaction in the
boundary layer of thickness $O(\epsilon^{\frac{1}{2}})$ between
 the vorticity generated from the boundary layer and the remainder terms, even for short time (but independent of $\epsilon$). Precisely, it arises in the first and third equations of \eqref{ER} for the tangential components $(u,h)$, and exhibits the following forms: 
\begin{equation}\label{bad}
\begin{cases}
	v\partial_y u^a-g\partial_y h^a=\epsilon^{-\frac{1}{2}}v\partial_\eta u_b^0-\epsilon^{-\frac{1}{2}}g\partial_\eta h_b^0+O(1),\\
	v\partial_y h^a-g\partial_y u^a=\epsilon^{-\frac{1}{2}}v\partial_\eta h_b^0-\epsilon^{-\frac{1}{2}}g\partial_\eta u_b^0+O(1).
\end{cases}\end{equation}
Then by the energy method,
\begin{equation*}
\begin{cases}
\big|\int_{\mathbb{T}\times\mathbb{R}_+}u\cdot(v\partial_y u^a-g\partial_y h^a)dxdy\big|\leq O(1)	\epsilon^{-\frac{1}{2}}\|(u,v,g)\|_{L^2}^2,\\
\big|\int_{\mathbb{T}\times\mathbb{R}_+}h\cdot(v\partial_yh^a-g\partial_yu^a)dxdy\big|\leq O(1)	\epsilon^{-\frac{1}{2}}\|(h,v,g)\|_{L^2}^2,
\end{cases}
\end{equation*}
which prevents us from attaining the uniform estimates in $\epsilon.$
To overcome this difficulty, we will apply the idea used in 
\cite{LXY}
to take care of the cancellations between some physical terms
according to the  structure of the system.

\subsection{Key transformation}

By the divergence free condition
$\nabla\cdot\mathbf{H}=\partial_x h+\partial_y g=0,$
there exists a stream function $\psi(t,x,y)$, such that
\begin{align}\label{def_psi}
h=\p_y\psi,\quad g=-\p_x\psi,\quad \psi|_{y=0}=0,~\psi|_{t=0}=0,
\end{align}
and $\psi$ satisfies
\begin{align}\label{eq_psi}
\partial_t \psi+(u^\epsilon\partial_x+v^\epsilon\partial_y )\psi-g^a u+h^a v-\kappa\epsilon \triangle\psi=\partial_y^{-1}r_3^\epsilon\triangleq r_5^\epsilon.
\end{align}

Recall the boundary layer profiles given in \eqref{bl_def} and \eqref{LB01}:
\begin{align}\label{uh-p}
 (u^p, h^p)(t,x,\eta)=(u^0, h^0)(t,x,0)+(u_b^0, h_b^0)(t,x,\eta)=(\overline{u^0}, \overline{h^0})(t,x)+(u_b^0, h_b^0)(t,x,\eta),
 \end{align}
  and take the positive condition \eqref{positive-hp} for $h^p$ into account.
Note that there is a term $h^a v$ in the equation \eqref{eq_psi} of $\psi$ that has the form: 
  \begin{align}\label{good}
  	h^a v~=~h^p v+(h^0-\overline{h^0})v+O(\sqrt{\epsilon}).
  \end{align}
  Inspired by the idea of \cite{LXY} and comparing \eqref{bad} with \eqref{good}, we introduce the new quantities to replace the tangential components $(u,h)$:
  \begin{align}\label{tr1}
  \begin{aligned}
   \displaystyle 	\hat{u}(t,x,y):=u(t,x,y)-\frac{\partial_y u^p\big(t,x,\frac{y}{\sqrt{\epsilon}}\big)}{h^p\big(t,x,\frac{y}{\sqrt{\epsilon}}\big)}\psi(t,x,y)=u(t,x,y)-\frac{\epsilon^{-\frac{1}{2}} \partial_\eta u_0^b\big(t,x,\frac{y}{\sqrt{\epsilon}}\big)}{h^p\big(t,x,\frac{y}{\sqrt{\epsilon}}\big)}\psi(t,x,y),\\
    \displaystyle	 \hat{h}(t,x,y):=h(t,x,y)-\frac{\partial_y h^p\big(t,x,\frac{y}{\sqrt{\epsilon}}\big)}{h^p\big(t,x,\frac{y}{\sqrt{\epsilon}}\big)}\psi(t,x,y)=h(t,x,y)-\frac{\epsilon^{-\frac{1}{2}} \partial_\eta h_0^b\big(t,x,\frac{y}{\sqrt{\epsilon}}\big)}{h^p\big(t,x,\frac{y}{\sqrt{\epsilon}}\big)}\psi(t,x,y).
  \end{aligned}\end{align}
  Consequently, in the equations of new unknown $\hat{u}$ and $\hat{h}$, the large terms $\epsilon^{-\frac{1}{2}}v\partial_\eta u_b^0$ and $\epsilon^{-\frac{1}{2}}v\partial_\eta h_b^0$ in \eqref{bad} are directly cancelled, 
  respectively. On the other hand, for the large terms $\epsilon^{-\frac{1}{2}}g\partial_\eta h_b^0$ and $\epsilon^{-\frac{1}{2}}g\partial_\eta u_b^0$ in \eqref{bad}  produced by $g\partial_y h^a$ and $g\partial_y u^a$ respectively, we utilize another important observation in \cite{LXY} to eliminate such terms by the mixed terms. More precisely, it implies that by virtue of \eqref{def_psi} and \eqref{tr1},
  \begin{align*}
  	-h^\epsilon\partial_x h-g\partial_y h^a&=-h^p\partial_x h-g\partial_y h^p+O(1)=-h^p\partial_x h+\partial_y h^p\partial_x\psi+O(1)\\
  	&=-h^p\partial_x\hat{h}-h^p\partial_x\big(\frac{\partial_y h^p}{h^p}\big)\psi+O(1)=-h^p\partial_x\hat{h}+O(1),
  	\end{align*}
  	and
  	\begin{align*}
  	-h^\epsilon\partial_x u-g\partial_y u^a&=-h^p\partial_x u-g\partial_y u^p+O(1)=-h^p\partial_x u+\partial_y u^p\partial_x\psi+O(1)\\
  	&=-h^p\partial_x\hat{u}-h^p\partial_x\big(\frac{\partial_y u^p}{h^p}\big)\psi+O(1)=-h^p\partial_x\hat{u}+O(1),
\end{align*}
 where we have used 
  \begin{align*}\begin{cases}
  \displaystyle h^p\partial_x\big(\frac{\partial_y h^p}{h^p}\big)\psi=y h^p\partial_x\big(\frac{\partial_y h^p}{h^p}\big)\cdot\frac{\psi}{y}\sim\eta h^p\partial_x\big(\frac{\partial_\eta h^p}{h^p}\big)\cdot h=O(1),\\
 \displaystyle h^p\partial_x\big(\frac{\partial_y u^p}{h^p}\big)\psi=y h^p\partial_x\big(\frac{\partial_y u^p}{h^p}\big)\cdot\frac{\psi}{y}\sim\eta h^p\partial_x\big(\frac{\partial_\eta u^p}{h^p}\big)\cdot h=O(1).
  \end{cases}\end{align*}
  Moreover, similar as the results in Lemma 3.4 in \cite{LXY}, we can show the almost equivalence in the
energy space between $(u,h)$ and $(\hat{u},\hat{h})$.
  
  Although we can use the transformation \eqref{tr1} to get over the difficulty induced by \eqref{bad}, it also creates new difficulties. Specifically, the new quantity $\hat{u}$ given in \eqref{tr1} will destroy the divergence free structure of velocity field, i.e., $\partial_x\hat{u}+\partial_y v\neq0$. Once it happens, to obtain the uniform estimates in 
  $\epsilon$ of the pressure $p$ becomes an extreme hard issue. It is worth noting that such difficulty doesn't exist in the boundary layer problem, since the pressure involved is a known function. Therefore, we need to find new transformation such that the new quantities should contain $\hat{u}$ and $\hat{h}$ in their tangential components, and preserve the divergence free condition for velocity field as well. 

Before we introduce the desired transformation, let us first define some notations:
\begin{align}\label{eta}
a^p(t,x,y)~:=~\chi(y)\frac{u^p\big(t,x,\frac{y}{\sqrt{\epsilon}}\big)}{h^p\big(t,x,\frac{y}{\sqrt{\epsilon}}\big)},\qquad 
b^p(t,x,y)~:=~\frac{\partial_y h^p\big(t,x,\frac{y}{\sqrt{\epsilon}}\big)}{h^p\big(t,x,\frac{y}{\sqrt{\epsilon}}\big)}=\epsilon^{-\frac{1}{2}}\frac{\partial_\eta h^p\big(t,x,\frac{y}{\sqrt{\epsilon}}\big)}{h^p\big(t,x,\frac{y}{\sqrt{\epsilon}}\big)},
\end{align}
with the cut-off function $\chi(y)\in C^\infty(\mathbb{R}_+), 0\leq\chi(y)\leq1$ with
\begin{align*}
	\chi(y)~=~\begin{cases}
	1,\quad 0\leq y\leq 1,\\0,\quad y\geq2.
	\end{cases}\end{align*}
From Proposition \ref{PROP2}, we know that for suitable $\alpha\in\mathbb{N}^2, j\in\mathbb{N}$ and any $l\geq0,$
\begin{align}\label{est_p0}
	(1+\eta)^{l+j}\partial_{tx}^\alpha\partial_\eta^j\Big(u^p(t,x,\eta)-\overline{u^0}(t,x),~h^p(t,x,\eta)-\overline{h^0}(t,x)\Big)~=~O(1),
\end{align}
therefore we impose the cut-off function $\chi(y)$  in the definition of $a^p(t,x,y)$ in \eqref{eta} such that for any $k\in\mathbb{R}_+$,
\begin{align}\label{est_eta}
	y^k\partial_y^j\partial_{tx}^\alpha a^p(t,x,y)~=~\begin{cases}O(1),\quad & j\leq k,\\O(\epsilon^{\frac{k-j}{2}}),\quad &j>k\end{cases}~.
	\end{align}
Also, from \eqref{est_p0} it is easy to get that for any $k\in\mathbb{R}_+$,
 \begin{align}\label{est_eta1}
	y^k\partial_y^j\partial_{tx}^\alpha b^p(t,x,y)~=~O(\epsilon^{\frac{k-j-1}{2}}).
	\end{align}
Moreover, the boundary conditions of $u^p$ and $h^p$ in \eqref{LB01} yield
\begin{align}\label{bd_eta}
	a^p(t,x,0)~=~b^p(t,x,0)~=~0.
\end{align}

Now, we introduce the following transformation, which satisfies the requests:
\begin{align}\label{trans}
&\tilde u(t,x,y)~:=~u(t,x,y)-\partial_y(a^p\cdot\psi)(t,x,y),\quad \tilde v(t,x,y)~:=~v(t,x,y)+\partial_x(a^p\cdot\psi)(t,x,y),\nonumber\\
	&\tilde h(t,x,y)~:=~h(t,x,y)-(b^p\cdot\psi)(t,x,y),\qquad \tilde g(t,x,y)~:=~g(t,x,y).
\end{align}
Comparing \eqref{tr1} with \eqref{trans} it implies that $\hat{u}$ is the major item of $\tilde u$, and $\hat{h}=\tilde{h}$. Also it is easy to show that the new velocity $(\tilde{u}, \tilde{v})$ is still divergence free. 
Combining the initial-boundary values of $(\mathbf{u}, \mathbf{H})$ and $\psi$, given in \eqref{ER} and  \eqref{def_psi} respectively, and using \eqref{bd_eta} it yields that the new unknown $(\tilde u,\tilde v, \tilde h, \tilde g)$ preserves the initial-boundary values:
\begin{align*}
(\tilde u,\tilde v, \tilde h, \tilde g)|_{t=0}~=~\mathbf0, \quad	(\tilde u,\tilde v,\partial_y \tilde h, \tilde g)|_{y=0}~=~\mathbf0.
\end{align*}
Moreover, the following lemma shows the original unknown $(\mathbf{u}, \mathbf{H})$  is dominated in $L^p (1<p\leq\infty)$ norm
by the newly defined functions $(\tilde u,\tilde v, \tilde h, \tilde g)$ given by \eqref{trans}.
\begin{lem}\label{lem_equ}
	There exists a positive constant $C$ independent of $\epsilon$, such that 
	\begin{align}\label{equ_norm}
	\|\partial_{tx}^\alpha(\mathbf{u}, \mathbf{H})(t,\cdot)\|_{L^p}+\|y^{-1}\partial_{tx}^\alpha\psi(t,\cdot)\|_{L^p}	
		\leq C\sum_{\beta\leq\alpha}\|\partial_{tx}^\beta U(t,\cdot)\|_{L^p},\quad |\alpha|\leq2,~1<p\leq\infty.
	\end{align}
\end{lem}
The proof is similar to the one of Lemma 3.4 in \cite{LXY}, and for sake of self-containedness, we will provide it in Appendix B. 

\subsection{The transformed problem and preliminaries}

Based on the transformation \eqref{trans},  denote by
\begin{align}\label{U}
	U(t,x,y)~:=~(\tilde u,\tilde v,\tilde h,\tilde g)^T(t,x,y),
\end{align}
then the problem \eqref{ER} can be reduced as follows:
\begin{align}\label{pr_main}
\begin{cases}
	\partial_t U+\textit{A}_1(U)\partial_x U+\textit{A}_2(U)\partial_y U+\textit{C}(U) U+\psi D+(p_x,p_y,0,0)^T-\epsilon \textit{B}\triangle U =E^\epsilon,\\
		\partial_x \tilde u+\partial_y \tilde v=0,\\
		(\tilde u,\tilde v,\partial_y \tilde h, \tilde g)|_{y=0}=\mathbf 0,\quad
		U|_{t=0}=\mathbf 0.
\end{cases}\end{align}
Here, the matrices $A_i(U), i=1,2$ have the forms:
\begin{align}\label{def_A}
	A_i(U)~=~A_i^a+\sqrt{\epsilon}A_i^p+\epsilon \tilde A_i(U),\quad i=1,2,
\end{align}
where
\begin{align*}
		&A_1^a=\begin{pmatrix}
		(u^a+a^ph^a)~I_{2\times2} & [(a^p)^2-1]h^a~I_{2\times2}\\
		-h^a ~I_{2\times2} & (u^a-a^p h^a) ~I_{2\times2}
		\end{pmatrix},\quad A_2^a=\begin{pmatrix}
		(v^a+a^p g^a)~I_{2\times2} & [(a^p)^2-1]g^a ~I_{2\times2}\\
		-g^a ~I_{2\times2} & (v^a-a^p g^a) ~I_{2\times2}
	\end{pmatrix},
\end{align*}
\begin{align*}
	&A_1^p=\begin{pmatrix}
		\mathbf{0}_{2\times2} & \begin{matrix}
-2\mu\sqrt{\epsilon}\partial_xa^p & (\mu-\kappa)\sqrt{\epsilon}(\partial_y a^p+\eta_0b^p) \\
		0 & -(3\mu-\kappa)\sqrt{\epsilon}\partial_xa^p\end{matrix}\\
		\mathbf{0}_{2\times2} & \mathbf{0}_{2\times2}	
		\end{pmatrix},\\
		& A_2^p=\begin{pmatrix}
		\mathbf{0}_{2\times2} & \begin{matrix}-(3\mu-\kappa)\sqrt{\epsilon}\partial_ya^p-(\mu-\kappa)\sqrt{\epsilon} a^pb^p & 0\\
		(\mu-\kappa)\sqrt{\epsilon}\partial_xa^p  & -2\mu\sqrt{\epsilon}\partial_ya^p\end{matrix}\\
		\mathbf{0}_{2\times2} & \mathbf{0}_{2\times2}
\end{pmatrix},
\end{align*}
and
\begin{align*}
		&\tilde A_1(U)=\begin{pmatrix}
		(u+a^ph)~I_{2\times2} & [(a^p)^2-1]h~I_{2\times2}\\
		-h ~I_{2\times2} & (u-a^p h) ~I_{2\times2}
		\end{pmatrix},\quad\tilde  A_2(U)=\begin{pmatrix}
		(v+a^p g)~I_{2\times2} & [(a^p)^2-1]g~I_{2\times2}\\
		-g ~I_{2\times2} & (v-a^p g) ~I_{2\times2}
	\end{pmatrix}.
\end{align*}
The matrix $C(U)$ and vector $D$ have the forms:
\begin{align}\label{def_C}
C(U)~=~C^a+\epsilon \tilde{C}(U), \quad  D~=~D^a+\epsilon \psi D^p,
\end{align}
and the expressions of $C^a, \tilde{C}(U), D^a$ and $D^p$ are complicated that are given in Appendix C.
Also, the matrix $B$ and vector $E^\epsilon=(E^\epsilon_i)_{1\leq i\leq 4}$ are given as follows:
\begin{align}\label{source}
	&B=\left(\begin{array}{ccc}
	\mu ~I_{2\times2} & (\mu-\kappa)a^p ~I_{2\times2}\\
	\mathbf{0}_{2\times2} & \kappa ~I_{2\times2}
	\end{array}\right),\quad E^\epsilon~=~\Big(r_1^\epsilon-\partial_y(a^p r_5^\epsilon),r_2^\epsilon+\partial_x(a^p r_5^\epsilon),r_3^\epsilon-b^p r_5^\epsilon,r_4^\epsilon\Big)^T.
\end{align}
Moreover, we have the following estimates for the above notations, which shows the difficulty in \eqref{bad} is absent in the new problem \eqref{pr_main} for $U$.  The proof will be given in Appendix C. 
\begin{prop}\label{prop_coe-est}
There is a constant $C>0$ independent of $\epsilon$ such that for $|\alpha|\leq2$ and $i=1,2,$
\begin{align}\label{est-coe}
\big\|\partial_{tx}^\alpha (A_i^a, A_i^p, B, C^a)(t,\cdot)\big\|_{L^\infty}+\big\|y\partial_{tx}^\alpha D^a(t,\cdot)\big\|_{L^\infty}+\big\|y^2\partial_{tx}^\alpha  D^p(t,\cdot)\big\|_{L^\infty}+\|\partial_{tx}^\alpha E^\epsilon(t,\cdot)\|_{L^2}~\leq~C,
\end{align}
 and 
\begin{align}\label{est_AU}
	\big\|\partial_{tx}^\alpha\tilde{A}_i(U)(t,\cdot)\big\|_{L^p}+\sqrt{\epsilon} 
	\big\|\partial_{tx}^\alpha\tilde{C}(U)(t,\cdot)\big\|_{L^p}\leq C\sum_{\beta\leq\alpha}\big\|\partial_{tx}^\beta U(t,\cdot)\big\|_{L^p},\quad 1<p\leq+\infty.
	\end{align}	
\end{prop}

Now, we will make some preliminaries
 for the problem \eqref{pr_main} of $U$. Set
\begin{align}\label{S}
	S~:=~diag\big(1,1,1-(a^p)^2,1-(a^p)^2\big),
\end{align}
then $SA_i^a, S\tilde A_i(U), i=1,2$ are symmetric, and
\begin{align*}
	&S \big(A_1^a+\epsilon\tilde A_1(U)\big)=\begin{pmatrix}
	\big(u^\epsilon+a^p h^\epsilon\big)~I_{2\times2} & [(a^p)^2-1]h^\epsilon~I_{2\times2}\\
	[(a^p)^2-1]h^\epsilon~I_{2\times2}
	& [1-(a^p)^2]\big(u^\epsilon-a^p h^\epsilon\big)~I_{2\times2}\end{pmatrix},\\
	 & S \big(A_2^a+\epsilon\tilde A_2(U)\big)=\begin{pmatrix}
\big(v^\epsilon+a^pg^\epsilon\big)~I_{2\times2} & [(a^p)^2-1]g^\epsilon ~I_{2\times2}\\
	 [(a^p)^2-1]g^\epsilon ~I_{2\times2}
	 &
	 [1-(a^p)^2]\big(v^\epsilon-a^p g^\epsilon\big) ~I_{2\times2}
	 \end{pmatrix}.
\end{align*}
Also,
\begin{align}\label{def-SB}
	SB~=~\left(\begin{array}{ccc}
		\mu ~I_{2\times2} & (\mu-\kappa)a^p ~I_{2\times2}\\
		0 & \kappa[1-(a^p)^2]~ I_{2\times2}
	\end{array}\right).
\end{align}
To ensure that the symmetrizer $S$ and the matrix $SB$ in \eqref{def-SB} are positive definite, we need to impose some restriction on the function $a^p$. Specifically speaking,
by using the local well-posedness results for problem \eqref{LB01} obtained in Proposition \ref{PROP2}, and combining with $a^p|_{t=0}=0$ by the initial data of \eqref{LB01}, we know that for any fixed $\delta>0$ sufficiently small, there
exists a $T_\delta:~0<T_\delta\leq T_4$ such that
\begin{equation}\label{UB0small}
\sup_{t\in[0,T_\delta]}\|a^p(t,\cdot)\|^2_{L^{\infty}}
\leq\frac{4(\mu-\delta)(\kappa-\delta)}{(\mu+\kappa)^2-4\delta\kappa}.
\end{equation}
Then, it is easy to check that under the condition \eqref{UB0small}, $SB$ given by \eqref{def-SB} is positive definite, i.e.,  for any vector $X=(x_1,x_2,x_3,x_4)^T\in \mathbb{R}^4,$
\begin{align}\label{SB}
	SBX\cdot X\geq \delta|X|^2.
\end{align}
Also, we have
\begin{align}\label{positive_S}
	1-(a^p)^2(t,x,y)~\geq~\frac{(\mu-\kappa)^2+4\delta(\mu-\delta)}{(\mu+\kappa)^2-4\delta\kappa}~\triangleq~c_\delta>0,\quad t\in[0,T_\delta],~(x,y)\in\mathbb{T}\times\mathbb{R}_+,
	\end{align}
which, along with \eqref{S} implies the positive definiteness of $S$.
\begin{rem}\label{rem-T}
	From the definition \eqref{eta} for $a^p$, the condition \eqref{UB0small} means that for the leading order boundary layer in the time interval $[0, T_\delta]$, the component of tangential velocity $u^p$ is dominated by the component of tangential magnetic field $h^p$. This represents in some sense the stabilizing effect of the magnetic field on the velocity field in the boundary layers.
	\end{rem}

Now, we are able to establish the crucial estimates of the solution $U$ to the problem \eqref{pr_main}, which will be given in the following two subsections.

\subsection{Energy estimates for $U$ and $U_x$}

This subsection is devoted to obtain the $L^2-$estimates of $U$ and $U_x$ for the problem \eqref{pr_main}.
\begin{prop}\label{prop_U}
For any fixed small $\delta>0$ such that \eqref{UB0small} holds,	there exists a $0<T_*\leq T_\delta$ and a unique solution $U(t,x,y)$ to $(\ref{pr_main})$ on $[0,T_*]$ satisfying the following estimate:
\begin{equation}\label{est_1}
\|U(t,\cdot)\|_{L^2}^2+\epsilon\| U_x(t,\cdot)\|^2_{L^2}+\epsilon\int_0^t\big(\|\nabla U(s,\cdot)\|^2_{L^2}+\epsilon\|\nabla U_x(s,\cdot)\|^2_{L^2}\big)ds\leq C,\quad\forall~ t\in[0,T_*]
\end{equation}
for some constant $C>0$ independent of $\epsilon$.
\end{prop}

\begin{proof}[\textbf{Proof.}]
The local existence and uniqueness of the solution $U$ to problem \eqref{pr_main}, in some time interval $[0, T]$ ($T$ may depends on $\epsilon$), follows from the standard well-posedness result, so we will only show the estimate \eqref{est_1} in the following two steps.

\indent\newline
 (1) \textbf{$L^2-$estimate for $U$.} Multiplying $\eqref{pr_main}_1$ by $S$ from the left and taking the inner product of the resulting equations and $U$, it follows that
\begin{align}\label{L2-0}\begin{aligned}
&\frac{d}{2dt}(SU, U)+\Big(SA_1(U)\partial_x U+SA_2(U)\partial_y U, ~U\Big)+\Big(S\big(C(U)U+\psi D\big)-\frac{1}{2} S_tU, ~U\Big)-\epsilon(SB\triangle U, ~U)\\
&=(SE^\epsilon, ~U).
\end{aligned} \end{align}
 Note that in the above equality we have used the facts $S(p_x, p_y, 0, 0)^T=(p_x, p_y, 0, 0)^T$ and
 \[\Big(S(p_x, p_y, 0, 0)^T, ~U\Big)=0,
 \]
which can be obtained by integration by parts, the divergence free condition $\partial_x \tilde u+\partial_y\tilde v=0$  and the boundary condition $\tilde v|_{y=0}=0.$

 Each term in \eqref{L2-0} can be treated as follows. First, combining \eqref{positive_S} with \eqref{S} yields that
 \begin{align}\label{L2-1}
 	(SU, U)~\geq~c_\delta\|U(t,\cdot)\|_{L^2}^2.
 \end{align}
From the expressions \eqref{def_A}, it reads
\begin{align}\label{L2-2-0}\begin{aligned}
	&\Big(SA_1(U)\partial_x U+SA_2(U)\partial_y U, ~U\Big)\\
	=&\Big(S\big(A_1^a+\epsilon\tilde A_1(U)\big)\partial_x U+S\big(A_2^a+\epsilon\tilde A_2(U)\big)\partial_y U, ~U\Big)+\sqrt{\epsilon}\Big(SA_1^p\partial_x U+SA_2^p\partial_y U, ~U\Big)\\
	\triangleq~ & I_1+I_2.
\end{aligned}\end{align}
Since $SA_i^a, S\tilde A_i(U),~i=1,2$ are symmetric, from the boundary conditions $SA_2^a|_{y=0}=0$ and $S\tilde A_2(U)|_{y=0}=0$, it yields that by integration by parts,
 \begin{align*}
 	I_1=-\frac{1}{2}\Big(\big[\partial_x\big(SA_1^a+\epsilon S\tilde A_1(U)\big)+\partial_y\big(SA_2^a+\epsilon S\tilde A_2(U)\big)\big]U, U\Big).
 	 \end{align*}
By the divergence free conditions $\partial_x u^a+\partial_y v^a=0$ and $\partial_x h^a+\partial_y g^a=0$, one has
\begin{align*}
 	&\partial_x\big(SA_1^a\big)+\partial_y\big(SA_2^a\big)\nonumber\\
 		=&\begin{pmatrix}
	\big(h^a\partial_x+ g^a\partial_y\big)a^p~I_{2\times2} & 2a^p\big(h^a\partial_x+ g^a\partial_y\big)a^p~I_{2\times2}\\
	2a^p\big(h^a\partial_x+ g^a\partial_y\big)a^p~I_{2\times2}
	& \big\{[(a^p)^2-1]\big(h^a\partial_x+ g^a\partial_y\big)a^p-2a^p\big[(u^a-a^p h^a)\partial_x+(v^a-a^p g^a)\partial_y\big]a^p\big\}~I_{2\times2}\end{pmatrix}.
 \end{align*}
Taking the estimates \eqref{est_eta} into account, and by using the facts 
 \begin{align*}
 	\|v^a\partial_y a^p\|_{L^\infty}\leq\|y\partial_ya^p\|_{L^\infty}\big\| \frac{v^a}{y}\big\|_{L^\infty}\lesssim \|y\partial_ya^p\|_{L^\infty}\|\partial_y v^a\|_{L^\infty}=O(1),
 \end{align*}
 and similarly,
 \begin{align*}
 	\|g^a\partial_y a^p\|_{L^\infty}=O(1),
 \end{align*}
it implies 
 \begin{align}\label{est_SAa}
 	\partial_x\big(SA_1^a\big)+\partial_y\big(SA_2^a\big)~=~O(1).
 \end{align}
On the other hand, the divergence free conditions $\partial_x u+\partial_y v=0$ and $\partial_x h+\partial_y g=0$ yields
 \begin{align*}
 &\partial_x\big(S\tilde A_1(U)\big)+\partial_y\big(S\tilde A_2(U)\big)\nonumber\\
 		=&\begin{pmatrix}
	\big(h\partial_x+ g\partial_y\big)a^p~I_{2\times2} & 2a^p\big(h\partial_x+ g\partial_y\big)a^p~I_{2\times2}\\
	2a^p\big(h\partial_x+ g\partial_y\big)a^p~I_{2\times2}
	& \big\{[(a^p)^2-1]\big(h\partial_x+ g\partial_y\big)a^p-2a^p\big[(u-a^p h)\partial_x+(v-a^p g)\partial_y\big]a^p\big\}~I_{2\times2}\end{pmatrix},
 \end{align*}
 and then, it follows that by \eqref{est_eta},
  \begin{align}\label{est_SAt}
 	\partial_x\big(S\tilde A_1(U)\big)+\partial_y\big(S\tilde A_2(U)\big)~=~O(\epsilon^{-\frac{1}{2}})(u,v,h,g).
 \end{align}
Thus, applying \eqref{est_SAa} and \eqref{est_SAt} in $I_1$ we obtain that
 \begin{align}\label{I1}
 	|I_1|\lesssim \|U(t,\cdot)\|^2_{L^2}+\epsilon^{\frac{1}{2}}\|(u,v,h,g)(t,\cdot)\|_{L^2}\|U(t,\cdot)\|_{L^4}^2.
 \end{align}
 From the Sobolev inequality and interpolation inequality, it holds
 \begin{align}\label{4-2}
 	\|U(t,\cdot)\|_{L^4}^2\lesssim\|U(t,\cdot)\|_{L^2}\|U(t,\cdot)\|_{H^1}\lesssim\|U(t,\cdot)\|_{L^2}\|\nabla U(t,\cdot)\|_{L^2}+\|U(t,\cdot)\|_{L^2}^2,
 	\end{align}
  then, applying \eqref{equ_norm} with $p=2$ and \eqref{4-2} to \eqref{I1} yields that
 \begin{align}\label{est_I1}\begin{aligned}
 	|I_1|\lesssim~&\|U(t,\cdot)\|^2_{L^2}+\epsilon^{\frac{1}{2}}\|U(t,\cdot)\|_{L^2}\big(\|U(t,\cdot)\|_{L^2}\|\nabla U(t,\cdot)\|_{L^2}+\|U(t,\cdot)\|_{L^2}^2\big)\\
 	\leq & \frac{\delta\epsilon}{16}\|\nabla U(t,\cdot)\|_{L^2}^2+C\big(1+\epsilon^{\frac{1}{2}}\|U(t,\cdot)\|_{L^2}+\|U(t,\cdot)\|_{L^2}^2\big)\|U(t,\cdot)\|_{L^2}^2.
\end{aligned} \end{align}
For the terms $I_2$, it is easy to obtain that by \eqref{est-coe},
 \begin{align}\label{est_I2}
 |I_2|\leq C\sqrt{\epsilon}\|\nabla U(t,\cdot)\|_{L^2}	\|U(t,\cdot)\|_{L^2}\leq & \frac{\delta\epsilon}{16}\|\nabla U(t,\cdot)\|_{L^2}^2+C\|U(t,\cdot)\|_{L^2}^2.
 \end{align}
Thus, plugging \eqref{est_I1} and \eqref{est_I2} into \eqref{L2-2-0} it follows
\begin{align}\label{L2-2}
	\big(SA_1(U)\partial_x U+SA_2(U)\partial_y U, U\big)	\leq&\frac{\delta\epsilon}{8}\|\nabla U(t,\cdot)\|_{L^2}^2+C\Big(1+\|U(t,\cdot)\|_{L^2}^2\Big)\|U(t,\cdot)\|_{L^2}^2.
\end{align}

Next, from the definitions \eqref{def_C} and \eqref{S}, it gives
\begin{align}\label{L2-3-0}
\Big(S\big(C(U)U+\psi D\big)-\frac{1}{2}S_t U, U\Big)=\Big(S\big(C^a U+\psi D^a\big)-\frac{1}{2}S_t U,  U\Big)+\epsilon\big(\tilde C(U)U+\psi^2 D^p, SU\big).
	\end{align}
Thanks to the estimates \eqref{equ_norm} and  \eqref{est-coe}, it follows that \begin{align}\label{L2-3-1}\begin{aligned}
	\Big|\Big(S\big(C^a U+\psi D^a\big)-\frac{1}{2}S_t U, U\Big)\Big|\leq~&\big\|SC^a-\frac{1}{2}S_t\big\|_{L^\infty}\|U(t,\cdot)\|_{L^2}^2+\|ySD^a\|_{L^\infty}\|y^{-1}\psi\|_{L^2}\|U(t,\cdot)\|_{L^2}\\
	\leq  ~&C\|U(t,\cdot)\|_{L^2}^2.
\end{aligned}\end{align}
For the second term on the right-hand side of \eqref{L2-3-0}, we use \eqref{equ_norm}, \eqref{est-coe} and \eqref{est_AU} with $p=4$, to obtain
\begin{align*}
\epsilon\big|\big(\tilde C(U)U+\psi^2 D^p, SU\big)\big|
	\leq~&\epsilon\|SU(t,\cdot)\|_{L^2}\Big(\|\tilde C(U)\|_{L^4}\|U(t,\cdot)\|_{L^4}+\|y^2D^p(t,\cdot)\|_{L^\infty}\|y^{-1}\psi(t,\cdot)\|_{L^4}^2\Big)\nonumber\\
	\lesssim~&\sqrt{\epsilon}\|U(t,\cdot)\|_{L^2}\|U(t,\cdot)\|_{L^4}^2+\epsilon\|U(t,\cdot)\|_{L^2}\|\tilde h(t,\cdot)\|_{L^4}^2\\
	\lesssim~&\sqrt{\epsilon}\|U(t,\cdot)\|_{L^2}\|U(t,\cdot)\|_{L^4}^2,
	\end{align*}
which implies that by virtue of \eqref{4-2},
\begin{align}\label{L2-3-2}\begin{aligned}
\epsilon\big|\big(\tilde C(U)U+\psi^2 D^p, SU\big)\big|
	\lesssim~&
	\sqrt{\epsilon}\|U(t,\cdot)\|_{L^2}\big(\|U(t,\cdot)\|_{L^2}\|\nabla U(t,\cdot)\|_{L^2}+\|U(t,\cdot)\|_{L^2}^2\big)\\
	\leq&\frac{\delta\epsilon}{8}\|\nabla U(t,\cdot)\|_{L^2}^2+C\big(\sqrt{\epsilon}\|U(t,\cdot)\|_{L^2}+\|U(t,\cdot)\|_{L^2}^2\big)\|U(t,\cdot)\|_{L^2}^2.
\end{aligned}\end{align}
 Substituting \eqref{L2-3-1} and \eqref{L2-3-2} into \eqref{L2-3-0} gives
 \begin{align}\label{L2-3}
 	\Big|\Big(S\big(C(U)U+\psi D\big)-\frac{1}{2}S_t U, U\Big)\Big|\leq&\frac{\delta\epsilon}{8}\|\nabla U(t,\cdot)\|_{L^2}^2+C\Big(1+\|U(t,\cdot)\|_{L^2}^2\Big)\|U(t,\cdot)\|_{L^2}^2.
 \end{align}

To estimate the term $-\epsilon(SB\triangle U,U),$ one has that by integration by parts and the boundary conditions given in \eqref{pr_main},
 \begin{align*}
 	-\epsilon(SB\triangle U,U)=\epsilon(SB\partial_x U, \partial_x U)+\epsilon(SB\partial_y U, \partial_y U)+\epsilon\big(\partial_x(SB)\partial_x U+\partial_y(SB)\partial_y U, U\big),
 \end{align*}
and note that $\partial_y(SB)=O(\epsilon^{-\frac{1}{2}}),$ it implies that by \eqref{est_eta} and \eqref{SB},
 \begin{align}\label{L2-4}
 	-\epsilon(SB\triangle U,U)\geq\delta\epsilon\|\nabla U(t,\cdot)\|_{L^2}^2-C\sqrt{\epsilon}\|\nabla U(t,\cdot)\|_{L^2}\|U(t,\cdot)\|_{L^2}\geq\frac{3\delta\epsilon}{4}\|\nabla U(t,\cdot)\|_{L^2}^2-C\|U(t,\cdot)\|_{L^2}^2.
 \end{align}
 Also, it is easy to obtain that
 \begin{align}\label{L2-5}
 	\big(SE^\epsilon, U\big)\leq\|E^\epsilon(t,\cdot)\|_{L^2}\|SU(t,\cdot)\|_{L^2}\leq \frac{1}{2} \|E^\epsilon(t,\cdot)\|_{L^2}^2+C\|U(t,\cdot)\|_{L^2}^2.
 \end{align}

 Now, plugging \eqref{L2-2}, \eqref{L2-3}, \eqref{L2-4} and \eqref{L2-5} into \eqref{L2-0}, we obtain that
 \begin{align}\label{L2-final}
 	\frac{d}{dt}(SU, U)+\delta\epsilon\|\nabla U(t,\cdot)\|_{L^2}^2\leq \|E^\epsilon(t,\cdot)\|_{L^2}^2+C\big(1+\|U(t,\cdot)\|_{L^2}^2\big)\|U(t,\cdot)\|_{L^2}^2,
 \end{align}
therefore, by using \eqref{est-coe} and \eqref{L2-1}, there exists a $0<T_*\leq T_\delta$ and a constant $C>0$ independent of $\epsilon$, such that for $t\in[0,T_*],$
  \begin{align}\label{L2-fin}
 	\|U(t,\cdot)\|_{L^2}^2+\epsilon\int_0^t\|\nabla U(s,\cdot)\|_{L^2}^2ds~\leq~ C.
 	 \end{align}
\\
  (2) \textbf{$L^2$-estimate for $U_x$.}
 From the problem \eqref{pr_main}, we know that $U_x$ satisfies the following initial-boundary value problem:
\begin{align}\label{pr_main-x}
\begin{cases}
	\partial_t U_x
	+\textit{A}_1(U)\partial_x U_x+\textit{A}_2(U)\partial_y U_x+\partial_x\textit{A}_1(U)\partial_x U+\partial_x\textit{A}_2(U)\partial_y U+\partial_x\big(\textit{C}(U) U+\psi D\big)\\
	\qquad+(p_{xx},p_{yx},0,0)^T-\epsilon \textit{B}\triangle U_x-\epsilon\partial_x B\triangle U =\partial_xE^\epsilon,\\
		\partial_x\tilde u_x+\partial_y\tilde v_x=0,\\
		(\tilde u_x,\tilde v_x,\partial_y\tilde h_x,\tilde g_x)|_{y=0}=\mathbf 0,\quad		U_x|_{t=0}=\mathbf 0.
\end{cases}\end{align}
   Multiplying $\eqref{pr_main-x}_1$ by $S$ from the left and taking the inner product of the resulting equations and $U_x$, it follows that
\begin{align}\label{x-0}
\begin{aligned}
&\frac{d}{2dt}(SU_x, U_x)+\Big(S\big[A_1(U)\partial_x U_x+A_2(U)\partial_y U_x\big]-\epsilon SB\triangle U_x, U_x\Big)+\big(S(p_{xx}, p_{yx},0,0)^T, U_x\big)\\
&+\big(S\partial_x(C(U)U+\psi D)-\frac{1}{2}S_t U_x, U_x\big)+\big(S\big[\partial_x A_1(U) U_x+\partial_x A_2(U) U_y\big], U_x\big)-\epsilon(S\partial_x B\triangle U, U_x)\\
&=(S\partial_x E^\epsilon, U_x).
\end{aligned} \end{align}

To deal with the above terms in \eqref{x-0}, firstly, by similar arguments as given in the above step for $L^2$-norm of $U$, we can obtain that
\begin{align}
	&(SU_x, U_x)~\geq~c_\delta\|U_x(t,\cdot)\|_{L^2}^2,\label{x-1}\\
	&\big|\big(SA_1(U)\partial_x U_x+SA_2(U)\partial_y U_x, U_x\big)\big|\leq\frac{\delta\epsilon}{8}\|\nabla U_x(t,\cdot)\|_{L^2}^2+C\big(1+\|U(t,\cdot)\|_{L^2}^2\big)\|U_x(t,\cdot)\|_{L^2}^2,\label{x-2-0}\\
	&-\epsilon(SB\triangle U_x,U_x)\geq\frac{3\delta\epsilon}{4}\|\nabla U_x(t,\cdot)\|_{L^2}^2-C\|U_x(t,\cdot)\|_{L^2}^2,\nonumber\\
	&\big(S(p_{xx}, p_{yx}, 0, 0)^T, U_x\big)=0,\quad (S\partial_x E^\epsilon, U_x)\leq \frac{1}{2}\|\partial_x E^\epsilon(t,\cdot)\|_{L^2}^2+C\|U_x(t,\cdot)\|_{L^2}^2.\label{x-2}
\end{align}

Next, we will estimate the other terms in \eqref{x-0}.
For the term $\big(S\partial_x(C(U)U+\psi D)-\frac{1}{2}S_t U_x, ~U_x\big)$, it follows that by \eqref{def_C},
\begin{align*}
	\Big(S\partial_x(C(U)U+\psi D)-\frac{1}{2}S_t U_x, ~U_x\Big)=\Big(S\partial_x(C^aU+\psi D^a)-\frac{1}{2}S_t U_x, ~U_x\Big)+\epsilon\Big(\partial_x(\tilde C(U)U+\psi^2 D^p), ~SU_x\Big),
\end{align*}
which implies that
\begin{align}\label{Ca}\begin{aligned}
	&\Big|\Big(S\partial_x(C(U)U+\psi D)-\frac{1}{2}\partial_t SU_x, ~U_x\Big)\Big|\\
	\leq~&C\|U_x(t,\cdot)\|_{L^2}^2+\big\|S\partial_x(C^aU+\psi D^a)-\frac{1}{2}\partial_t SU_x\big\|_{L^2}^2+\epsilon^2\big\|\partial_x(\tilde C(U)U+\psi^2 D^p)\big\|_{L^2}^2.
\end{aligned}\end{align}
It is easy to obtain that by virtue of \eqref{est-coe},
\begin{align}\label{Ca-x}\begin{aligned}
&\big\|S\partial_x(C^aU+\psi D^a)-\frac{1}{2}S_t U_x\big\|_{L^2}^2\\
\leq~&\big\|SC^a-\frac{1}{2}S_t\big\|_{L^\infty}^2\|U_x(t,\cdot)\|_{L^2}^2+\big\|S\partial_x C^a\big\|_{L^\infty}^2\|U(t,\cdot)\|_{L^2}^2+\|ySD^a\|_{L^\infty}\|y^{-1}\psi_x\|_{L^2}^2+\|yS\partial_x D^a\|_{L^\infty}^2\|y^{-1}\psi\|_{L^2}^2\\
\lesssim~&\|U_x(t,\cdot)\|_{L^2}^2+\|U(t,\cdot)\|_{L^2}^2+\big(\|\tilde h_x(t,\cdot)\|_{L^2}+\|\tilde h(t,\cdot)\|_{L^2}\big)^2+\|\tilde h(t,\cdot)\|_{L^2}^2\\
\lesssim~&\|U_x(t,\cdot)\|_{L^2}^2+\|U(t,\cdot)\|_{L^2}^2.
\end{aligned}	\end{align}
On the other hand, by using the estimate \eqref{est_AU} for $\tilde C(U)$ with $p=4$ it follows that
\begin{align*}
	\big\|\partial_x(\tilde C(U)U+\psi^2 D^p)\big\|_{L^2}=~&\big\|\partial_x\tilde C(U)U+\tilde C(U)U_x+2\psi\partial_x\psi D^p+\psi^2 \partial_x D^p\big\|_{L^2}\\
	\leq~&\big\|\partial_x\tilde C(U)\big\|_{L^4}\|U(t,\cdot)\|_{L^4}+\big\|\tilde C(U)\big\|_{L^4}\|U_x(t,\cdot)\|_{L^4}\\
	&+2\|y^2D^p\|_{L^\infty}\|y^{-1}\psi(t,\cdot)\|_{L^4}\|y^{-1}\partial_x\psi(t,\cdot)\|_{L^4}+\|y^2\partial_x D^p\|_{L^\infty}\|y^{-1}\psi(t,\cdot)\|_{L^4}^2\nonumber\\
	\lesssim~&\epsilon^{-\frac{1}{2}}\|U(t,\cdot)\|_{L^4}\big(\|U_x(t,\cdot)\|_{L^4}+\|U(t,\cdot)\|_{L^4}\big)+\|\tilde h(t,\cdot)\|_{L^4}\big(\|\tilde h_x(t,\cdot)\|_{L^4}+\|\tilde h(t,\cdot)\|_{L^4}\big)\\
	\lesssim~&\epsilon^{-\frac{1}{2}}\|U(t,\cdot)\|_{L^4}\big(\|U_x(t,\cdot)\|_{L^4}+\|U(t,\cdot)\|_{L^4}\big),
\end{align*}
and then, along with \eqref{4-2} we get
\begin{align}\label{tC-x}\begin{aligned}
	&\epsilon^2\big\|\partial_x(\tilde C(U)U+\psi^2 D^p)\big\|_{L^2}^2\\
	\lesssim~&\epsilon\|U(t,\cdot)\|_{L^2}\|U(t,\cdot)\|_{H^1}\Big(\|U_x(t,\cdot)\|_{L^2}\|U_x(t,\cdot)\|_{H^1}+\|U(t,\cdot)\|_{L^2}\|U(t,\cdot)\|_{H^1}\Big)\\
	\leq~&\frac{\delta\epsilon}{24}\|\nabla U_x(t,\cdot)\|_{L^2}^2+C\epsilon\Big(\|U(t,\cdot)\|_{L^2}\|U(t,\cdot)\|_{H^1}+\|U(t,\cdot)\|_{L^2}^2\|U(t,\cdot)\|_{H^1}^2\Big)\|U_x(t,\cdot)\|_{L^2}^2\\
	&+C\epsilon\|U(t,\cdot)\|_{L^2}^2\|U(t,\cdot)\|_{H^1}^2.
\end{aligned}\end{align}
Substituting \eqref{Ca-x} and \eqref{tC-x} into \eqref{Ca} yields that
\begin{align}\label{x-C}
\begin{aligned}
	&\big|\big(S\partial_x(C(U)U+\psi D)-\frac{1}{2}S_t U_x, U_x\big)\big|\\
	\leq~&\frac{\delta\epsilon}{24}\|\nabla U_x(t,\cdot)\|_{L^2}^2+C\Big(1+\epsilon\|U(t,\cdot)\|_{L^2}^2\|U(t,\cdot)\|_{H^1}^2\Big)\|U_x(t,\cdot)\|_{L^2}^2+C\big(1+\epsilon\|U(t,\cdot)\|_{H^1}^2\big)\|U(t,\cdot)\|_{L^2}^2.
\end{aligned}\end{align}

For the term
$\big(S\big[\partial_x A_1(U) U_x+\partial_x A_2(U) U_y\big], U_x\big)=\big(\big[\partial_x A_1(U) U_x+\partial_x A_2(U) U_y\big], SU_x\big),$
 we first get that from the definitions \eqref{def_A} of $A_1(U)$,
 \begin{align}\label{def_J}\begin{aligned}
	\partial_x A_1(U) U_x+\partial_x A_2(U) U_y=
	&\partial_x \big(A_1^a+\sqrt{\epsilon}A_1^p\big) U_x+\partial_x \big(A_2^a+\sqrt{\epsilon}A_2^p\big)  U_y\\
	&+\epsilon\big(\partial_x \tilde A_1(U) U_x+\partial_x\tilde A_2(U) U_y\big)\\
	\triangleq &J_1+J_2.
\end{aligned}\end{align}
Then, it follows from the estimate \eqref{est-coe} that, \begin{align}\label{est_J1}
	|\big(J_1, SU_x\big)|~\leq~&\|J_1\|_{L^2}\|SU_x(t,\cdot)\|_{L^2}~\lesssim~\|\nabla U(t,\cdot)\|_{L^2}\|U_x(t,\cdot)\|_{L^2}.
	\end{align}
On the other hand, we obtain that by virtue of \eqref{est_AU} with $p=4$,
\begin{align*}
	|\big(J_2, SU_x\big)|~\lesssim~&\epsilon\|SU_x(t,\cdot)\|_{L^4}\cdot\big(\big\|\partial_x\tilde A_1(U)\big\|_{L^4}+\big\|\partial_x\tilde A_2(U)\big\|_{L^4}\big)\|\nabla U(t,\cdot)\|_{L^2}\\
		\lesssim~&\epsilon\|U_x(t,\cdot)\|_{L^4}\cdot\big(\|U_x(t,\cdot)\|_{L^4}+\|U(t,\cdot)\|_{L^4}\big)\|\nabla U(t,\cdot)\|_{L^2}\\
	\lesssim~&\epsilon\big(\|U_x(t,\cdot)\|_{L^4}^2+\|U(t,\cdot)\|_{L^4}^2\big)\cdot\|\nabla U(t,\cdot)\|_{L^2},\end{align*}
which, along with \eqref{4-2}, implies that
\begin{align}\label{est_J2}\begin{aligned}
	|\big(J_2, SU_x\big)|\lesssim~&\epsilon\big(\|U_x(t,\cdot)\|_{L^2}\| U_x(t,\cdot)\|_{H^1}+\|U(t,\cdot)\|_{L^2}\| U(t,\cdot)\|_{H^1}\big)\|\nabla U(t,\cdot)\|_{L^2}\\
	\leq~&\frac{\delta\epsilon}{24}\|\nabla U_x(t,\cdot)\|_{L^2}^2+C\epsilon \big(\|\nabla U(t,\cdot)\|_{L^2}+\|\nabla U(t,\cdot)\|_{L^2}^2\big)\|U_x(t,\cdot)\|_{L^2}^2\\
	&+C\epsilon\|U(t,\cdot)\|_{L^2}\|U(t,\cdot)\|_{H^1}^2.
\end{aligned}\end{align}
Collecting \eqref{def_J}, \eqref{est_J1} and \eqref{est_J2}, we get that
 \begin{align}\label{x-A}\begin{aligned}
 	&\Big|\Big(S\big[\partial_x A_1(U) U_x+\partial_x A_2(U) U_y\big], ~U_x\big)\Big|\\
	\leq~&\frac{\delta\epsilon}{24}\|\nabla U_x(t,\cdot)\|_{L^2}^2+C\big(1+\epsilon\|\nabla  U(t,\cdot)\|_{L^2}^2\big)\|U_x(t,\cdot)\|_{L^2}^2+C\big(1+\epsilon\|U(t,\cdot)\|_{L^2}\big)\|U(t,\cdot)\|_{H^1}^2.
\end{aligned} \end{align}

It remains to control the term $-\epsilon(S\partial_x B\triangle U, U_x)$. By integration by parts and the boundary conditions, we get
\begin{align*}
	-\epsilon(S\partial_x B\triangle U, U_x)=&\epsilon(S\partial_x BU_x, U_{xx})+\epsilon(\partial_x(S\partial_x B)U_x, U_x)+\epsilon(S\partial_x B U_y, U_{xy})+\epsilon(\partial_y(S\partial_x B) U_y, U_x),
\end{align*}
and note that $\partial_y(S\partial_x B)=O(\epsilon^{-\frac{1}{2}})$, it implies that by virtue of \eqref{est-coe},
\begin{align}\label{x-B}\begin{aligned}
	\big|\epsilon(S\partial_x B\triangle U, ~U_x)\big|~\lesssim~&\epsilon\|\nabla U_x(t,\cdot)\|_{L^2}\|\nabla U(t,\cdot)\|_{L^2}+\sqrt{\epsilon}\|\nabla U(t,\cdot)\|_{L^2}\|U_x(t,\cdot)\|_{L^2}\\
	 \leq~&\frac{\delta\epsilon}{24}\|\nabla U_x(t,\cdot)\|_{L^2}^2+C\epsilon\|\nabla U(t,\cdot)\|_{L^2}^2+C\|U_x(t,\cdot)\|_{L^2}^2.
\end{aligned}\end{align}

 Finally, plugging \eqref{x-2-0}-\eqref{x-2}, \eqref{x-C}, \eqref{x-A} and \eqref{x-B} into \eqref{x-0}, we obtain that
 \begin{align*}
 	\frac{d}{dt}(SU_x, U_x)+\delta\epsilon\|\nabla U_x(t,\cdot)\|_{L^2}^2
 	\leq ~&\|\partial_xE^\epsilon(t,\cdot)\|_{L^2}^2 +C\big(1+\epsilon\| U(t,\cdot)\|_{L^2}^2\big)\| U(t,\cdot)\|_{H^1}^2\nonumber\\
 	&+C\big[1+\|U(t,\cdot)\|_{L^2}^2+\epsilon\big(1+\| U(t,\cdot)\|_{L^2}^2\big)\|U(t,\cdot)\|_{H^1}^2
 	\big]	\|U_x(t,\cdot)\|_{L^2}^2,
 \end{align*}
and then, it yields that by using \eqref{L2-fin},
\begin{align}\label{x-fin}
	\frac{d}{dt}(SU_x, U_x)+\delta\epsilon\|\nabla U_x(t,\cdot)\|_{L^2}^2
 	\leq ~&\|\partial_xE^\epsilon(t,\cdot)\|_{L^2}^2+C\big(1+\epsilon\|U(t,\cdot)\|_{H^1}^2\big)\|U_x(t,\cdot)\|_{L^2}^2+C\|U(t,\cdot)\|_{H^1}^2.
 	 	\end{align}
Applying Gronwall inequality to the above inequality, and  using \eqref{x-1} yields
\begin{align}\label{x-f}\begin{aligned}
	&\|U_x(t,\cdot)\|_{L^2}^2+\epsilon\int_0^t\|\nabla U_x(s,\cdot)\|_{L^2}^2ds\\
 	\leq ~&\Big(\int_0^t\|\partial_xE^\epsilon(s,\cdot)\|_{L^2}^2ds+C\int_0^t\|U(s,\cdot)\|_{H^1}^2ds\Big)\exp\big\{C\int_0^t\big(1+\epsilon\|U(s,\cdot)\|_{H^1}^2\big)ds\big\}\\
 \leq~ &C\epsilon^{-1},\qquad t\in[0,T_*],
\end{aligned}\end{align}
where we have used \eqref{est-coe} and \eqref{L2-fin} again in the second inequality.
Thus, we obtain \eqref{est_1} from \eqref{L2-fin} and \eqref{x-f}, and complete the proof.
\end{proof}

\subsection{Energy estimates for $U_t$ and $U_{tx}$}

In this subsection, we are concerned with the $L^2-$estimates of $U_t$ and $U_{tx}$ for the problem \eqref{pr_main}.	 To this end, first of all, 
from \eqref{pr_main} we know that $U_t$ satisfies the following initial-boundary value problem:
\begin{align}\label{pr_main-t}
\begin{cases}
	\partial_t U_t+\textit{A}_1(U)\partial_x U_t+\textit{A}_2(U)\partial_y U_t+\partial_t\textit{A}_1(U)\partial_x U+\partial_t\textit{A}_2(U)\partial_y U+\partial_t\big(\textit{C}(U) U\\
	\qquad+\psi D\big)+(p_{xt},p_{yt},0,0)^T
	-\epsilon \textit{B}\triangle U_t-\epsilon\partial_t B\triangle U =\partial_t E^\epsilon,\\
		\partial_x\tilde u_t+\partial_y\tilde v_t=0,\\
		(\tilde u_t,\tilde v_t,\partial_y\tilde h_t,\tilde g_t)|_{y=0}=\mathbf 0,\quad		U_t|_{t=0}=E^\epsilon(0,x,y)-(p_x,p_y,0,0)^T(0,x,y).
\end{cases}\end{align}
Note that the initial data of $U_t$ depends on the initial pressure $p|_{t=0}$, for which we do not have any estimates. Therefore, we need to deal with the initial data $U_t|_{t=0}$ first. Actually, one has the following result.
\begin{prop}
There exists a constant $C>0$ independent of $\epsilon$, such that
\begin{align}\label{est_ini}
	\|U_t(0,\cdot)\|_{L^2}+\|\partial_x U_t(0,\cdot)\|_{L^2}~\leq~ C.
	\end{align}
\end{prop}

\begin{proof}[\textbf{Proof.}]
Firstly, from the initial data of \eqref{pr_main-t} and the definition \eqref{source} of $E^\epsilon$, it is easy to obtain that for the last two components of $U_t$,
\begin{align*}
	(\tilde h_t, \tilde g_t)(0,x,y)=(E_3^\epsilon,E_4^\epsilon)(0,x,y)=(r_3^\epsilon-b^p\cdot\partial_y^{-1}r_3^\epsilon, ~r^\epsilon_4)(0,x,y),
\end{align*}
which implies that by virtue of \eqref{est-coe},
\begin{align}\label{est_ht}
	\|(\tilde h_t,\tilde g_t)(0,\cdot)\|_{L^2}+\|(\partial_x\tilde h_t,\partial_x\tilde g_t)(0,\cdot)\|_{L^2}~\leq~C.
\end{align}
Next, from \eqref{pr_main-t} it follows that for the first two component of $U_t$,
\begin{align}\label{ini_t}
	(\tilde u_t, \tilde v_t)(0,x,y)=(E_1^\epsilon,E_2^\epsilon)(0,x,y)-(p_x,p_y)(0,x,y),
	\end{align}	
and in order to estimate $(\tilde u_t, \tilde v_t)|_{t=0}$, we need to estimate $\nabla p|_{t=0}.$

Thanks to the divergence free condition $\partial_x\tilde u_t+\partial_y\tilde v_t=0$ and the boundary condition $\tilde v_t|_{y=0}=0$, from \eqref{ini_t} we obtain that $p|_{t=0}$ satisfies the following elliptic equation with the Neumann boundary condition, 
\begin{align*}\begin{cases}
	\triangle p(0,x,y)=(\partial_x E_1^\epsilon+\partial_y E_2^\epsilon)|_{t=0}=(\partial_x r_1^\epsilon+\partial_y r_2^\epsilon)|_{t=0},\\\
	 p_y(0,x,0)=E_2^\epsilon(0,x,0)=r_2^\epsilon(0,x,0).
\end{cases}
\end{align*}
Then, the standard elliptic theory yields that
\begin{align}\label{est_p}
	\|\nabla p|_{t=0}\|_{L^2}+\|\nabla p_x|_{t=0}\|_{L^2}~\leq~C\big(\|(r_1^\epsilon,r_2^\epsilon)|_{t=0}\|_{L^2}+\|\partial_x(r_1^\epsilon,r_2^\epsilon)|_{t=0}\|_{L^2} \big).
\end{align}
Combining \eqref{ini_t} with \eqref{est_p} and using \eqref{est-coe}, we know that there is a constant $C>0$ independent of $\epsilon$, such that
\begin{align}\label{est_ut}
	\|(\tilde u_t, \tilde v_t)(0,\cdot)\|_{L^2}+\|(\partial_x\tilde u_t,\partial_x\tilde v_t)(0,\cdot)\|_{L^2}~\leq~C.
\end{align}
Consequently, \eqref{est_ini} follows immediately from \eqref{est_ht} and \eqref{est_ut}.
\end{proof}

As the estimate on $U_t|_{t=0}$ has been obtained, we can obtain the following result for $U_t$.
\begin{prop}\label{prop_Ut}
Under the assumptions of Proposition \ref{prop_U}, it holds
\begin{equation}\label{est_2}
\epsilon\|U_t(t,\cdot)\|_{L^2}^2+\epsilon^2\| U_{tx}(t,\cdot)\|^2_{L^2}+\epsilon^2\int_0^t\big(\|\nabla U_t(s,\cdot)\|^2_{L^2}+\epsilon\|\nabla U_{tx}(s,\cdot)\|^2_{L^2}\big)ds\leq C,\quad\forall~ t\in[0,T_*]
\end{equation}
for some constant $C>0$ independent of $\epsilon$.
\end{prop}

\begin{proof}[\textbf{Proof.}]
	The desired estimate of $U_t$ can be obtained in a similar way as the one for $U_x$, given in the second step of Proposition \ref{prop_U}. Indeed, we can obtain
	 \begin{align*}
	\frac{d}{dt}(SU_t, U_t)+\delta\epsilon\|\nabla U_t(t,\cdot)\|_{L^2}^2
 	\leq~& \|\partial_t E^\epsilon(t,\cdot)\|_{L^2}^2+C\big(1+\epsilon\|U(t,\cdot)\|_{H^1}^2\big)\|U_t(t,\cdot)\|_{L^2}^2+C\|U(t,\cdot)\|_{H^1}^2,
 	\end{align*}
 and then, applying the Gronwall inequality to the above inequality, one deduces that
  \begin{align}\label{t-f}\begin{aligned}
	&\|U_t(t,\cdot)\|_{L^2}^2+\epsilon\int_0^t\|\nabla U_t(s,\cdot)\|_{L^2}^2ds\\
	\leq~&\Big(\|U_t(0,\cdot)\|_{L^2}^2+\int_0^t\|\partial_t E^\epsilon(s,\cdot)\|_{L^2}^2ds+C\int_0^t\|U(s,\cdot)\|_{H^1}^2ds\Big)\cdot\exp\big\{C\int_0^t\big(1+\epsilon\|U(s,\cdot)\|_{H^1}^2\big)ds\big\}\\	
	\leq~ &C\epsilon^{-1},\qquad t\in[0,T_*],
\end{aligned}\end{align}
where we have used \eqref{est-coe}, \eqref{L2-fin} and \eqref{est_ini} in the above second inequality.

It remains to obtain the estimate of $U_{tx}$. From \eqref{pr_main-t} we know that $U_{tx}$ satisfies the following initial-boundary value problem:
\begin{align}\label{pr_main-tx}
\begin{cases}
	\partial_t U_{tx}
	+\textit{A}_1(U)\partial_x U_{tx}+\textit{A}_2(U)\partial_y U_{tx}+(\partial_{tx}^2 p_{x}, \partial_{tx}^2p_{y},0,0)^T+\partial_{tx}^2\big(C(U)U+\psi D\big)-\epsilon B\triangle U_{tx} \\
	\qquad+\big[\partial_{tx}^2, \textit{A}_1(U)\partial_x +\textit{A}_2(U)\partial_y\big] U-\epsilon\big[\partial_{tx}^2, \textit{B}\big]\triangle U=\partial_{tx}^2 E^\epsilon,\\
		\partial_x\tilde u_{tx}+\partial_y\tilde v_{tx}=0,\\
		(\tilde u_{tx}, \tilde v_{tx},\partial_y\tilde h_{tx}, \tilde g_{tx})|_{y=0}=\mathbf 0,\quad	U_{tx}|_{t=0}=\partial_x E^\epsilon(0,x,y)-(p_{xx},p_{yx},0,0)^T(0,x,y),
\end{cases}\end{align}
where the notation $[\cdot,\cdot]$ stands for the commutator.

Multiplying $\eqref{pr_main-tx}_1$ by applying $S$ from the left and taking the inner product of the resulting equations and $U_{tx}$, it follows that
\begin{align}\label{tx-0}\begin{aligned}
&\frac{d}{2dt}(SU_{tx}, U_{tx})+\Big(S\big[A_1(U)\partial_x U_{tx}+A_2(U)\partial_y U_{tx}\big]-\epsilon SB\triangle U_{tx}, U_{tx}\Big)+\big(S\big(\partial_{tx}^2p_{x}, \partial_{tx}^2p_{y},0,0\big)^T, U_{tx}\Big)\\
&+\big(S\big[\partial_{tx}^2\big(C(U)U+\psi D\big)-\frac{1}{2}S_tU_{tx}\big], U_{tx}\big)+\big(S\big[\partial_{tx}^2, A_1(U) \partial_x+A_2(U) \partial_y\big]U, U_{tx}\big)\\
&-\epsilon\big(S[\partial_{tx}^2, B]\triangle U, U_{tx}\big)=\big(\partial_{tx}^2 E^\epsilon, U_{tx}\big).
\end{aligned} \end{align}

Now, each term in \eqref{tx-0} can be estimated as follows. Firstly, by a similar argument as given in the above step for $L^2$-norm of $U$, one can obtain that
\begin{align}\label{tx-1}
	&\big(SU_{tx}, U_{tx}\big)~\geq~c_\delta\|U_{tx}(t,\cdot)\|_{L^2}^2,
	\end{align}
and
\begin{align}
	&\big|\big(SA_1(U)\partial_x U_{tx}+SA_2(U)\partial_y U_{tx}, U_{tx}\big)\big|\leq\frac{\delta\epsilon}{8}\|\nabla U_{tx}(t,\cdot)\|_{L^2}^2+C\big(1+\|U(t,\cdot)\|_{L^2}^2\big)\|U_{tx}(t,\cdot)\|_{L^2}^2,\label{tx-2-0}\\
	&-\epsilon(SB\triangle U_{tx},U_{tx})\geq\frac{3\delta\epsilon}{4}\|\nabla U_{tx}(t,\cdot)\|_{L^2}^2-C\|U_{tx}(t,\cdot)\|_{L^2}^2,\nonumber\\
	&\big(S(\partial_{tx}^2p_{x}, \partial_{tx}^2p_{y}, 0, 0)^T, U_{tx}\big)=0,\quad (S\partial_{tx}^2 E^\epsilon, U_{tx})\leq \frac{1}{2}\|\partial_{tx}^2 E^\epsilon(t,\cdot)\|_{L^2}^2+C\|U_{tx}(t,\cdot)\|_{L^2}^2.\label{tx-2}
\end{align}

Next, we proceed to estimate the other terms in \eqref{tx-0}. By  \eqref{def_C}, it follows that for the term $\big(S\big[\partial_{tx}^2\big(C(U)U+\psi D\big)-\frac{1}{2}\partial_t S~U_{tx}\big], U_{tx}\big),$
\begin{align*}
	&\Big(S\big[\partial_{tx}^2\big(C(U)U+\psi D\big)-\frac{1}{2}\partial_t S~U_{tx}\big], ~U_{tx}\Big)\\
	=~&\Big(S\big[\partial_{tx}^2\big(C^a U+\psi D^a\big)-\frac{1}{2}\partial_t S~U_{tx}\big], ~U_{tx}\Big)+\epsilon\Big(\partial_{tx}^2\big(\tilde C(U)U+\psi^2 D^p\big), ~SU_{tx}\Big)\\
	\triangleq~&I_1+I_2.
\end{align*}
For $I_1$, note that
\begin{align*}
	\partial_{tx}^2\big(C^a U+\psi D^a\big)=C^a U_{tx}+\partial_{tx}^2\psi D^a+\partial_xC^a U_t+\partial_t\psi \partial_x D^a+\partial_t C^a U_x+\partial_x\psi \partial_t D^a+\partial_{tx}^2C^a U+\psi \partial_{tx}^2D^a,
\end{align*}
which, along with \eqref{est-coe} yields that
\begin{align*}
	\big\|\partial_{tx}^2\big(C^a U+\psi D^a\big)\big\|_{L^2}\lesssim~&\|U_{tx}(t,\cdot)\|_{L^2}+\|y^{-1}\partial_{tx}^2\psi(t,\cdot)\|_{L^2}+\|U_{t}(t,\cdot)\|_{L^2}+\|y^{-1}\partial_{t}\psi(t,\cdot)\|_{L^2}\\
	&+\|U_{x}(t,\cdot)\|_{L^2}+\|y^{-1}\partial_{x}\psi(t,\cdot)\|_{L^2}+\|U(t,\cdot)\|_{L^2}+\|y^{-1}\psi(t,\cdot)\|_{L^2}\\
	\lesssim~&\|U_{tx}(t,\cdot)\|_{L^2}+\|U_{t}(t,\cdot)\|_{L^2}+\|U_{x}(t,\cdot)\|_{L^2}+\|U(t,\cdot)\|_{L^2}\\
	&+\|\tilde h_{tx}(t,\cdot)\|_{L^2}+\|\tilde h_{t}(t,\cdot)\|_{L^2}+\|\tilde h_{x}(t,\cdot)\|_{L^2}+\|\tilde h(t,\cdot)\|_{L^2}\\
	\lesssim~&\|U_{tx}(t,\cdot)\|_{L^2}+\|U_{t}(t,\cdot)\|_{L^2}+\|U_{x}(t,\cdot)\|_{L^2}+\|U(t,\cdot)\|_{L^2}.
\end{align*}
Thus, we obtain
\begin{align}\label{Ca-I1}\begin{aligned}
	|I_1|\leq~&\big\|\partial_{tx}^2\big(C^a U+\psi D^a\big)\big\|_{L^2}\|SU_{tx}(t,\cdot)\|_{L^2}+\frac{1}{2}\|S_t\|_{L^\infty}\|U_{tx}(t,\cdot)\|_{L^2}^2\\
	\lesssim~&\|U_{tx}(t,\cdot)\|_{L^2}^2+\|U_{t}(t,\cdot)\|_{L^2}^2+\|U_{x}(t,\cdot)\|_{L^2}^2+\|U(t,\cdot)\|_{L^2}^2.
\end{aligned}\end{align}
Next, by integration by parts with respect to $x$, the term $I_2$ can be reduced as
\begin{align*}
	I_2=-\epsilon\Big(\partial_t\big(\tilde C(U)U+\psi^2 D^p\big), S\partial_x U_{tx}+S_xU_{tx}\Big),
\end{align*}
and then, 
\begin{align*}
	|I_2|\leq~&\frac{\delta\epsilon}{24}\|\partial_x U_{tx}(t,\cdot)\|_{L^2}^2+C\epsilon\big\|\partial_t\big(\tilde C(U)U+\psi^2 D^p\big)\big\|_{L^2}^2+C\epsilon\|U_{tx}(t,\cdot)\|_{L^2}^2.
\end{align*}
Similar to \eqref{tC-x}, we can obtain that
\begin{align*}
	\epsilon\big\|\partial_t(\tilde C(U)U+\psi^2 D^p)\big\|_{L^2}^2\lesssim~&\|U_t(t,\cdot)\|_{H^1}^2+\|U(t,\cdot)\|_{L^2}^2\|U(t,\cdot)\|_{H^1}^2\|U_t(t,\cdot)\|_{L^2}^2
 	+\|U(t,\cdot)\|_{L^2}^2\|U(t,\cdot)\|_{H^1}^2,
\end{align*}
 then combining the above two inequalities yields that
\begin{align}\label{Ca-I2}
	|I_2|\leq~&\frac{\delta\epsilon}{24}\|\partial_x U_{tx}(t,\cdot)\|_{L^2}^2+C\|U_t(t,\cdot)\|_{H^1}^2+C\|U(t,\cdot)\|_{L^2}^2\|U(t,\cdot)\|_{H^1}^2\big(1+\|U_t(t,\cdot)\|_{L^2}^2\big).
 \end{align}
Consequently, we get that by combining \eqref{Ca-I1} with \eqref{Ca-I2},
\begin{align}\label{tx-C}
\begin{aligned}
	&\big|\big(S\big[\partial_{tx}^2\big(C(U)U+\psi D\big)-\frac{1}{2}\partial_t S~U_{tx}\big], U_{tx}\big)\big|\\
	\leq~&\frac{\delta\epsilon}{24}\|\partial_x U_{tx}(t,\cdot)\|_{L^2}^2+C\|U_t(t,\cdot)\|_{H^1}^2+C\|U(t,\cdot)\|_{L^2}^2\|U(t,\cdot)\|_{H^1}^2\big(1+\|U_t(t,\cdot)\|_{L^2}^2\big)+C\|U(t,\cdot)\|_{H^1}^2.
\end{aligned}\end{align}

Next, we consider the term $\big(S\big[\partial_{tx}^2, A_1(U) \partial_x+A_2(U) \partial_y\big]U, U_{tx}\big)$. By \eqref{def_A} and direct calculation, one gets that
\begin{align}\label{A-I}\begin{aligned}
	\big(S\big[\partial_{tx}^2, A_1(U) \partial_x+A_2(U) \partial_y\big]U, U_{tx}\big)=~&\big(\big[\partial_{tx}^2, \big(A_1^a+\sqrt{\epsilon}A_1^p\big) \partial_x+\big(A_1^a+\sqrt{\epsilon}A_2^p\big) \partial_y\big]U, SU_{tx}\big)\\
	&+\epsilon\big(\big[\partial_{tx}^2, \tilde A_1(U) \partial_x+\tilde A_2(U) \partial_y\big]U, SU_{tx}\big)\\
	\triangleq~&I_3+I_4.
\end{aligned}\end{align}
From \eqref{est-coe}, it is easy to have
\begin{align}\label{A-I3}
\begin{aligned}
	|I_3|\lesssim~&\|SU_{tx}(t,\cdot)\|_{L^2}\Big(\|\nabla U_t(t,\cdot)\|_{L^2}+\|\nabla U_x(t,\cdot)\|_{L^2}+\|\nabla U(t,\cdot)\|_{L^2}\Big)\\
	\lesssim~&\|\nabla U_t(t,\cdot)\|_{L^2}^2+\|\nabla U_x(t,\cdot)\|_{L^2}^2+\|\nabla U(t,\cdot)\|_{L^2}^2.
\end{aligned}\end{align}
We know that $I_4$ in \eqref{A-I} reads:
\begin{align*}
	I_4=~\epsilon\Big(&\big[\partial_{x}\tilde A_1(U) \partial_{tx}^2+\partial_x\tilde A_2(U) \partial_{ty}^2\big]U+\big[\partial_{t}\tilde A_1(U) \partial_{x}^2+\partial_t\tilde A_2(U) \partial_{xy}^2\big]U\nonumber\\
	&+\big[\partial_{tx}^2\tilde A_1(U) \partial_{x}+\partial_{tx}^2\tilde A_2(U) \partial_{y}\big]U, SU_{tx}\Big),
\end{align*}
 which implies that by virtue of \eqref{est_AU},
 \begin{align*}
 	|I_4|\lesssim~&\epsilon\|SU_{tx}(t,\cdot)\|_{L^4}\Big\{\|\nabla U_t(t,\cdot)\|_{L^2}\cdot\sum_{i=1}^2\|\partial_x\tilde A_i(U)\|_{L^4}+\|\nabla U_x(t,\cdot)\|_{L^2}\cdot\sum_{i=1}^2\|\partial_t\tilde A_i(U)\|_{L^4}\nonumber\\
 	&\hspace{1in} +\|\nabla U(t,\cdot)\|_{L^2}\cdot\sum_{i=1}^2\|\partial_{tx}^2\tilde A_i(U)\|_{L^4}\Big\}\nonumber\\
 	\lesssim~&\|\nabla U_t(t,\cdot)\|_{L^2}^2+\|\nabla U_x(t,\cdot)\|_{L^2}^2+\epsilon^2\|U_{tx}(t,\cdot)\|_{L^4}^2\Big(\|U_x(t,\cdot)\|_{L^4}^2+\|U_t(t,\cdot)\|_{L^4}^2+\|U(t,\cdot)\|_{L^4}^2\Big)\nonumber\\
 	&+\epsilon\|\nabla U(t,\cdot)\|_{L^2}\|U_{tx}(t,\cdot)\|_{L^4}\Big(\|U_{tx}(t,\cdot)\|_{L^4}+\|U_t(t,\cdot)\|_{L^4}+\|U_x(t,\cdot)\|_{L^4}+\|U(t,\cdot)\|_{L^4}\Big)\nonumber\\
 	\leq~&C\big(\|\nabla U_t(t,\cdot)\|_{L^2}^2+\|\nabla U_x(t,\cdot)\|_{L^2}^2+\|\nabla U(t,\cdot)\|_{L^2}^2\big)+C\epsilon\|\nabla U(t,\cdot)\|_{L^2}\|U_{tx}(t,\cdot)\|_{L^4}^2\nonumber\\
 	&+C\epsilon^2\|U_{tx}(t,\cdot)\|_{L^4}^2\Big(\|U_x(t,\cdot)\|_{L^4}^2+\|U_t(t,\cdot)\|_{L^4}^2+\|U(t,\cdot)\|_{L^4}^2\Big).
 \end{align*}
Also, it follows that by \eqref{4-2}, 
\begin{align*}
	C\epsilon\|\nabla U(t,\cdot)\|_{L^2}\|U_{tx}(t,\cdot)\|_{L^4}^2\lesssim~&\epsilon\|\nabla U(t,\cdot)\|_{L^2}\|U_{tx}(t,\cdot)\|_{L^2}\|U_{tx}(t,\cdot)\|_{H^1}\nonumber\\
	\leq~&\frac{\delta\epsilon}{48}\|\nabla U_{tx}(t,\cdot)\|_{L^2}^2+C\big(1+\epsilon\|\nabla U(t,\cdot)\|_{L^2}^2\big)\| U_{tx}(t,\cdot)\|_{L^2}^2,
\end{align*}
and
\begin{align*}
	&C\epsilon^2\|U_{tx}(t,\cdot)\|_{L^4}^2\Big(\|U_x(t,\cdot)\|_{L^4}^2+\|U_t(t,\cdot)\|_{L^4}^2+\|U(t,\cdot)\|_{L^4}^2\Big)\nonumber\\
	\lesssim~&\epsilon^2\|U_{tx}(t,\cdot)\|_{L^2}\|U_{tx}(t,\cdot)\|_{H^1}\Big(\|U_{x}(t,\cdot)\|_{L^2}\|U_{x}(t,\cdot)\|_{H^1}+\|U_{t}(t,\cdot)\|_{L^2}\|U_{t}(t,\cdot)\|_{H^1}+\|U(t,\cdot)\|_{L^2}\|U(t,\cdot)\|_{H^1}\Big)\nonumber\\
	\leq~&\frac{\delta\epsilon}{48}\|\nabla U_{tx}(t,\cdot)\|_{L^2}^2+C\Big(1+\epsilon^3\|U_{x}(t,\cdot)\|_{L^2}^2\|U_{x}(t,\cdot)\|_{H^1}^2+\epsilon^3\|U_{t}(t,\cdot)\|_{L^2}^2\|U_{t}(t,\cdot)\|_{H^1}^2\nonumber\\
	&\hspace{1.4in} +\epsilon^3\|U(t,\cdot)\|_{L^2}^2\|U(t,\cdot)\|_{H^1}^2\Big)\| U_{tx}(t,\cdot)\|_{L^2}^2.
\end{align*}
Collecting the above three inequalities yields that
\begin{align}\label{A-I4}
\begin{aligned}
	|I_4|\leq~&\frac{\delta\epsilon}{24}\|\nabla U_{tx}(t,\cdot)\|_{L^2}^2+C\big(\|\nabla U_t(t,\cdot)\|_{L^2}^2+\|\nabla U_x(t,\cdot)\|_{L^2}^2+\|\nabla U(t,\cdot)\|_{L^2}^2\big)\\
	&+C\| U_{tx}(t,\cdot)\|_{L^2}^2\cdot\Big(1+\epsilon\|\nabla U(t,\cdot)\|_{L^2}^2+\epsilon^3\|U_{x}(t,\cdot)\|_{L^2}^2\|U_{x}(t,\cdot)\|_{H^1}^2\\
	&\hspace{1.2in} +\epsilon^3\|U_{t}(t,\cdot)\|_{L^2}^2\|U_{t}(t,\cdot)\|_{H^1}^2+\epsilon^3\|U(t,\cdot)\|_{L^2}^2\|U(t,\cdot)\|_{H^1}^2\Big).
\end{aligned}\end{align}
Then, substituting \eqref{A-I3} and \eqref{A-I4} into \eqref{A-I}, one has
\begin{align}\label{tx-A}
\begin{aligned}
	&\Big|\big(S\big[\partial_{tx}^2, A_1(U) \partial_x+A_2(U) \partial_y\big]U, U_{tx}\big)\Big|\\
	\leq~&\frac{\delta\epsilon}{24}\|\nabla U_{tx}(t,\cdot)\|_{L^2}^2+C\big(\|\nabla U_t(t,\cdot)\|_{L^2}^2+\|\nabla U_x(t,\cdot)\|_{L^2}^2+\|\nabla U(t,\cdot)\|_{L^2}^2\big)\\
	&+C\| U_{tx}(t,\cdot)\|_{L^2}^2\cdot\Big(1+\epsilon\|\nabla U(t,\cdot)\|_{L^2}^2+\epsilon^3\|U_{x}(t,\cdot)\|_{L^2}^2\|U_{x}(t,\cdot)\|_{H^1}^2\\
	&\hspace{1.2in} +\epsilon^3\|U_{t}(t,\cdot)\|_{L^2}^2\|U_{t}(t,\cdot)\|_{H^1}^2+\epsilon^3\|U(t,\cdot)\|_{L^2}^2\|U(t,\cdot)\|_{H^1}^2\Big).
\end{aligned}\end{align}

Now, let us estimate the term $-\epsilon\big(S[\partial_{tx}^2, B]\triangle U, U_{tx}\big)$. Firstly, it follows that by integration by parts,
\begin{align}\label{B-I}
\begin{aligned}
	&-\epsilon\big(S[\partial_{tx}^2, B]\triangle U, U_{tx}\big)\\
	=~&\epsilon\Big(S[\partial_{tx}^2, B]\partial_x U, \partial_x U_{tx}\Big)+\epsilon\Big(S[\partial_{tx}^2, B]\partial_y U, \partial_y U_{tx} \Big)\\
	&+\epsilon\Big(S[\partial_{tx}^2, B_x]\partial_x U, U_{tx}\Big)+\epsilon\Big(S_x[\partial_{tx}^2, B]\partial_x U,  U_{tx}\Big)+\epsilon\Big(S[\partial_{tx}^2, B_y]\partial_y U, U_{tx}\Big)+\epsilon\Big(S_y[\partial_{tx}^2, B]\partial_y U,  U_{tx} \Big)\\
	\triangleq~&I_5+I_6.
\end{aligned}\end{align}
It is easy to obtain by \eqref{est-coe} that
\begin{align}\label{B-I5}
\begin{aligned}
	|I_5|\leq~&\epsilon\|\nabla U_{tx}(t,\cdot)\|_{L^2}\Big(\big\|S[\partial_{tx}^2, B]\partial_x U\big\|_{L^2}+\big\|S[\partial_{tx}^2, B]\partial_y U\big\|_{L^2}\Big)\\
	\lesssim~&\epsilon\|\nabla U_{tx}(t,\cdot)\|_{L^2}\big(\|\nabla U_{t}(t,\cdot)\|_{L^2}+\|\nabla U_{x}(t,\cdot)\|_{L^2}+\|\nabla U(t,\cdot)\|_{L^2}\big)\\
	\leq~&\frac{\delta\epsilon}{24}\|\nabla U_{tx}(t,\cdot)\|_{L^2}^2+C\epsilon\big(\|\nabla U_t(t,\cdot)\|_{L^2}^2+\|\nabla U_x(t,\cdot)\|_{L^2}^2+\|\nabla U(t,\cdot)\|_{L^2}^2\big).
\end{aligned}\end{align}
Note that $S_y, B_y=O(\epsilon^{-\frac{1}{2}})$, then we obtain
\begin{align}\label{B-I6}
\begin{aligned}
	|I_6|	\lesssim~&\sqrt{\epsilon}\|U_{tx}(t,\cdot)\|_{L^2}\big(\|\nabla U_{t}(t,\cdot)\|_{L^2}+\|\nabla U_{x}(t,\cdot)\|_{L^2}+\|\nabla U(t,\cdot)\|_{L^2}\big)\\
	\leq~&C\|U_{tx}(t,\cdot)\|_{L^2}^2+C\epsilon\big(\|\nabla U_t(t,\cdot)\|_{L^2}^2+\|\nabla U_x(t,\cdot)\|_{L^2}^2+\|\nabla U(t,\cdot)\|_{L^2}^2\big).
\end{aligned}\end{align}
Thus, plugging \eqref{B-I5} and \eqref{B-I6} into \eqref{B-I} implies that
\begin{align}\label{tx-B}
	&\big|\epsilon\big(S[\partial_{tx}^2, B]\triangle U, U_{tx}\big)\big|\nonumber\\
	\leq~&\frac{\delta\epsilon}{24}\|\nabla U_{tx}(t,\cdot)\|_{L^2}^2+C\|U_{tx}(t,\cdot)\|_{L^2}^2+C\epsilon\big(\|\nabla U_t(t,\cdot)\|_{L^2}^2+\|\nabla U_x(t,\cdot)\|_{L^2}^2+\|\nabla U(t,\cdot)\|_{L^2}^2\big).
\end{align}

Finally, we substitute \eqref{tx-2-0}-\eqref{tx-2}, \eqref{tx-C}, \eqref{tx-A} and \eqref{tx-B} into \eqref{tx-0}, to obtain
\begin{align*}
	&\frac{d}{dt}(SU_{tx}, U_{tx})+\delta\epsilon\|\nabla U_{tx}(t,\cdot)\|_{L^2}^2\nonumber\\
 	\leq ~&\|\partial_{tx}^2E^\epsilon(t,\cdot)\|_{L^2}^2+C\big(\|U_t(t,\cdot)\|_{H^1}^2+\|U_x(t,\cdot)\|_{H^1}^2+\|U(t,\cdot)\|_{H^1}^2\big)+C\|U(t,\cdot)\|_{L^2}^2\|U(t,\cdot)\|_{H^1}^2\big(1+\|U_t(t,\cdot)\|_{L^2}^2\big)\nonumber\\
	&+C\| U_{tx}(t,\cdot)\|_{L^2}^2\cdot\Big(1+\|U(t,\cdot)\|_{L^2}^2+\epsilon\|\nabla U(t,\cdot)\|_{L^2}^2+\epsilon^3\|U_{x}(t,\cdot)\|_{L^2}^2\|U_{x}(t,\cdot)\|_{H^1}^2\nonumber\\
	&\hspace{1.2in} +\epsilon^3\|U_{t}(t,\cdot)\|_{L^2}^2\|U_{t}(t,\cdot)\|_{H^1}^2+\epsilon^3\|U(t,\cdot)\|_{L^2}^2\|U(t,\cdot)\|_{H^1}^2\Big),
\end{align*}
which, along with \eqref{est_1} and \eqref{t-f} implies that
\begin{align*}
	&\frac{d}{dt}(SU_{tx}, U_{tx})+\delta\epsilon\|\nabla U_{tx}(t,\cdot)\|_{L^2}^2\\
	\leq ~&\|\partial_{tx}^2E^\epsilon(t,\cdot)\|_{L^2}^2+C\big(\|U_t(t,\cdot)\|_{H^1}^2+\|U_x(t,\cdot)\|_{H^1}^2+\epsilon^{-1}\|U(t,\cdot)\|_{H^1}^2\big)\nonumber\\
	&+C\| U_{tx}(t,\cdot)\|_{L^2}^2\cdot\Big(1+\epsilon\|U(t,\cdot)\|_{H^1}^2+\epsilon^2\|U_{x}(t,\cdot)\|_{H^1}^2+\epsilon^2\|U_{t}(t,\cdot)\|_{H^1}^2\Big).
\end{align*}
Consequently, applying the Gronwall inequality to the above inequality, it holds that
  \begin{align}\label{tx-f}
  \begin{aligned}
	&\|U_{tx}(t,\cdot)\|_{L^2}^2+\epsilon\int_0^t\|\nabla U_{tx}(s,\cdot)\|_{L^2}^2ds\\
	\leq~&\Big(\|U_{tx}(0,\cdot)\|_{L^2}^2+\int_0^t\|\partial_{tx}^2 E^\epsilon(s,\cdot)\|_{L^2}^2ds+C\int_0^t\big(\|U_t(s,\cdot)\|_{H^1}^2+\|U_x(s,\cdot)\|_{H^1}^2+\epsilon^{-1}\|U(s,\cdot)\|_{H^1}^2\big)ds\Big)\\
	&\cdot\exp\big\{C\int_0^t\big(1+\epsilon\|U(s,\cdot)\|_{H^1}^2++\epsilon^2\|U_{x}(t,\cdot)\|_{H^1}^2+\epsilon^2\|U_{t}(t,\cdot)\|_{H^1}^2\big)ds\big\}\\	
	\leq~ &C\epsilon^{-2},\qquad t\in[0,T_*],
\end{aligned}\end{align}
where we have used \eqref{est-coe}, \eqref{est_1},  \eqref{est_ini} and \eqref{t-f} in the second inequality. Thus, from \eqref{t-f} and \eqref{tx-f} we obtain \eqref{est_2} and complete the proof.

\end{proof}

\subsection{Proof of the main theorem}

Now, we are ready to prove our main theorem.
\begin{proof}[\textbf{Proof of Theorem \ref{MAIN THM}.}]
From \eqref{def_err}, one has
\begin{align}\label{Def_err}
(\mathbf{u^\epsilon}, \mathbf{H^\epsilon}, p^\epsilon)=(\mathbf{u^{a}}, \mathbf{H^{a}}, p^a)
+\epsilon(\mathbf{u}, \mathbf{H}, p),
\end{align}
then combining with the expression \eqref{NAP} of the approximate solution $(\mathbf{u^a}, \mathbf{H^a}, p^a)$, the local existence in $[0,T_*]$ of the solution $(\mathbf{u^\epsilon}, \mathbf{H^\epsilon},  p^\epsilon)$ to the problem \eqref{1.1}-\eqref{BCM} follows from the local existence of $(\mathbf{u}, \mathbf{H}, p)$ given in Proposition \ref{prop_U}. Also, we know that from the expression \eqref{NAP} for $(\mathbf{u^a}, \mathbf{H^a}, p^a)$,
\begin{align}\label{L-inf-a}
	(\mathbf{u^a}, \mathbf{H^a})(t,x,y)=(\mathbf{u^0}, \mathbf{H^0})(t,x,y)+\big(u^0_b, \sqrt{\epsilon}v^0_b, h^0_b, \sqrt{\epsilon}g^0_b\big)\big(t,x,\frac{y}{\sqrt{\epsilon}}\big)+O(\sqrt{\epsilon}).
\end{align}
Therefore, combining \eqref{Def_err} with \eqref{L-inf-a} we only need to obtain the $L^\infty$-estimate of $(\mathbf{u}, \mathbf{H})$ to show \eqref{ASP}, and further, it turns out to get the $L^\infty$-estimate of $U$ by utilizing Lemma \ref{lem_equ}.

Next, it suffices to get the $L^\infty$-estimate of $U$ from the estimates \eqref{est_1} and \eqref{est_2}. Indeed, the Sobolev embedding inequality and interpolation inequality yields that for any small $\lambda>0,$
\begin{align}\label{U-0}
	\|U\|_{L^\infty_{txy}}~\lesssim~\|U\|_{L^\infty_{tx}L^2_y}^{\frac{1}{2}-\lambda}\cdot\|U\|_{L^\infty_{tx}H_y^1}^{\frac{1}{2}+\lambda}.
\end{align}
It also holds that by combining with \eqref{est_1}, 
\begin{align}\label{U-1}
	\|U\|_{L^\infty_{tx}L^2_y}~\lesssim~\|U\|_{L^\infty_{t}L^2_{xy}}^{\frac{1}{2}-\lambda}\cdot\|U\|_{L^\infty_{t}H_x^1L_y^2}^{\frac{1}{2}+\lambda}~\lesssim~\epsilon^{-\frac{1}{2}(\frac{1}{2}+\lambda)}.
\end{align}
Similarly, one has by virtue of \eqref{est_2} that
\begin{align}\label{U-2}
\begin{aligned}
	\|U\|_{L^\infty_{tx}H^1_y}~\lesssim~&\|U\|_{L_t^2L^\infty_{x}H^1_{y}}^{\frac{1}{2}-\lambda}\cdot\|U\|_{H^1_{t}L_x^\infty H_y^1}^{\frac{1}{2}+\lambda}\\
	\lesssim~&\|U\|_{L_{tx}^2H_y^1}^{(\frac{1}{2}-\lambda)^2}\|U\|_{L_{t}^2H_x^1H_y^1}^{(\frac{1}{2}-\lambda)(\frac{1}{2}+\lambda)}\cdot\|U\|_{H_t^1L_{x}^2H_y^1}^{(\frac{1}{2}+\lambda)(\frac{1}{2}-\lambda)}\|U\|_{H_{t}^1H_x^1H^1_{y}}^{(\frac{1}{2}+\lambda)^2}\\
	\lesssim~&\epsilon^{-\frac{1}{2}(\frac{1}{2}-\lambda)^2-(\frac{1}{2}-\lambda)(\frac{1}{2}+\lambda)-(\frac{1}{2}-\lambda)(\frac{1}{2}+\lambda)-\frac{3}{2}(\frac{1}{2}+\lambda)^2}\\
	\lesssim~&\epsilon^{-1-\lambda}.
\end{aligned}\end{align}
Substituting \eqref{U-1} and \eqref{U-2} into \eqref{U-0}, we have
\begin{align}\label{U_infty}
	\|U\|_{L^\infty_{txy}}~\lesssim~\epsilon^{-\frac{1}{2}(\frac{1}{2}+\lambda)(\frac{1}{2}-\lambda)}\cdot\epsilon^{-(1+\lambda)(\frac{1}{2}+\lambda)}~\leq~C\epsilon^{-\frac{5}{8}-\frac{3\lambda}{2}-\frac{\lambda^2}{2}},
\end{align}
which along with \eqref{equ_norm} implies that
\begin{align}\label{uh_infty}
	\|(\mathbf{u}, \mathbf{H})\|_{L^\infty_{txy}}~\leq~C\epsilon^{-\frac{5}{8}-\frac{3\lambda}{2}-\frac{\lambda^2}{2}}.
\end{align}
Therefore, applying \eqref{L-inf-a} and \eqref{uh_infty} in \eqref{Def_err} yields
\begin{align}\label{con}
\begin{aligned}
	&
	\Big\|(\mathbf{u^{\epsilon}}, \mathbf{H^{\epsilon}})(t,x,y)-(\mathbf{u^0}, \mathbf{H^0})(t,x,y)-\big(u^0_b, \sqrt{\epsilon}v^0_b, h^0_b, \sqrt{\epsilon}g^0_b\big)\big(t,x,\frac{y}{\sqrt{\epsilon}}\big)\Big\|_{L^\infty_{txy}}\\
	\leq~&C\sqrt{\epsilon}+C\epsilon\|(\mathbf{u}, \mathbf{H})\|_{L^\infty_{txy}}~\leq~C\sqrt{\epsilon}+C\epsilon^{\frac{3}{8}-\frac{3\lambda}{2}-\frac{\lambda^2}{2}}~\leq~C\epsilon^{\frac{3}{8}-\frac{3\lambda}{2}-\frac{\lambda^2}{2}},
\end{aligned}\end{align}
provided that $\lambda$ is small enough. This ends the proof of Theorem \ref{MAIN THM}.
\end{proof}

\begin{rem}
	From the above proof, we believe the decay rate in \eqref{con} with respect to $\epsilon$ can be improved to order $\sqrt{\epsilon}$. For this purpose,  we need to construct more accurate approximate solutions to the problem \eqref{1.1}-\eqref{BCM}, such that the corresponding remainder terms in \eqref{EA} are of order $\epsilon^\gamma, \gamma>\frac{9}{8}$. Of course, more regularity requirement on the initial data of \eqref{1.1}-\eqref{BCM} is needed.
\end{rem}

\appendix
\section{Expression of error caused by the approximation and its estimates}

Now, we will give the expressions of the remainders $R_i (i=1\sim4)$ in \eqref{EA}, which are generated by the approximate solution $(\mathbf{u^a}, \mathbf{H^a}, p^a)$ in \eqref{NAP},  then prove Proposition \ref{PROP1.5}.
For the simplicity of notations, denote by
\begin{align*}
\begin{aligned}
		&\tau_u(t,x,y) =\chi'(y)\int_0^{\frac{y}{\sqrt{\epsilon}}} u_b^1(t,x,\tilde{\eta})d\tilde{\eta},\quad\tau_h(t,x,y)=\chi'(y)\int_0^{\frac{y}{\sqrt{\epsilon}}} h_b^1(t,x,\tilde{\eta})d\tilde{\eta}+\rho\big(t,x,\frac{y}{\sqrt{\epsilon}}\big),\\
		&\tau_g(t,x,y)=-\int_0^{\frac{y}{\sqrt{\epsilon}}}\p_x\rho(t,x,\tilde{\eta})d\tilde{\eta}),
\end{aligned}\end{align*}
and 
\begin{align*}
\begin{aligned}
\left\{
\begin{array}{ll}
\widetilde{u_b^1}(t,x,y)&=\chi(y)u^1_b\big(t,x,\frac{y}{\sqrt{\epsilon}}\big)+\sqrt{\epsilon}\chi'(y)\int_0^{\frac{y}{\sqrt{\epsilon}}} u_b^1(t,x,\tilde{\eta})d\tilde{\eta},\\
\widetilde{v_b^1}(t,x,y)&=\chi(y)v^1_b\big(t,x,\frac{y}{\sqrt{\epsilon}}\big),\\
\widetilde{h_b^1}(t,x,y)&=\chi(y)h^1_b\big(t,x,\frac{y}{\sqrt{\epsilon}}\big)+\sqrt{\epsilon}\chi'(y)\int_0^{\frac{y}{\sqrt{\epsilon}}} h_b^1(t,x,\tilde{\eta})d\tilde{\eta}+\sqrt{\epsilon}\rho\big(t,x,\frac{y}{\sqrt{\epsilon}}\big),\\
\widetilde{g_b^1}(t,x,y)&=\chi(y)g^1_b(t,x,\eta)-\sqrt{\epsilon}\int_0^{\frac{y}{\sqrt{\epsilon}}}\p_x\rho(t,x,\tilde{\eta})d\tilde{\eta}).
\end{array}\right.
\end{aligned}\end{align*}
Then, the approximation \eqref{NAP} can be rewritten as follows:
\begin{align*}
	\begin{cases}
		(\mathbf{u^{a}},\mathbf{H^{a}})(t,x,y)=(\mathbf{u^0},\mathbf{H^0})(t,x,y)+\big(u^0_b, \sqrt{\epsilon} v_b^0, h^0_b, \sqrt{\epsilon} g_b^0\big)\big(t,x,\frac{y}{\sqrt{\epsilon}}\big)\\
\hspace{1.1in}+\sqrt{\epsilon}\Big[(\mathbf{u^1},\mathbf{H^1})(t,x,y)+\big(\widetilde{u^1_b}, \sqrt{\epsilon} \widetilde{v_b^1}, \widetilde{h^1_b}, \sqrt{\epsilon} \widetilde{g_b^1}\big)(t,x,y)\Big],\\
\hspace{.36in} p^{a}(t,x,y)=p^0(t,x,y)+\sqrt{\epsilon}p^1(t,x,y)+\epsilon p^1_b(t,x,\frac{y}{\sqrt{\epsilon}}).
	\end{cases}
\end{align*}
Moreover, from Proposition \ref{PROP1.4} and the estimate \eqref{rho} of $\rho$ with $m$ large enough, we know that there is a positive constant $C$ independent of $\epsilon$, such that for $|\alpha|\leq5, 0\leq i\leq2$ and $t\in[0,T_4]$,
\begin{align}\label{tau}
	\epsilon^{\frac{i}{2}}\big\|\partial_y^i\partial_{tx}^\alpha (\tau_u, \tau_h, \tau_g)(t,\cdot)\big\|_{L^2}+\big\|(y\partial_y)^i\partial_{tx}^\alpha (\tau_u, \tau_h, \tau_g)(t,\cdot)\big\|_{L^2}\leq C,~
\end{align}
and 
\begin{align}\label{tu}
	\epsilon^{\frac{i}{2}}\big\|\partial_y^i\partial_{tx}^\alpha \big(\widetilde{u_b^1}, \widetilde{v_b^1}, \widetilde{h_b^1}, \widetilde{g_b^1} \big)(t,\cdot)\big\|_{L^2}+\big\|(y\partial_y)^i\partial_{tx}^\alpha \big(\widetilde{u_b^1}, \widetilde{v_b^1}, \widetilde{h_b^1}, \widetilde{g_b^1} \big)(t,\cdot)\big\|_{L^2}\leq C.
\end{align}

Next, we find the remainder terms $R_i, i=1\sim4$ in \eqref{EA} can be divided as follows.
\begin{align}\label{def_R}
R_i=R_i^0+\sqrt{\epsilon}\chi R_i^1+R_i^{C}+\epsilon R_i^{H},\quad i=1, 3,
\end{align}
and 
\begin{align}\label{def_R2}
R_i=R_i^0+\sqrt{\epsilon}\chi R_i^1+\epsilon R_i^{H},\quad i=2,  4.
\end{align}
Each term in the above equalities can be expressed explicitly as follows.
Firstly, $R_i^0,  i=1\sim4$ consist of terms including the leading order profiles $(u_b^0, v_b^0, h_b^0, g_b^0)$:
\begin{align}\label{def-R0}
\begin{aligned}
R_1^0=&\big(u^0-\overline{u^0}-y\overline{\p_y u^0}\big)\p_x u_b^0+\Big[v^0-y\overline{\p_y v^0}-\frac{y^2}{2}\overline{\p^2_{y}v^0}+\sqrt{\epsilon}\big(v^1-\overline{v^1}-y\overline{\p_y v^1}\big)\Big]\p_y u_b^0\\
&-\big(h^0-\overline{h^0}-y\overline{\p_y h^0}\big)\p_x h_b^0-\Big[g^0-y\overline{\p_y g^0}-\frac{y^2}{2}\overline{\p^2_y g^0}+\sqrt{\epsilon}\big(g^1-\overline{g^1}-y\overline{\p_y g^1}\big)\Big]\p_y h_b^0\\
&+\big(\p_x u^0-\overline{\p_x u^0}-y\overline{\p_{xy}^2u^0}\big)u_b^0+\sqrt{\epsilon}\big(\p_y u^0-\overline{\p_y u^0}\big)v_b^0\\
&-\big(\p_x h^0-\overline{\p_x h^0}-y\overline{\p_{xy}^2h^0}\big)h_b^0-\sqrt{\epsilon} \big(\p_y h^0-\overline{\p_y h^0}\big)g_b^0,\\
R_3^0=&\big(u^0-\overline{u^0}-y\overline{\p_y u^0}\big)\p_x h_b^0+\Big[v^0-y\overline{\p_y v^0}-\frac{y^2}{2}\overline{\p^2_{y}v^0}+\sqrt{\epsilon}\big(v^1-\overline{v^1}-y\overline{\p_y v^1}\big)\Big]\p_y h_b^0\\
&-\big(h^0-\overline{h^0}-y\overline{\p_y h^0}\big)\p_x u_b^0-\Big[g^0-y\overline{\p_y g^0}-\frac{y^2}{2}\overline{\p^2_y g^0}+\sqrt{\epsilon}\big(g^1-\overline{g^1}-y\overline{\p_y g^1}\big)\Big]\p_y u_b^0\\
&+\big(\p_x h^0-\overline{\p_x h^0}-y\overline{\p_{xy}^2h^0}\big)u_b^0+\sqrt{\epsilon}\big(\p_y h^0-\overline{\p_y h^0}\big)v_b^0\\
&-\big(\p_x u^0-\overline{\p_x u^0}-y\overline{\p_{xy}^2u^0}\big)h_b^0-\sqrt{\epsilon}\big(\p_y u^0-\overline{\p_y u^0}\big)g_b^0,
\end{aligned}\end{align}
and
\begin{align*}
R_2^{0}=&\sqrt{\epsilon}\big(u^0-\overline{u^0}\big)\partial_x v_b^0+\sqrt{\epsilon}\Big[v^0-y\overline{\partial_y v^0}+\sqrt{\epsilon}\big(v^1-\overline{v^1}\big)\Big]\partial_y v_b^0-\sqrt{\epsilon}\big(h^0-\overline{h^0}\big)\partial_x g_b^0\\
&-\sqrt{\epsilon}\Big[g^0-y\overline{\partial_y g^0}+\sqrt{\epsilon}\big(g^1-\overline{g^1}\big)\Big]\partial_y g_b^0+\Big[\partial_x v^0-y\overline{\partial_{xy}^2 v^0}+\sqrt{\epsilon}\big(\partial_x v^1-\overline{\partial_x v^1}\big)\Big]u_b^0\\
&+\sqrt{\epsilon}\big(\partial_y v^0-\overline{\partial_y v^0} \big)v_b^0-\Big[\partial_x g^0-y\overline{\partial_{xy}^2 g^0}+\sqrt{\epsilon}\big(\partial_x g^1-\overline{\partial_x g^1}\big)\Big]h_b^0-\sqrt{\epsilon}\big(\partial_y g^0-\overline{\partial_y g^0} \big)g_b^0,\\
R_4^{0}=&\sqrt{\epsilon}\big(u^0-\overline{u^0}\big)\partial_x g_b^0+\sqrt{\epsilon}\Big[v^0-y\overline{\partial_y v^0}+\sqrt{\epsilon}\big(v^1-\overline{v^1}\big)\Big]\partial_y g_b^0-\sqrt{\epsilon}\big(h^0-\overline{h^0}\big)\partial_x v_b^0\\
&-\sqrt{\epsilon}\Big[g^0-y\overline{\partial_y g^0}+\sqrt{\epsilon}\big(g^1-\overline{g^1}\big)\Big]\partial_y v_b^0+\Big[\partial_x g^0-y\overline{\partial_{xy}^2 g^0}+\sqrt{\epsilon}\big(\partial_x g^1-\overline{\partial_x g^1}\big)\Big]u_b^0\\
&+\sqrt{\epsilon}\big(\partial_y g^0-\overline{\partial_y g^0} \big)v_b^0
-\Big[\partial_x v^0-y\overline{\partial_{xy}^2 v^0}+\sqrt{\epsilon}\big(\partial_x v^1-\overline{\partial_x v^1}\big)\Big]h_b^0-\sqrt{\epsilon}\big(\partial_y v^0-\overline{\partial_y v^0} \big)g_b^0.
\end{align*}

Secondly, $R_i^1,  i=1\sim4$ are composed of terms related to the first order profiles $(u_b^1, v_b^1, h_b^1, g_b^1)$:
\begin{align}\label{def-R1}
\begin{aligned}
	R_1^1=&\big(u^0-\overline{u^0} \big)\partial_x u_b^1+\Big[v^0-y\overline{\partial_y v^0}+\sqrt{\epsilon}\big(v^1-\overline{v^1}\big)\Big]\partial_y u_b^1-\big(h^0-\overline{h^0} \big)\partial_x h_b^1\\
	&-\Big[g^0-y\overline{\partial_y g^0}+\sqrt{\epsilon}\big(g^1-\overline{g^1}\big)\Big]\partial_y h_b^1+\big(\partial_x u^0-\overline{\partial_x u^0} \big)u_b^1-\big(\partial_x h^0-\overline{\partial_x h^0} \big)h_b^1,\\
R_3^1=&\big(u^0-\overline{u^0} \big)\partial_x h_b^1+\Big[v^0-y\overline{\partial_y v^0}+\sqrt{\epsilon}\big(v^1-\overline{v^1}\big)\big]\partial_y h_b^1-\big(h^0-\overline{h^0} \big)\partial_x u_b^1\\
&-\Big[g^0-y\overline{\partial_y g^0}+\sqrt{\epsilon}\big(g^1-\overline{g^1}\big)\Big]\partial_y u_b^1+\big(\partial_x h^0-\overline{\partial_x h^0} \big)u_b^1-\big(\partial_x u^0-\overline{\partial_x u^0} \big)h_b^1, 	
\end{aligned}\end{align}
and
\begin{align*}
R_2^1=&\sqrt{\epsilon}v^0\partial_y v_b^1-\sqrt{\epsilon}g^0\partial_y g_b^1 +\partial_x v^0 u_b^1-\partial_x g^0 h_b^1,\\
	R_4^1=&\sqrt{\epsilon}v^0\partial_y g_b^1-\sqrt{\epsilon}g^0\partial_y v_b^1 +\partial_x g^0 u_b^1-\partial_x v^0 h_b^1.
\end{align*}

Thirdly, $R_i^{C} (i=1,3)$ in \eqref{def_R} 
are caused by the cut-off function $\chi(y)$ and listed as follows.
\begin{align}\label{CC1}
\begin{aligned}
R_1^{C}=&(1-\chi)\Big[\big(y\overline{\partial_y u^0}+\sqrt{\epsilon}\overline{u^1}\big)\partial_x u_b^0+\big(\frac{y^2}{2}\overline{\partial_y^2v^0}+\sqrt{\epsilon}y\overline{\partial_y v^1}\big)\partial_y u_b^0+\big(y\overline{\partial_{xy}^2u^0}+\sqrt{\epsilon}\overline{\partial_x u^1}\big) u_b^0+\sqrt{\epsilon}\overline{\partial_y u^0}~v_b^0\\
&\qquad\quad-\big(y\overline{\partial_y h^0}+\sqrt{\epsilon}\overline{h^1}\big)\partial_x h_b^0-\big(\frac{y^2}{2}\overline{\partial_y^2g^0}+\sqrt{\epsilon}y\overline{\partial_y g^1}\big)\partial_y h_b^0-\big(y\overline{\partial_{xy}^2h^0}+\sqrt{\epsilon}\overline{\partial_x h^1}\big) h_b^0-\sqrt{\epsilon}\overline{\partial_y h^0}~g_b^0\Big]\\
&+\sqrt{\epsilon}v^0\big(\chi'u_b^1+\sqrt{\epsilon} \partial_y \tau_u\big)-\sqrt{\epsilon}g^0\big(\chi' h_b^1+\sqrt{\epsilon}\partial_y \tau_h\big),\\
R_3^{C}=&(1-\chi)\Big[\big(y\overline{\partial_y u^0}+\sqrt{\epsilon}\overline{u^1}\big)\partial_x h_b^0+\big(\frac{y^2}{2}\overline{\partial_y^2v^0}+\sqrt{\epsilon}y\overline{\partial_y v^1}\big)\partial_y h_b^0+\big(y\overline{\partial_{xy}^2 h^0}+\sqrt{\epsilon}\overline{\partial_x h^1}\big) u_b^0+\sqrt{\epsilon}\overline{\partial_y h^0}~v_b^0\\
&\qquad\quad-\big(y\overline{\partial_y h^0}+\sqrt{\epsilon}\overline{h^1}\big)\partial_x u_b^0-\big(\frac{y^2}{2}\overline{\partial_y^2g^0}+\sqrt{\epsilon}y\overline{\partial_y g^1}\big)\partial_yu_b^0-\big(y\overline{\partial_{xy}^2u^0}+\sqrt{\epsilon}\overline{\partial_x u^1}\big) h_b^0-\sqrt{\epsilon}\overline{\partial_y u^0}~g_b^0\Big]\\
&+\sqrt{\epsilon}v^0\big(\chi'h_b^1+\sqrt{\epsilon}\partial_y \tau_h\big)-\sqrt{\epsilon}g^0\big(\chi' u_b^1+\sqrt{\epsilon}\partial_y \tau_u\big).
\end{aligned}\end{align}

Finally, $R_i^H,  i=1\sim4$ correspond to the high order terms of $R_i$ with respect to $\epsilon$:
\begin{align}\label{def-RH}
\begin{aligned}
	R_1^{H}=&\partial_x p_b^1+ \partial_t \tau_u+\big(u^1+\widetilde{u_b^1}\big)\partial_x\big(u^1+\widetilde{u_b^1}\big)+\partial_x\big[(u^0+u_b^0)\tau_u\big]+(v^1+v_b^0)\big(\partial_y u^1+\chi' u_b^1+\sqrt{\epsilon}\partial_y\tau_u \big)\\
	&+\widetilde{v_b^1}\partial_y\big(u^0+\sqrt{\epsilon}u^1+\sqrt{\epsilon}\widetilde{u_b^1}\big)-\big(h^1+\widetilde{h_b^1}\big)\partial_x\big(h^1+\widetilde{h_b^1}\big)-\partial_x\big[(h^0+h_b^0)\tau_h\big]\\
	&-(g^1+g_b^0)\big(\partial_y h^1+\chi' h_b^1+\sqrt{\epsilon}\partial_y\tau_h \big)-\widetilde{g_b^1}\partial_y\big(h^0+\sqrt{\epsilon}h^1+\sqrt{\epsilon}\widetilde{h_b^1}\big)-\sqrt{\epsilon}\tau_g \partial_y h_b^0\\
	&-\mu\Big[\triangle \big(u^0+\sqrt{\epsilon} u^1\big)+\partial_x^2\big( u_b^0+\sqrt{\epsilon} \widetilde{u_b^1}\big)+2\sqrt{\epsilon}\chi'\partial_y u_b^1+\sqrt{\epsilon}\chi'' u_b^1+\epsilon\partial_y^2\tau_u\Big],\\
	R_3^{H}=& \partial_t \tau_h+\big(u^1+\widetilde{u_b^1}\big)\partial_x\big(h^1+\widetilde{h_b^1}\big)+(u^0+u_b^0)\partial_x\tau_h+\tau_u\partial_x(h^0+h_b^0)+(v^1+v_b^0)\big(\partial_y h^1+\chi' h_b^1+\sqrt{\epsilon}\partial_y\tau_h \big)\\
	&+\widetilde{v_b^1}\partial_y\big(h^0+\sqrt{\epsilon}h^1+\sqrt{\epsilon}\widetilde{h_b^1}\big)-\big(h^1+\widetilde{h_b^1}\big)\partial_x\big(u^1+\widetilde{u_b^1}\big)-(h^0+h_b^0)\partial_x\tau_u-\tau_h\partial_x(u^0+u_b^0)\\
	&-(g^1+g_b^0)\big(\partial_y u^1+\chi' u_b^1+\sqrt{\epsilon}\partial_y\tau_u \big)-\widetilde{g_b^1}\partial_y\big(u^0+\sqrt{\epsilon}u^1+\sqrt{\epsilon}\widetilde{u_b^1}\big)-\sqrt{\epsilon}\tau_g \partial_y u_b^0\\
	&-\kappa\Big[\triangle\big( h^0+\sqrt{\epsilon} h^1\big)+\partial_x^2\big( h_b^0+\sqrt{\epsilon}\widetilde{h_b^1}\big)+2\sqrt{\epsilon}\chi'\partial_y h_b^1+\sqrt{\epsilon}\chi'' h_b^1+\epsilon\partial_y^2\tau_h\Big],
	\end{aligned}\end{align}
	and
	\begin{align*}
		R_2^H=&\partial_t\widetilde{v_b^1}+\big(u^1+\widetilde{u_b^1} \big)\partial_x(v^1+v_b^0)+(v^1+v_b^0)\partial_y\big(v^1+\sqrt{\epsilon}\widetilde{v_b^1}\big)+u^a\partial_x\widetilde{v_b^1}+\widetilde{v_b^1}\partial_y v^a+\partial_x v^0 \tau_u+\chi' v^0 v_b^1\\
	&-\big(h^1+\widetilde{h_b^1} \big)\partial_x(g^1+g_b^0)-(g^1+g_b^0)\partial_y\big(g^1+\sqrt{\epsilon}\widetilde{g_b^1}\big)-h^a\partial_x\widetilde{g_b^1}-\widetilde{g_b^1}\partial_y g^a-\partial_x g^0 \tau_h-g^0\big(\chi' g_b^1+\sqrt{\epsilon}\partial_y\tau_g\big)\\
	&-\mu\big[\triangle\big(v^0+\sqrt{\epsilon}v^1+\epsilon\widetilde{v_b^1} \big)+\sqrt{\epsilon}\partial_x^2v_0^b \big],\\
	R_4^H=&\partial_t\widetilde{g_b^1}+\big(u^1+\widetilde{u_b^1} \big)\partial_x(g^1+g_b^0)+(v^1+v_b^0)\partial_y\big(h^1+\sqrt{\epsilon}\widetilde{g_b^1}\big)+u^a\partial_x\widetilde{g_b^1}+\widetilde{v_b^1}\partial_y g^a+\partial_x g^0 \tau_u+v^0\big(\chi' g_b^1+\sqrt{\epsilon}\partial_y\tau_g\big)\\
	&-\big(h^1+\widetilde{h_b^1} \big)\partial_x(v^1+v_b^0)-(g^1+g_b^0)\partial_y\big(v^1+\sqrt{\epsilon}\widetilde{v_b^1}\big)-h^a\partial_x\widetilde{v_b^1}-\widetilde{g_b^1}\partial_y v^a-\partial_x v^0 \tau_h-\chi' g^0v_b^1\\
	 &-\kappa\big[\triangle\big(g^0+\sqrt{\epsilon}g^1+\epsilon\widetilde{g_b^1} \big)+\sqrt{\epsilon}\partial_x^2g_0^b \big].
\end{align*}

Based on the above exact expressions for the error terms $R_i(i=1\sim4)$, taking into account the estimates of $(\mathbf{u^i}, \mathbf{H^i}) (i=0,1)$ in Propositions \ref{PROP1.1} and \ref{PROP1.3} respectively, and the estimates of $(u^j_b,h^j_b) (j=0,1)$ in Propositions \ref{PROP2} and \ref{PROP1.4} respectively, and the estimate \eqref{est_p} of $p_b^1$, we are able to prove Proposition \ref{PROP1.5}.

\begin{proof}[\textbf{Proof of Proposition \ref{PROP1.5}.}]
We only show the $L^2$-estimate of $R_1$ in (\ref{ERED}). The estimates for other $R_i, i=2,3,4$ 
can be estimated similarly. Moreover, If one applies the tangential derivatives operators $\p^\alpha_{tx}, |\alpha|\leq 3$ on the error terms $R_i, 1\leq i\leq4$, it does not produce any singular factor $\displaystyle \frac{1}{\sqrt{\epsilon}}$ in the formulation. Consequently, we can prove (\ref{ERED}) for $|\alpha|\leq 3$.
\bigskip

From \eqref{def_R}, the $L^2-$estimate of $R_1$ will be divided into three parts.\\
{\bf Part I:} Estimates of $R_1^0$ and $\sqrt{\epsilon}\chi R_1^1$.\\
By Taylor expansion, it follows that for $\displaystyle \eta=\frac{y}{\sqrt{\epsilon}}$ and some $\theta_y, \tilde\theta_y, \hat\theta_y \in[0,y]$,
\begin{align*}
\big(u^0-\overline{u^0}-y\overline{\p_y u^0}\big)\p_x u_b^0=\frac{y^2}{2}\p_y^2u^0(t,x,\theta_y)\cdot\p_x u_b^0\big(t,x,\frac{y}{\sqrt{\epsilon}}\big)=\frac{\epsilon}{2}\p_y^2u^0(t,x,\theta_y)\cdot\eta^2\p_x u_b^0(t,x,\eta),
\end{align*}
 and combining with the boundary condition $\overline{v^0}(t,x)=v^0|_{y=0}=0,$
\begin{align*}
	\Big[v^0-y\overline{\p_y v^0}-\frac{y^2}{2}\overline{\p^2_{y}v^0}+\sqrt{\epsilon}\big(v^1-\overline{v^1}-y\overline{\p_y v^1}\big)\Big]\p_y u_b^0
	=&\Big[\frac{y^3}{6}\p_y^3v^0(t,x,\tilde {\theta}_y)+\sqrt{\epsilon} \frac{y^2}{2}\p_y^2 v^1(t,x,\hat {\theta}_y)\Big]\p_y u_b^0\big(t,x,\frac{y}{\sqrt{\epsilon}}\big)\\
	=&\epsilon\Big[\frac{\p_y^3v^0(t,x,\tilde {\theta}_y)}{6}\eta^3+ \frac{\p_y^2 v^1(t,x,\hat {\theta}_y)}{2}\eta^2\Big]\p_\eta u_b^0(t,x,\eta).
\end{align*}
  Then from Propositions \ref{PROP1.1} and \ref{PROP2},
\begin{align*}
\begin{aligned}
\int_{\mathbb{T}}\int_{\mathbb{R}_+}\Big[\big(u^0-\overline{u^0}-y\overline{\p_y u^0}\big)\p_x u_b^0\Big]^2dydx
=~&\epsilon^{5/2}\int_{\mathbb{T}}\int_{\mathbb{R}_+}\Big[\frac{\p_y^2u^0(t,x,\theta_y)}{2}\eta^2\p_xu_b^0(t,x,\eta)\Big]^2d\eta dx\\
\leq~& \frac{\epsilon^{5/2}}{2}\|\p_y^2u^0(t,\cdot)\|^2_{L^\infty(\mathbb{T}\times\mathbb{R}_+)}\|\partial_x u_b^0\|^2_{L^2_2(\Omega)}\\ 
\leq& C\epsilon^{5/2},
\end{aligned}\end{align*}
which implies
\begin{align*}
\Big\|\Big[\big(u^0-\overline{u^0}-y\overline{\p_y u^0}\big)\p_x u_b^0\Big](t,\cdot)\Big\|_{L^2}~\leq ~C\epsilon^{5/4}.
\end{align*}
Similarly, we have
\begin{align*}
\Big\|\Big[v^0-y\overline{\p_y v^0}-\frac{y^2}{2}\overline{\p^2_{y}v^0}+\sqrt{\epsilon}\big(v^1-\overline{v^1}-y\overline{\p_y v^1}\big)\Big]\p_y u_b^0(t,\cdot)\Big\|_{L^2)}~\leq ~C\epsilon^{5/4}.
\end{align*}
Other terms in $R_1^0$ and $R_1^1$ can be estimated in the same way by using the results in Propositions \ref{PROP1.1}, \ref{PROP2}, \ref{PROP1.3} and \ref{PROP1.4}. Consequently,
\begin{align*}
\big\|R_1^0(t,\cdot)\big\|_{L^2}+\sqrt{\epsilon}\big\|(\chi R_1^1)(t,\cdot)\big\|_{L^2} ~\leq~ C\epsilon^{5/4}.
\end{align*}

\indent\newline
{\bf Part II:} Estimates of $R_1^{C}$.\\
By the cut-off function $\chi(y)$ and Propositions \ref{PROP1.1} and \ref{PROP1.3}, it yields
\begin{align}\label{R1C-1}
\begin{aligned}
&\Big\|\Big[(1-\chi)\big(y\overline{\partial_y u^0}+\sqrt{\epsilon} \overline{u^1}\big)\partial_x u_b^0\Big](t,\cdot)\Big\|_{L^2}^2\\
=~&\epsilon^{\frac{3}{2}}\int_{\mathbb{T}}\int_{1/\sqrt{\epsilon}}^{\infty}\Big[\big(1-\chi(\sqrt{\epsilon}\eta)\big)\big(\eta\overline{\partial_y u^0}(t,x)+\overline{u^1}(t,x)\big)\partial_x u_b^0(t,x,\eta)\Big]^2d\eta dx\\
\leq~&2\epsilon^{\frac{3}{2}}\int_{\mathbb{T}}\int_{1/\sqrt{\epsilon}}^{\infty}(\sqrt{\epsilon}\eta)^{2l}\cdot\Big[\big(\eta\overline{\partial_y u^0}(t,x)+\overline{u^1}(t,x)\big)\partial_x u_b^0(t,x,\eta)\Big]^2d\eta dx\\
\leq~&4\epsilon^{\frac{3}{2}+l}\Big(\big\|\overline{\p_y u^0}\big\|_{L^\infty(\mathbb{T})}^2\|\partial_x u_b^0\|^2_{L_{1+l}^2(\Omega)}+\big\|\overline{u^1}\big\|_{L^\infty(\mathbb{T})}^2\|\partial_x u_b^0\|^2_{L_{l}^2(\Omega)}\Big)\\
\leq~&C\epsilon^{\frac{3}{2}+l}
\end{aligned}\end{align}
for any $l\geq0$.
The analogous argument yields that for any $l\geq0$,
\begin{align}
\begin{aligned}
&\Big\|\Big[(1-\chi)\big(\frac{y^2}{2}\overline{\partial_y^2v^0}+\sqrt{\epsilon}y\overline{\partial_y v^1}\big)\partial_y u_b^0\Big](t,\cdot)\Big\|_{L^2}^2\leq C\epsilon^{\frac{3}{2}+l}.
\end{aligned}\end{align}
Similarly, it holds by using Proposition \ref{PROP1.4} that
\begin{align}\label{R1C-2}
\begin{aligned}
\Big\|\big(\sqrt{\epsilon}\chi' v^0 u_b^1\big)(t,\cdot)\Big\|_{L^2(\mathbb{T}\times\mathbb{R}_+)}^2=~&\epsilon^{\frac{3}{2}}\int_{\mathbb{T}}\int_{1/\sqrt{\epsilon}}^{2/\sqrt{\epsilon}}\Big[\chi'(\sqrt{\epsilon}\eta)v^0(t,x,\sqrt{\epsilon}\eta)\cdot u_b^1(t,x,\eta)\Big]^2d\eta dx\\
\leq~&C\epsilon^{\frac{3}{2}+l}\|v^0\|^2_{L^\infty(\mathbb{T}\times\mathbb{R}_+)}\|u_b^1\|_{L^2_l(\Omega)}^2
\leq~C\epsilon^{\frac{3}{2}+l}
\end{aligned}\end{align}
for any $l\geq0.$ In addition, it follows that by the boundary condition $v^0|_{y=0}=0,$
\begin{align*}
	\epsilon v^0\partial_y\tau_u=\epsilon\partial_y v^0(t,x,\theta_y)y\cdot\partial_y\tau_u(t,x,y)
\end{align*}
for some $\theta_y\in[0,y]$, which along with
\eqref{tau} implies that
\begin{align}\label{R1C-3}
	\big\|\big(\epsilon v^0\partial_y\tau_u\big)(t,\cdot)\|_{L^2}\leq~&\epsilon\|\partial_y v^0(t,\cdot)\|_{L^\infty}\|(y\partial_y\tau_u)(t,\cdot)\|_{L^2}\leq~ C\epsilon.
\end{align}
Thus, combining \eqref{R1C-1}-\eqref{R1C-2} with $l\geq\frac{1}{2}$ and \eqref{R1C-3}, and noting that other terms of $R_1^C$ can be investigated similarly, one can obtain
\begin{align*}
	\big\|R_1^C\big\|_{L^2(\mathbb{T}\times\mathbb{R}_+)}~\leq~C\epsilon.
\end{align*}

{\bf Part III:} Estimates of $R_1^{H}$.\\
From Propositions \ref{PROP1.1}-\ref{PROP1.4}, and using \eqref{est_p}, \eqref{tau} and \eqref{tu}, it is easy to check that the $L^2$-norm of each term in $R_1^{H}$ is uniformly bounded with respect to $\epsilon$. As a consequence, 
\begin{align*}
	\big\|\epsilon R_1^H(t,\cdot)\big\|_{L^2(\mathbb{T}\times\mathbb{R}_+)}~\leq~C\epsilon.
\end{align*}

Finally, combining all estimates in Parts I-III reads
\begin{align*}
\big\|R_1(t,\cdot)\big\|_{L^2}~\leq~ C\epsilon.
\end{align*}

\end{proof}

\section{ Proof of Lemma \ref{lem_equ}}
In this part, we give the proof of Lemma \ref{lem_equ} to show the domination of the newly defined functions $(\tilde u,\tilde v, \tilde h, \tilde g)$, given by \eqref{trans}, over the original unknown $(\mathbf{u}, \mathbf{H})$  in $L^p (1<p\leq\infty)$ norm. 
\begin{proof}[Proof of Lemma \ref{lem_equ}]
	Combining \eqref{def_psi} with \eqref{trans}, it follows
	\begin{align}\label{trans1}
		\frac{\tilde h(t,x,y)}{h^p\big(t,x,\frac{y}{\sqrt{\epsilon}}\big)}~=~\partial_y\Big(\frac{\psi(t,x,y)}{h^p\big(t,x,\frac{y}{\sqrt{\epsilon}}\big)}\Big),\quad \psi(t,x,y)~=~h^p\big(t,x,\frac{y}{\sqrt{\epsilon}}\big)\cdot\partial_y^{-1}\Big(\frac{\tilde h(t,x,y)}{h^p\big(t,x,\frac{y}{\sqrt{\epsilon}}\big)}\Big),
	\end{align}
	and then, the Hardy inequality yields that by the upper-lower bound of $h^p(t,x,\eta)$ given in Proposition \ref{PROP2},
	\begin{align}\label{est_psi}
		\Big\|\frac{\psi(t,x,y)}{y}\Big\|_{L^p}~\lesssim~\Big\|\frac{1}{y}\partial_y^{-1}\Big(\frac{\tilde h(t,x,y)}{h^p\big(t,x,\frac{y}{\sqrt{\epsilon}}\big)}\Big)\Big\|_{L^p}~\lesssim~\Big\|\frac{\tilde h(t,x,y)}{h^p\big(t,x,\frac{y}{\sqrt{\epsilon}}\big)}\Big\|_{L^p}~\leq~C\|\tilde h(t,\cdot)\|_{L^p},\quad 1<p\leq\infty.
	\end{align}
By a direct calculation,
\begin{align*}
\partial_{tx}^\alpha \psi(t,x,y)=\sum_{\beta\leq\alpha}C_\alpha^\beta\Big\{\partial_{tx}^{\alpha-\beta}h^p\big(t,x,\frac{y}{\sqrt{\epsilon}}\big)\cdot\partial_y^{-1}\partial_{tx}^\beta\Big(\frac{\tilde h(t,x,y)}{h^p\big(t,x,\frac{y}{\sqrt{\epsilon}}\big)}\Big) \Big\},
\end{align*}
similarly, 
it implies that for $|\alpha|\leq2$ and $1\leq p\leq\infty,$
\begin{align}\label{est_psi-x}
\begin{aligned}
\displaystyle
\Big\|\frac{1}{y}\partial_{tx}^\alpha \psi(t,x,y)\Big\|_{L^p}\leq&\sum_{\beta\leq\alpha}C_\alpha^\beta\Big\{\big\|\partial_{tx}^{\alpha-\beta}h^p\big(t,x,\frac{y}{\sqrt{\epsilon}}\big)\big\|_{L^\infty}\cdot\Big\|\frac{1}{y}\partial_y^{-1}\partial_{tx}^\beta\Big(\frac{\tilde h(t,x,y)}{h^p\big(t,x,\frac{y}{\sqrt{\epsilon}}\big)}\Big)\Big\|_{L^p} \Big\}\\
\lesssim&\sum_{\beta\leq\alpha}\Big\|\partial_{tx}^\beta\Big(\frac{\tilde h(t,x,y)}{h^p\big(t,x,\frac{y}{\sqrt{\epsilon}}\big)}\Big)\Big\|_{L^p}\leq C\sum_{\beta\leq\alpha}\big\|\partial_{tx}^\beta \tilde h(t,\cdot)\big\|_{L^p}.
\end{aligned}\end{align}

	Next, by the transformation \eqref{trans},
	\begin{align*}
	\begin{aligned}
		&u(t,x,y)=\tilde u(t,x,y)+a^p\big(t,x,\frac{y}{\sqrt{\epsilon}}\big)\cdot\tilde h(t,x,y)+\big(\partial_y a^p+a^pb^p\big)\big(t,x,\frac{y}{\sqrt{\epsilon}}\big)\cdot\psi(t,x,y),\\
		&v(t,x,y)=\tilde v(t,x,y)+a^p\big(t,x,\frac{y}{\sqrt{\epsilon}}\big)\cdot\tilde g(t,x,y)-\partial_x a^p\big(t,x,\frac{y}{\sqrt{\epsilon}}\big)\cdot\psi(t,x,y),\\
				&h(t,x,y)=\tilde h(t,x,y)+b^p\big(t,x,\frac{y}{\sqrt{\epsilon}}\big)\cdot\psi(t,x,y),	\quad g(t,x,y)=\tilde g(t,x,y).	
		\end{aligned}\end{align*}
		It yields that by using \eqref{est_psi}, 		\begin{align*}
			\|u(t,\cdot)\|_{L^p}\leq &\|\tilde u(t,\cdot)\|_{L^p}+\|a^p(t,\cdot)\|_{L^\infty}\|\tilde h(t,\cdot)\|_{L^p}+\Big(\|y\partial_ya^p(t,\cdot)\|_{L^\infty}+\|y(a^pb^p)(t,\cdot)\|_{L^\infty}\Big)\Big\|\frac{1}{y}\psi(t,\cdot)\Big\|_{L^p}\\
       \leq&\|\tilde u(t,\cdot)\|_{L^p}+C\|\tilde h(t,\cdot)\|_{L^p},
		\end{align*}
		and similarly,
		\begin{align*}
			\|v(t,\cdot)\|_{L^p}\leq \|\tilde v(t,\cdot)\|_{L^p}+C\|\tilde g(t,\cdot)\|_{L^p}+C\|\tilde h(t,\cdot)\|_{L^p},\quad \|h(t,\cdot)\|_{L^p}\leq C\|\tilde h(t,\cdot)\|_{L^p}.
		\end{align*}
Furthermore, one has 
\begin{align*}
	\partial_{tx}^\alpha u(t,x,y)=\partial_{tx}^\alpha\tilde u(t,x,y)+\sum_{\beta\leq\alpha}C_{\alpha}^{\beta}\Big\{&\partial_{tx}^{\alpha-\beta} a^p\big(t,x,\frac{y}{\sqrt{\epsilon}}\big)\cdot\partial_{tx}^\beta\tilde h(t,x,y)\nonumber\\
	&+\partial_{tx}^{\alpha-\beta}\big(\partial_y a^p+a^pb^p\big)\big(t,x,\frac{y}{\sqrt{\epsilon}}\big)\cdot\partial_{tx}^\beta\psi(t,x,y)\Big\},
\end{align*}
along with \eqref{est_psi-x} and the boundedness of $a^p$, it follows that for $|\alpha|\leq2$ and $1\leq p\leq\infty,$
\begin{align*}
	\big\|\partial_{tx}^\alpha u(t,\cdot)\big\|_{L^p}\leq~&\big\|\partial_{tx}^\alpha\tilde u(t,\cdot)\big\|_{L^p}+\sum_{\beta\leq\alpha}C_{\alpha}^{\beta}\Big\{\big\|\partial_{tx}^{\alpha-\beta} a^p\big(t,x,\frac{y}{\sqrt{\epsilon}}\big)\big\|_{L^\infty}\cdot\big\|\partial_{tx}^\beta\tilde h(t,\cdot)\|_{L^p}\\
	&\quad+\big(\big\|y\partial_y\partial_{tx}^{\alpha-\beta} a^p\big(t,x,\frac{y}{\sqrt{\epsilon}}\big)\big\|_{L^\infty}+\big\|y\partial_{tx}^{\alpha-\beta} (a^pb^p)\big(t,x,\frac{y}{\sqrt{\epsilon}}\big)\big\|_{L^\infty}\big)\cdot\Big\|\frac{1}{y}\partial_{tx}^\beta\psi(t,x,y)\Big\|_{L^p}\Big\}\\
	\leq~& \big\|\partial_{tx}^\alpha\tilde u(t,\cdot)\big\|_{L^p}+C\sum_{\beta\leq\alpha}\big\|\partial_{tx}^\beta\tilde h(t,\cdot)\|_{L^p}.
\end{align*}
Likewise, one can obtain that
\begin{align*}
	\big\|\partial_{tx}^\alpha v(t,\cdot)\big\|_{L^p}
	&\leq \big\|\partial_{tx}^\alpha\tilde v(t,\cdot)\big\|_{L^p}+C\sum_{\beta\leq\alpha}\Big(\big\|\partial_{tx}^\beta \tilde g(t,\cdot)\|_{L^p}+\big\|\partial_{tx}^\beta\tilde  h(t,\cdot)\|_{L^p}\Big),~
	\end{align*}
	and
	\begin{align*}
	\big\|\partial_{tx}^\alpha h(t,\cdot)\big\|_{L^p}
	\leq C\sum_{\beta\leq\alpha}
	\big\|\partial_{tx}^\beta\tilde h(t,\cdot)\|_{L^p}.
\end{align*}
As $g=\tilde g$, it is nature to get
\begin{align*}
	\big\|\partial_{tx}^\alpha g(t,\cdot)\|_{L^p}~=~\big\|\partial_{tx}^\alpha\tilde g(t,\cdot)\|_{L^p}.
\end{align*}
Combining the above estimates, we obtain \eqref{equ_norm} immediately.
\end{proof}

\section{Expressions of some notations in the problem $\ref{pr_main}$ and their estimates}
Firstly, we provide the explicit expressions of matrices $C^a, \tilde{C}(U)$ and vectors $D^a, D^p$ given in \eqref{def_C}. Precisely,
\begin{align*}
	C^a=\left(\begin{array}{cccc}
		-\partial_y(v^a-a^p g^a) & \partial_y(u^a-a^p h^a) & C_{13}^a & C_{14}^a\\
		\partial_x(v^a-a^p g^a) & -\partial_x(u^a-a^p h^a) & C_{23}^a & C_{24}^a\\
		\partial_x h^a+b^p g^a & \partial_y h^a-b^p h^a & C_{33}^a & C_{34}^a\\
		\partial_x g^a &\partial_y g^a & C_{43}^a & C_{44}^a
	\end{array}\right)
	\end{align*}
	with
\begin{align*}
	C^a_{13}=&-2a^p\partial_y(v^a-a^p g^a)-[(a^p)^2-1](\partial_y g^a-b^p g^a)+[\partial_t+(u^a+a^p h^a)\partial_x+(v^a+a^p g^a)\partial_y-\mu\epsilon\triangle-2\mu\epsilon\partial_y^2]a^p\\
&-\epsilon\big\{(3\mu-\kappa)b^p\partial_y a^p+(\mu-\kappa)a^p\big[2\partial_y b^p+(b^p)^2\big]\big\},\\ 
	C^a_{14}=&2a^p\partial_y(u^a-a^p h^a)+[(a^p)^2-1](\partial_y h^a-b^p h^a)-2\epsilon\big[(\mu-\kappa)a^p\partial_x b^p-\mu b^p\partial_xa^p-\mu\partial_{xy}^2a^p \big],\\
	C^a_{23}=&2a^p\partial_x(v^a-a^p g^a)+[(a^p)^2-1]\partial_x g^a+2\epsilon[\mu\partial_{xy}^2a^p+(\mu-\kappa)b^p\partial_x a^p],\\
	C^a_{24}=&-2a^p\partial_x(u^a-a^p h^a)-[(a^p)^2-1]\partial_x h^a+[\partial_t+(u^a+a^p h^a)\partial_x+(v^a+a^p g^a)\partial_y-\mu\epsilon\triangle-2\mu\epsilon\partial_x^2]a^p,\\
	C^a_{33}=&\partial_y(v^a-a^p g^a)-(h^a\partial_x+g^a\partial_y)a^p-2\kappa\epsilon\partial_y b^p,\quad
	C^a_{34}=-\partial_y(u^a-a^p h^a)+2\kappa\epsilon\partial_x b^p,\\
	C_{43}^a=&-\partial_x(v^a-a^p g^a),\quad C^a_{44}=\partial_x(u^a-a^p h^a)-(h^a\partial_x+g^a\partial_y)a^p.
\end{align*}

\begin{align*}
	\tilde C(U)=\left(\begin{array}{cccc}
		\partial_y(\partial_xa^p\cdot \psi) & \partial_y(\partial_ya^p\cdot \psi) & \tilde C_{13}(U) & \tilde C_{14}(U)\\
		 -\partial_x(\partial_xa^p\cdot \psi) & -\partial_x(\partial_ya^p\cdot \psi) & -2a^p\partial_x(\partial_xa^p\cdot\psi) & 2a^p\big[b^p\partial_xa^p \cdot\psi-\partial_x(\partial_ya^p\cdot\psi)\big]\\
		 \partial_xb^p\cdot\psi & \partial_yb^p\cdot\psi & -\partial_y(\partial_xa^p\cdot\psi)
		 -b^p\partial_xa^p\cdot\psi & -\partial_y(\partial_ya^p\cdot\psi )\\
		 0 & 0 & \partial_x(\partial_xa^p\cdot\psi) & \partial_x(\partial_ya^p\cdot\psi)-b^p\partial_xa^p\cdot\psi
	\end{array}\right)
\end{align*}
with
\begin{align*}
	&\tilde C_{13}(U)= 2a^p\big[\partial_y(\partial_xa^p\cdot \psi)+b^p\partial_xa^p\cdot\psi\big]+[(a^p)^2-1]\partial_xb^p\cdot\psi,\\
	&\tilde C_{14}(U)= 2a^p\partial_y(\partial_ya^p\cdot \psi)+[(a^p)^2-1]\partial_yb^p\cdot\psi.
\end{align*}
The vectors $ D^a=( D_i^a)_{1\leq i\leq 4}$ and $D^p=( D_i^p)_{1\leq i\leq 4}$ are given by:
\begin{align*}
	 D^a_1=&-\partial_xa^p\partial_y(u^a-a^p h^a)-(\partial_ya^p+2a^pb^p)\partial_y(v^a-a^pg^a)+[(a^p)^2-1][(h^a\partial_x+g^a\partial_y)b^p-b^p(\partial_y g^a-b^p g^a)]\\
	 &+[\partial_t+(u^a+a^ph^a)\partial_x+(v^a+a^p g^a)\partial_y-\mu\epsilon\triangle]\partial_ya^p+b^p[\partial_t+(u^a+a^ph^a)\partial_x+(v^a+a^p g^a)\partial_y-\mu\epsilon\triangle]a^p\\
	 &-2\mu\epsilon\big(b^p\partial_y^2a^p+\partial_xa^p\partial_xb^p \big)-(\mu-\kappa)\epsilon a^p\big[\triangle b^p+3b^p\partial_yb^p+(b^p)^3\big]-(3\mu-\kappa)\epsilon\partial_ya^p\big[\partial_yb^p+(b^p)^2\big], \\
	 D_2^a=&\partial_xa^p\partial_x(u^a-a^p h^a)+(\partial_ya^p+2a^pb^p)\partial_x(v^a-a^p g^a)+[(a^p)^2-1]b^p\partial_x g^a\\
	 &-[\partial_t+(u^a+a^ph^a)\partial_x+(v^a+a^p g^a)\partial_y-\mu\epsilon\triangle]\partial_xa^p+\epsilon\big\{2\mu b^p\partial_{xy}^2a^p+(\mu-\kappa)\partial_xa^p[\partial_yb^p+(b^p)^2]\big\},\\
	 D_3^a=&b^p\partial_y(v^a-a^p g^a)-\partial_y\big[(h^a\partial_x+g^a\partial_y)a^p\big]+\big[\partial_t+\big(u^a-a^ph^a\big) \partial_x+\big(v^a-a^pg^a-2\kappa\epsilon b^p\big)\partial_y-\kappa\epsilon\triangle\big]b^p,\\
	 D_4^a=&-b^p\partial_x(v^a-a^p g^a)+\partial_x\big[(h^a\partial_x+g^a\partial_y)a^p\big],
\end{align*}
and
\begin{align*}
	&D_1^p=(\partial_ya^p+2a^pb^p)\partial_{xy}^2a^p-\partial_xa^p\partial_y^2a^p+[(a^p)^2-1]b^p\partial_xb^p+2a^p\partial_xa^p(b^p)^2 , \\
	&D_2^p=-(\partial_ya^p+2a^pb^p)\partial_{x}^2a^p+\partial_xa^p\partial_{xy}^2a^p,\\
	& D_3^p=
	-b^p\partial_{xy}^2a^p-\partial_xa^p[\partial_yb^p+(b^p)^2]+\partial_ya^p\partial_xb^p,\\
	&D_4^p=b^p\partial_{x}^2a^p.
\end{align*}

Next, we give the proof of Proposition \ref{prop_coe-est}.

\begin{proof}[Proof of Proposition \ref{prop_coe-est}]
Firstly, we establish the $L^\infty$ estimates in \eqref{est-coe}, in other words, we will show the following uniform estimate in $\epsilon$:  
\begin{align}\label{est-coe1}
\big\|\partial_{tx}^\alpha (A_i^a, A_i^p, B, C^a)(t,\cdot)\big\|_{L^\infty}+\big\|y\partial_{tx}^\alpha D^a(t,\cdot)\big\|_{L^\infty}+\big\|y^2\partial_{tx}^\alpha  D^p(t,\cdot)\big\|_{L^\infty}~=~O(1).
\end{align}
Actually, the above estimate \eqref{est-coe1} is based on the estimates \eqref{est_eta}, \eqref{est_eta1} and the following facts (\textbf{F}):
\begin{enumerate}[(1)]
\item the definition \eqref{NAP} implies $(\mathbf{u^a}, \mathbf{H^a}),~
(\partial_yv^a, \partial_y g^a)~=~O(1);$
	\item by the boundary conditions $(v^a,g^a)|_{y=0}=0,$ the estimates \eqref{est_eta} and the Hardy inequality, it holds
 \begin{align*}
 	\|v^a\partial_ya^p\|_{L^\infty}\leq\|y\partial_ya^p\|_{L^\infty}\big\| \frac{v^a}{y}\big\|_{L^\infty}\lesssim \|y\partial_ya^p\|_{L^\infty}\|\partial_y v^a\|_{L^\infty}=O(1),
 \end{align*}
 and similarly,
 \begin{align*}
 &\|v^ab^p\|_{L^\infty},~\|g^a\partial_ya^p\|_{L^\infty},~\|g^ab^p\|_{L^\infty},~\big\|\big(u^0-u^0|_{y=0}\big)\partial_y a^p\big\|_{L^\infty},~\big\|\big(u^0-u^0|_{y=0}\big)b^p\big\|_{L^\infty},\\
 	&\big\|\big(h^0-h^0|_{y=0}\big)\partial_y a^p\big\|_{L^\infty},~\big\|\big(h^0-h^0|_{y=0}\big)b^p\big\|_{L^\infty}=O(1),
 \end{align*}
 which implies that $\partial_y(v^a-a^p g^a)~=~O(1);$
 \item it follows that from the definition \eqref{eta},
\begin{align*}
	\partial_y(u^a-a^ph^a)=&\partial_y\Big(u^0-u^0|_{y=0}-a^p\big(h^0-h^0|_{y=0}\big)\Big)+\partial_y\big(u^p-a^ph^p\big)+O(1)\\
	=&O(1)+\partial_y\big((1-\chi)u^p\big)=O(1)-\chi' u^p+\frac{1-\chi}{y}\cdot(y\partial_y u^p)=O(1),
\end{align*}
by using $\big|\frac{1-\chi}{y}\big|\leq1$, and similarly,
\begin{align*}
	\partial_y h^a-b^ph^a=&\partial_y\big(h^0-h^0|_{y=0}\big)-b^p\big(h^0-h^0|_{y=0}\big)+\partial_yu^p-b^ph^p+O(1)\\
	=&\partial_y h^0-b^p\big(h^0-h^0|_{y=0}\big)+O(1)
	=O(1);
	\end{align*}
	\item the above three properties also hold for the  derivatives, up to oder three with respect to $t$ and $x$ , of corresponding quantities.   
\end{enumerate}

Secondly, for the source term $E^\epsilon$ given in \eqref{source}, it follows that by virtue of $r_5^\epsilon=\partial_y^{-1}r_3^\epsilon$ given in \eqref{eq_psi},
\begin{align*}
	E^\epsilon~=~\Big(r_1^\epsilon-a^p\cdot r_3^\epsilon-\partial_ya^p\cdot\partial_y^{-1} r_3^\epsilon,r_2^\epsilon+\partial_x(a^p\cdot \partial_y^{-1}r_3^\epsilon),r_3^\epsilon-b^p \cdot\partial_y^{-1}r_3^\epsilon,r_4^\epsilon\Big)^T,
\end{align*}
which implies that by combining \eqref{est_err}, \eqref{est_eta}, \eqref{est_eta1} and the Hardy inequality,
\begin{align}\label{E_est}
	\|\partial_{tx}^\alpha E^\epsilon(t,\cdot)\|_{L^2}~\lesssim~\sum_{1\leq i\leq 4,~|\beta|\leq3}\|\partial_{tx}^\beta r_i^\epsilon(t,\cdot)\|_{L^2}~\leq~ C,\quad |\alpha|\leq 2
\end{align}
for some constant $C>0$ independent of $\epsilon$. Thus, we obtain \eqref{est-coe} by combining \eqref{est-coe1} with \eqref{E_est}.

Next, note that from the definitions of $\tilde A_i(U), i=1,2$ and $\tilde C(U)$, combining with the relations \eqref{def_psi}, \eqref{est_eta},  \eqref{est_eta1} and \eqref{est_psi}, a direct consequence of Lemma \ref{lem_equ} is that for $|\alpha|\leq2$ and $1<p\leq\infty$,
\begin{align*}
	\big\|\partial_{tx}^\alpha\tilde{A}_i(U)(t,\cdot)\big\|_{L^p}\lesssim\sum_{\beta\leq\alpha}\big\|\partial_{tx}^\beta(u,v,h,g)\big\|_{L^p}
	\leq C\sum_{\beta\leq\alpha}\big\|\partial_{tx}^\beta U(t,\cdot)\big\|_{L^p},\quad i=1,2,
\end{align*}
and
\begin{align*}
	\big\|\partial_{tx}^\alpha\tilde{C}(U)(t,\cdot)\big\|_{L^p}\lesssim\epsilon^{-\frac{1}{2}}\sum_{\beta\leq\alpha}\Big[\big\|\partial_{tx}^\beta (h,g)(t,\cdot)\big\|_{L^p}+\big\|y^{-1}\partial_{tx}^\beta\psi(t,\cdot)\big\|_{L^p}\Big]\leq C\epsilon^{-\frac{1}{2}}\sum_{\beta\leq\alpha}\big\|\partial_{tx}^\beta U(t,\cdot)\big\|_{L^p}.
	\end{align*}	
	Thus, we obtain \eqref{est_AU}.
	
\end{proof}

\medskip \noindent
{\bf Acknowledgements:}
 The research of the first author was sponsored by National Natural Science Foundation of China (Grant No. 11743009), Shanghai Sailing Program (Grant No. 18YF1411700) and Scientific Research Foundation of Shanghai Jiao Tong University (Grant No. WF220441906).  
The research of the second author
 was  supported by NSFC (Grant No.11571231). The research
of the third author is  supported by the General Research Fund of Hong Kong, CityU No. 11320016.

\bibliographystyle{plain}

\end{document}